\documentclass[11pt]{amsart}

\usepackage{
    amsmath,
    amsfonts,
    amssymb,
    amsthm,
    amscd,
    comment,
    enumitem,
    etoolbox,
    gensymb,    
    mathtools,
    mathdots,
    stmaryrd,
    booktabs,
}
\usepackage[usenames,dvipsnames]{xcolor}
\usepackage[all]{xy}

\usepackage[T1]{fontenc}
\usepackage{bbm}                     
\usepackage[colorlinks=true, linkcolor=blue, citecolor=blue, urlcolor=blue, breaklinks=true]{hyperref}

\DeclareFontFamily{OT1}{pzc}{}
\DeclareFontShape{OT1}{pzc}{m}{it}{<-> s * [1.10] pzcmi7t}{}
\DeclareMathAlphabet{\mathpzc}{OT1}{pzc}{m}{it}

\usepackage[capitalize]{cleveref}   

\crefname{defin}{Definition}{Definitions}
\crefname{eg}{Example}{Examples}
\crefname{egs}{Examples}{Examples}
\crefname{convention}{Convention}{Convention}
\crefname{lem}{Lemma}{Lemmas}
\crefname{prop}{Proposition}{Propositions}
\crefname{theo}{Theorem}{Theorems}
\crefname{rem}{Remark}{Remarks}
\crefname{equation}{}{}
\crefname{enumi}{}{}

\newcommand\la{\lambda}
\newcommand\La{\Lambda}
\newcommand\n{\mathfrak{n}}
\newcommand\sgn{\operatorname{sgn}}

\newcommand\C{\mathbb{C}}

\newcommand\GG{\mathbb{G}}
\newcommand\N{\mathbb{N}}

\newcommand\Q{\mathbb{Q}}

\newcommand\Z{\mathbb{Z}}
\newcommand\kk{\Bbbk}
\newcommand\one{\mathbbm{1}}

\newcommand\bB{\mathbf{B}}

\newcommand\bF{\mathbf{F}}
\newcommand\bAF{\hat{\mathbf{F}}}

\newcommand\fb{\mathfrak{b}}
\newcommand\fg{\mathfrak{g}}
\newcommand\fh{\mathfrak{h}}
\newcommand\h{\mathfrak{h}} 			

\newcommand\fS{\mathfrak{S}}
\newcommand\fso{\mathfrak{so}}

\newcommand\Cl{\mathrm{Cl}}             
\newcommand\GL{\mathrm{GL}}
\newcommand\GPin{\mathrm{GPin}}
\newcommand\Group{{\mathrm{G}}}         
\newcommand\op{\mathrm{op}}
\newcommand\OSp{\mathrm{OSp}}           

\newcommand\rev{\mathrm{rev}}
\newcommand\rO{\mathrm{O}}              

\newcommand\SO{\mathrm{SO}}             
\newcommand\Spin{\mathrm{Spin}}         
\newcommand\Pin{\mathrm{Pin}}           

\newcommand\triv{\mathrm{triv}}

\newcommand\Lie{\mathrm{Lie}}
\newcommand\Irr{\mathrm{Irr}}

\newcommand\md{\textup{-mod}}           
\newcommand\Md{\textup{-Mod}}           
\newcommand\Shuffle{\textup{Sh}}
\newcommand\transpose{\textup{t}}

\newcommand\even{{\bar{0}}}
\newcommand\odd{{\bar{1}}}

\newcommand\dotop{\tilde{\Omega}}

\newcommand\inv{^{-1}}

\newcommand\ASB{\mathpzc{ASB}}          
\newcommand\Brauer{\mathpzc{B}}         
\newcommand\cC{\mathcal{C}}
\newcommand\cD{\mathcal{D}}
\newcommand\cEnd{\mathpzc{End}}
\newcommand\cN{\mathpzc{N}}
\newcommand\OB{\mathpzc{OB}}            

\newcommand\SB{\mathpzc{SB}}            

\newcommand\Sgo{\mathsf{S}}             
\newcommand\Vgo{\mathsf{V}}             

\DeclareMathOperator{\diag}{diag}
\DeclareMathOperator{\End}{End}

\DeclareMathOperator{\flip}{flip}
\DeclareMathOperator{\Hom}{Hom}

\DeclareMathOperator{\Id}{Id}
\DeclareMathOperator{\id}{id}
\DeclareMathOperator{\Ind}{Ind}
\DeclareMathOperator{\Kar}{Kar}
\DeclareMathOperator{\Mat}{Mat}

\DeclareMathOperator{\Res}{Res}
\DeclareMathOperator{\Span}{span}

\DeclareMathOperator{\tr}{tr}
\DeclareMathOperator{\Tr}{Tr}
\DeclareMathOperator{\wt}{wt}

\usepackage{tikz}
\usetikzlibrary{arrows.meta}
\usetikzlibrary{decorations.markings,decorations.pathreplacing}
\usetikzlibrary{calc}
\usetikzlibrary{snakes}
\usetikzlibrary{patterns}   
\usepackage{tikz-cd}

\tikzset{anchorbase/.style={>=To,baseline={([yshift=-0.5ex]current bounding box.center)}}}
\tikzset{ 
    centerzero/.style={>=To,baseline={([yshift=-0.5ex](#1))}},
    centerzero/.default={0,0}
}
\tikzset{wipe/.style={white,line width=3pt}}
\tikzset{vec/.style={thick,blue,densely dotted}}
\tikzset{spin/.style={thick,black}}
\tikzset{multispin/.style={very thick,black}}
\tikzset{any/.style={red,densely dashed}}

\tikzset{->-/.style={decoration={
  markings,
  mark=at position #1 with {\arrow{>}}},postaction={decorate}}
}
\tikzset{-<-/.style={decoration={
  markings,
  mark=at position #1 with {\arrow{<}}},postaction={decorate}}
}

\newcommand\braidup{to[out=up,in=down]}
\newcommand\braiddown{to[out=down,in=up]}

\newcommand\dotlabel[1]{$\scriptstyle{#1}$}
\newcommand\strandlabel[1]{$\scriptstyle{#1}$}
\newcommand\botlabel[1]{node[anchor=north] {\strandlabel{#1}}}

\newcommand\singdot[1]{
    \filldraw[fill=white, draw=black] (#1) circle (1.5pt)
}
\newcommand\multdot[3]{
    \filldraw[fill=white, draw=black] (#1) circle (1.5pt) node[anchor=#2] {\dotlabel{#3}}
}

\newcommand{\genbox}[3]{
    \filldraw[fill=black!20!white,rounded corners] (#1) rectangle (#2) node[midway] {\strandlabel{#3}}
}
\newcommand{\altbox}[3]{
    \filldraw[fill=black!20!white] (#1) rectangle (#2) node[midway] {\strandlabel{#3}}
}

\newcommand\bub[2]{
    \draw[#1] (#2)++(0,0.2) arc(90:-270:0.2)
}

\newcommand\cbubgen[2][u]{
    \draw[->] (#2)++(0,0.2) arc(90:-270:0.2);
    \node at (#2) {\dotlabel{#1}}
}

\newcommand\ccbubgen[2][u]{
    \draw[->] (#2)++(0,0.2) arc(90:450:0.2);
    \node at (#2) {\dotlabel{#1}}
}

\newcommand\bubble[1]{
    \begin{tikzpicture}[centerzero]
        \bub{#1}{0,0};
    \end{tikzpicture}
}
\newcommand\multbubble[2]{
    \begin{tikzpicture}[centerzero]
        \bub{#1}{0,0};
        \multdot{0.2,0}{west}{#2};
    \end{tikzpicture}
}

\newcommand\cbubblegen[1][u]{
    \begin{tikzpicture}[centerzero]
        \cbubgen{0,0};
    \end{tikzpicture}
}

\newcommand\ccbubblegen[1][u]{
    \begin{tikzpicture}[centerzero]
        \ccbubgen{0,0};
    \end{tikzpicture}
}
\newcommand\idstrand[1]{
    \begin{tikzpicture}[centerzero]
        \draw[#1] (0,-0.2) -- (0,0.2);
    \end{tikzpicture}
}
\newcommand\dotstrand[1]{
    \begin{tikzpicture}[centerzero]
        \draw[#1] (0,-0.2) -- (0,0.2);
        \singdot{0,0};
    \end{tikzpicture}
}

\newcommand\crossmor[2]{
    \begin{tikzpicture}[centerzero]
        \draw[#1] (-0.2,-0.2) -- (0.2,0.2);
        \draw[#2] (0.2,-0.2) -- (-0.2,0.2);
    \end{tikzpicture}
}
\newcommand{\cupmor}[1]{
    \begin{tikzpicture}[centerzero]
        \draw[#1] (-0.15,0.15) -- (-0.15,0) arc(180:360:0.15) -- (0.15,0.15);
    \end{tikzpicture}
}
\newcommand{\capmor}[1]{
    \begin{tikzpicture}[centerzero]
        \draw[#1] (-0.15,-0.15) -- (-0.15,0) arc(180:0:0.15) -- (0.15,-0.15);
    \end{tikzpicture}
}
\newcommand{\mergemor}[3]{
    \begin{tikzpicture}[anchorbase]
        \draw[#1] (-0.197,-0.197) -- (0,0);
        \draw[#2] (0.197,-0.197) -- (0,0);
        \draw[#3] (0,0) -- (0,0.229);
    \end{tikzpicture}
}
\newcommand{\splitmor}[3]{
    \begin{tikzpicture}[anchorbase]
        \draw[#2] (-0.197,0.197) -- (0,0);
        \draw[#3] (0.197,0.197) -- (0,0);
        \draw[#1] (0,0) -- (0,-0.229);
    \end{tikzpicture}
}

\newcommand\hourglass{
    \begin{tikzpicture}[centerzero]
        \draw (-0.1,-0.2) -- (-0.1,-0.15) arc(180:0:0.1) -- (0.1,-0.2);
        \draw (-0.1,0.2) -- (-0.1,0.15) arc(180:360:0.1) -- (0.1,0.2);
    \end{tikzpicture}
}

\newtheorem{theo}{Theorem}[section]

\newtheorem{prop}[theo]{Proposition}
\newtheorem{lem}[theo]{Lemma}
\newtheorem{cor}[theo]{Corollary}
\newtheorem{conj}[theo]{Conjecture}

\theoremstyle{definition}
\newtheorem{defin}[theo]{Definition}
\newtheorem{rem}[theo]{Remark}

\numberwithin{equation}{section}
\allowdisplaybreaks

\setenumerate[1]{label=(\alph*)}          

\newtoggle{comments}
\newtoggle{details}
\newtoggle{detailsnote}

\iftoggle{comments}{%
    \usepackage[notref,notcite]{showkeys}   
    \newcommand{\acomments}[1]{
        \ \\
        {\color{red}
            \textbf{AS:} #1
        }
        \ \\
    }
    \newcommand{\pcomments}[1]{
        \ \\
        {\color{blue}
            \textbf{PM:} #1
        }
        \ \\
    }
}{%
    \newcommand{\acomments}[1]{\ignorespaces}
    \newcommand{\pcomments}[1]{\ignorespaces}
}

\iftoggle{details}{%
    \newcommand{\details}[1]{
        \ \\
        {\color{OliveGreen}
            \textbf{Details:} #1
        }
        \\
    }
}{%
    \newcommand{\details}[1]{\ignorespaces}
}

\begin{document}

\title{The spin Brauer category}

\author{Peter J. McNamara}
\address[P.M.]{
    School of Mathematics and Statistics \\
    University of Melbourne \\
    Parkville, VIC, 3010, Australia
}
\urladdr{\href{http://petermc.net/maths}{petermc.net/maths}, \textrm{\textit{ORCiD}:} \href{https://orcid.org/0000-0001-6111-1511}{orcid.org/0000-0001-6111-1511}}
\email{maths@petermc.net}

\author{Alistair Savage}
\address[A.S.]{
  Department of Mathematics and Statistics \\
  University of Ottawa \\
  Ottawa, ON, K1N 6N5, Canada
}
\urladdr{\href{https://alistairsavage.ca}{alistairsavage.ca}, \textrm{\textit{ORCiD}:} \href{https://orcid.org/0000-0002-2859-0239}{orcid.org/0000-0002-2859-0239}}
\email{alistair.savage@uottawa.ca}

\begin{abstract}
    We introduce a diagrammatic monoidal category, the \emph{spin Brauer category}, that plays the same role for the spin and pin groups as the Brauer category does for the orthogonal groups.  In particular, there is a full functor from the spin Brauer category to the category of finite-dimensional modules for the spin and pin groups.  This functor becomes essentially surjective after passing to the Karoubi envelope, and its kernel is the tensor ideal of negligible morphisms.  In this way, the spin Brauer category can be thought of as an interpolating category for the spin and pin groups.  We also define an affine version of the spin Brauer category, which acts on categories of modules for the pin and spin groups via translation functors.
\end{abstract}



\ifboolexpr{togl{comments} or togl{details}}{%
  {\color{magenta}DETAILS OR COMMENTS ON}
}{%
}

\maketitle
\thispagestyle{empty}


\section{Introduction}

One of the most classical results in representation theory is Schur--Weyl duality, one half of which is the statement that the algebra homomorphism
\[
    \C \fS_r \to \End_{\GL(V)}(V^{\otimes r})
\]
is surjective, where $\fS_r$ is the symmetric group on $r$ letters, acting on $V^{\otimes r}$ by permutation of the factors.  If one replaces the general linear group by the orthogonal group, the analogous statement is that one has a surjective algebra homomorphism
\[
    \textup{Br}_r \to \End_{\rO(V)}(V^{\otimes r}),
\]
where $\textup{Br}_r$ is the Brauer algebra.

A more modern approach to the above involves considering morphisms between \emph{different} powers of the natural module $V$ to rephrase the results in terms of monoidal categories.  More precisely, there is a full and essentially surjective functor
\[
    \OB(N) \to \GL(V)\md
\]
from the \emph{oriented Brauer category} to the category of finite-dimensional rational $\GL(V)$-modules, where $N = \dim V$.  See, for example, \cite[Th.~4.7.1]{CW12}, although the idea essentially goes back to Turaev \cite{Tur89}.  Similarly, one has a full and essentially surjective functor
\[
    \Brauer(N) \to \rO(V)\md,
\]
where $\Brauer(N)$ is the \emph{Brauer category}.  See, for example, \cite[Th.~4.8]{LZ15}.  The categories $\OB(N)$ and $\Brauer(N)$ are defined for \emph{any} choice of parameter $N \in \C$.  Their additive Karoubi envelopes are Deligne's interpolating categories for the general linear and orthogonal groups \cite{Del07}.

Since the orthogonal group is not simply connected, it is natural to want to extend the above picture to its double cover, the pin group $\Pin(V)$, or the identity component, the spin group $\Spin(V)$.  This desire is further underlined by the importance of the spin group in other areas of representation theory and physics.  A first step in this direction is the recent work of Wenzl \cite{Wen20} describing the endomorphism algebra of $S^{\otimes r}$, where $S$ is the spin module.  (In fact, \cite{Wen20} works with representations of quantized enveloping algebra.)  Other partial results were obtained in \cite{OW02,Wen12}.  In type $D$, see also \cite[Th.~4.3.4.1]{How95} and \cite[Th.~3.6]{Abo22} for similar results.  The goal of the current paper is to develop the monoidal category approach and find a spin analogue of the Brauer category, allowing one to describe morphisms between \emph{all} tensor products of the spin and vector modules.

After recalling and developing, in \cref{sec:spin,sec:Liealg,sec:reps}, some of the representation theory of the spin and pin groups, we introduce the \emph{spin Brauer category} $\SB(d,D;\kappa)$ in \cref{sec:spinBrauer}.  Here $d,D$ are elements of the ground field and $\kappa \in \{\pm 1\}$.  Our definition of this strict monoidal category is diagrammatic, given via a presentation in terms of generators and relations.  Whereas the Brauer category has one generating object, which should be thought of as a formal version of the natural module, the spin Brauer category has an additional generating object corresponding to the spin module.  The parameters $d$ and $D$ are the categorical dimensions of the two generating objects. We then describe, in \cref{incarnation}, a functor
\[
    \bF \colon \SB(N, \sigma_N 2^{n}; \kappa_N) \to \Group(V)\md,
\]
where $N = \dim V$, $n=\lfloor \frac{N}{2} \rfloor$, $\sigma_N, \kappa_N \in \{\pm 1\}$ depend on $N$ (see \cref{crazy}), and
\[
    \Group(V) :=
    \begin{cases}
        \Pin(V) & \text{if $N$ is even}, \\
        \Spin(V) & \text{if $N$ is odd}.
    \end{cases}
\]
(See \cref{hamster,gerbil} for an explanation of why we split into these cases.)

We prove that the functor $\bF$ is full (\cref{fullness}), essentially surjective after passing to the Karoubi envelope (\cref{essential}), and that its kernel is precisely the tensor ideal of negligible morphisms (\cref{negligible}).  This implies that the category $\Group(V)\md$ is equivalent to the semisimplification of the Karoubi envelope of $\SB(N, \sigma_N 2^n; \kappa_N)$.  The Karoubi envelope of $\SB(d,D;\kappa)$ should be thought of as an interpolating category for the spin and pin groups, in the spirit of Deligne's interpolating categories \cite{Del07}.

Both the oriented Brauer category and the Brauer category have affine analogues, defined in \cite{BCNR17,RS19}.  In \cref{sec:affine}, we define an affine version $\ASB(d,D;\kappa)$ of the spin Brauer category, together with functors (\cref{affinc})
\[
    \ASB(N, \sigma_N 2^n; \kappa_N) \to \cEnd_\C(\Group(V)\md)
    \quad \text{and} \quad
    \ASB(N, \sigma_N 2^n; \kappa_N) \to \cEnd_\C(\fso(V)\Md),
\]
where $\cEnd_\C(\cC)$ denotes the monoidal category of $\C$-linear endofunctors of a $\C$-linear category $\cC$, with natural transformations as morphisms, and $\fso(V)\Md$ denotes the category of all $\fso(V)$-modules.  This functor yields tools for studying the translation functors given by tensoring with the spin and vector modules.  Such translation functors have proved to be extremely useful in representation theory.  Finally, in \cref{centresurjection}, we show that the image of the induced algebra homomorphism
\[
    \End_{\ASB(N, \sigma_N 2^n; \kappa_N)}(\one) \to \cEnd_\C(\fso(V)\Md) \cong Z(\fso(V))
\]
is $Z(\fso(V))^{\Group(V)}$, where $Z(\fso(V))$ is the centre of the universal enveloping algebra $U(\fso(V))$.

The results of the current paper bring the power of diagrammatic techniques to the study of the representation theory of the spin and pin groups.  As such, they lead to many natural directions of future research.  We plan to develop a quantum analogue of our results, replacing the spin group by the quantized enveloping algebra $U_q(\fso(n))$. Such a quantum version would also suggest an approach to webs of types $B$ and $D$, and so should be related to recent work of Bodish and Wu \cite{BW23}.

\subsection*{Acknowledgements}

The research of P.M.\ was supported by Australian Research Council grant DE150101415.  The research of A.S.\ was supported by Discovery Grant RGPIN-2023-03842 from the Natural Sciences and Engineering Research Council of Canada.  The second author is also grateful for the support and hospitality of the Sydney Mathematical Research Institute (SMRI).  The authors thank Elijah Bodish, Ben Webster, and Geordie Williamson for helpful discussions.  Several ideas in this paper were also influenced by \cite{Del}.

\subsection*{Relation to published version}

After publication of this paper, the authors noticed a gap in the proof of \cref{essential}.  A footnote has been added to the proof, indicating how to fix this gap.

\section{The spin representation\label{sec:spin}}

In this section, we recall the explicit construction of the most important representation theoretic ingredient in the current paper: the spin representation.  Throughout this section we work over the field $\C$ of complex numbers.

\subsection{The Clifford algebra}

Let $V$ be a vector space of finite dimension $N$ and let $\Phi_V \colon V \times V \to \C$ be a nondegenerate symmetric bilinear form.  Let
\begin{equation} \label{Clifford}
    \Cl = \Cl(V) := T(V)/ \left( vw+wv - 2 \Phi_V(v, w) : v,w \in V \right)
\end{equation}
denote the Clifford algebra associated to $V$.  Here $T(V)$ is the tensor algebra on $V$.  The algebra $\Cl$ is $(\Z/2\Z)$-graded by declaring that elements of $V$ are odd (that is, have degree $\bar 1$).  We let $\deg x \in \Z/2\Z$ denote the degree of a homogeneous element $x \in \Cl$.

The factor of $2$ in \cref{Clifford} is chosen to make some later formulas slightly cleaner.  For instance, for $v \in V$ with $\Phi_V(v,v) = 1$, we have $v^2 = 1$.  Note, however, that not all elements of $V$ are invertible when $N \ge 2$.  For example, if $v \in V$ satisfies $\Phi_V(v,v) = 0$, then $v$ is not invertible.

Since, over the complex numbers, any nondegenerate symmetric form is equivalent to the standard one, we may fix an orthonormal basis $e_1,\dotsc,e_N$ of $V$.  Then, in $\Cl$, we have
\begin{equation} \label{ecomm}
    e_i e_j + e_j e_i = 2 \delta_{ij}.
\end{equation}
Let
\begin{equation}
    n = \left\lfloor \frac{N}{2} \right\rfloor \in \N,
    \quad \text{so that} \quad
    N =
    \begin{cases}
        2n & \text{if $N$ is even}, \\
        2n+1 & \text{if $N$ is odd}.
    \end{cases}
\end{equation}
Now define
\begin{equation} \label{tuna}
    \psi_j := \tfrac{1}{2} \left( e_{2j-1} + \sqrt{-1} e_{2j} \right),\qquad
    \psi_j^\dagger : = \tfrac{1}{2} \left( e_{2j-1} - \sqrt{-1} e_{2j} \right), \qquad
    1 \le j \le n.
\end{equation}
Then we have
\[
    \Phi_V(\psi_i, \psi_j) = 0, \quad
    \Phi_V(\psi_i^\dagger, \psi_j^\dagger) = 0, \quad
    \Phi_V(\psi_i, \psi_j^\dagger) = \tfrac{1}{2} \delta_{ij},\qquad
    1 \le i,j \le n,
\]
and so
\begin{equation} \label{fermion}
    \psi_i \psi_j + \psi_j \psi_i = 0 = \psi_i^\dagger \psi_j^\dagger + \psi_j^\dagger \psi_i^\dagger, \qquad
    \psi_i \psi_j^\dagger + \psi_j^\dagger \psi_i = \delta_{ij},\qquad
    1 \le i,j \le n.
\end{equation}

When $N$ is even, \cref{fermion} gives a presentation of $\Cl$. When $N$ is odd, we need to include the additional relations
\begin{equation}\label{oddextra}
    \psi_i e_N + e_N \psi_i=0=\psi_i^\dagger e_N + e_N \psi_i^\dagger, \qquad e_{N}^2=1, \qquad 1\le i \le n,
\end{equation}
to obtain a presentation of $\Cl$.

\subsection{Clifford modules\label{subsec:Cliffmod}}

The algebra $\Cl$ is semisimple.  If $N$ is even, then the algebra $\Cl$ has a unique simple module up to isomorphism. If $N$ is odd, then $\Cl$ has exactly two simple modules. We will now describe these.

Let
\[
    S := \Lambda(W) = \bigoplus_{r=0}^n \Lambda^r(W),\qquad
    \text{where}\quad W = \Span_\C \{ \psi_i^\dagger : 1 \le i \le n \}.
\]
As a $\C$-module, $S$ has basis
\begin{equation} \label{xIdef}
    \begin{gathered}
        x_I := \psi_{i_1}^\dagger \wedge \psi_{i_2}^\dagger \wedge \dotsb \wedge \psi_{i_k}^\dagger,\\
        I = \{i_1,\dotsc,i_k\} \subseteq [n],\quad i_1 < i_2 < \dotsc < i_k,\quad
        0 \le k \le n,
    \end{gathered}
\end{equation}
where
\[
    [n]=\{1,2,\dotsc,n\},
\]
a notation we use throughout.  In particular,
\begin{equation} \label{paneer}
    \dim_\C(S) = 2^n.
\end{equation}

We will now construct a $\Cl$-module structure on $S$. If $N$ is even, we turn $S$ into a $\Cl$-module by defining
\begin{equation} \label{mind}
    \psi_i^\dagger x_I = \psi_i^\dagger \wedge x_I,\qquad
    \psi_i x_I =
    \begin{cases}
        (-1)^{|\{ j \in I : j<i \}|} x_{I \setminus \{i\}} & \text{if } i \in I, \\
        0 & \text{if } i \notin I,
    \end{cases}
    \qquad
    I \subseteq [n],\ 1 \le i \le n.
\end{equation}
It is straightforward to verify that the relations \cref{fermion} are satisfied.

If $N$ is odd, then we define two $\Cl$-module structures on $S$, depending on a choice of $\varepsilon\in \{\pm 1\}$. We again use the action defined in \cref{mind}, and additionally define
\begin{equation} \label{mouse}
    e_{2n+1} x_I = \varepsilon (-1)^{|I|} x_I.
\end{equation}
It is straightforward to verify that relations \cref{oddextra} are satisfied.

For both even and odd $N$, the $\Cl$-modules defined above are called the \emph{spin modules}.  If $N$ is even, then $S$ is the unique simple $\Cl$-module.  If $N$ is odd, then the two modules constructed above are the two nonisomorphic simple $\Cl$-modules.
In both cases, $\Cl$ is semisimple.
\details{
    To prove that $\Cl$ is semisimple, we use the fact that, if $A$ is a finite-dimensional algebra over an algebraically closed field, then
    \[
        A \text{ is semisimple}
        \iff \dim A \le \sum_i n_i^2
        \iff \dim A = \sum_i n_i^2,
    \]
    where the $n_i$ are the dimensions of the simple $A$-modules.  Since is it clear that $\dim \Cl \le 2^N$, the result then follows from \cref{paneer} and the simplicity of spin modules.
}

\begin{rem} \label{snow}
    An equivalent construction of the spin module is as $\Cl/A$, where $A$ is the left ideal generated by the $\psi_j$, $1 \le j \le n$, if $N$ is even, and is the left ideal generated by the $\psi_j$, $1 \le j \le n$, and $e_N-\varepsilon$ if $N$ is odd.
\end{rem}

We conclude this subsection with a technical lemma that will be used in the proof of \cref{incarnation}. Suppose $N$ is odd, and let
\begin{equation} \label{spiky}
    \psi_0 = \tfrac{1}{\sqrt{2}} e_N,\qquad
    \psi_{-i}=\psi_i^\dagger, \quad i \in [n].
\end{equation}

\begin{lem}\label{signlemma}
    Suppose $N$ is odd.  For all permutations $\varpi$ of $\{-n,1-n,\cdots ,n-1,n\}$, we have
    \begin{equation} \label{dragon}
        \psi_{\varpi(-n)}\psi_{\varpi(1-n)}   \dotsm \psi_{\varpi(n-1)}\psi_{\varpi(n)} x_I
        =
        \begin{cases}
            \frac{\varepsilon}{\sqrt{2}} \sgn(\varpi) x_I & \text{if } I=I_\varpi, \\
            0 & \text{otherwise},
        \end{cases}
    \end{equation}
    where $I_\varpi=\{i \in [n] : \varpi\inv(-i) < \varpi\inv(i)\}$.
\end{lem}

\begin{proof}
    For $i=1,2,\ldots,n$, let
    \[
        A_i =
        \begin{cases}
            \psi_{-i}\psi_i & \text{if } i \in I_\varpi, \\
            \psi_i \psi_{-i} & \text{if } i \notin I_\varpi.
        \end{cases}
    \]
    We compare the products $\psi_{\varpi(-n)}\psi_{\varpi(1-n)} \dotsm \psi_{\varpi(n-1)}\psi_{\varpi(n)}$ and $\psi_0 A_1 A_2 \dotsm A_n$. They are both a product of $\psi_{-n},\psi_{1-n},\ldots,\psi_n$ in some order.  For each $i\in [n]$, the elements $\psi_i$ and $\psi_{-i}$ appear in the same order in each of these two products. Therefore, we can pass from one to the other by swapping adjacent pairs $\psi_i$ and $\psi_j$ for $i\neq \pm j$. In the Clifford algebra, each of these swaps introduces a minus sign since $\psi_i\psi_j=-\psi_j\psi_i$ for $i\neq \pm j$. So, in order to compare $\psi_{\varpi(-n)} \psi_{\varpi(1-n)} \dotsm \psi_{\varpi(n-1)}\psi_{\varpi(n)}$ and $\psi_0 A_1 A_2 \dotsm A_n$ in $\Cl$, we need to compute the sign of the permutation by which these two orderings of the indices differ.

    The sign of the permutation sending $(\varpi(-n),\varpi(1-n),\ldots,\varpi(n))$ to $(-n,1-n,\ldots,n)$ is $\sgn(\varpi)$.  A reduced expression of the permutation sending $(-n,1-n,\cdots,n)$ to $(0,-1,1,-2,2,\dotsc,-n,n)$ is a product of $1 + 3 + \dotsb + (2n-1) = n^2$ simple transpositions.  Therefore, its sign is $(-1)^{n^2} =  (-1)^n$. The sign of the permutation sending $(0,-1,1,-2,2,\dotsc,-n,n)$ to the order of the indices in the product $\psi_0 A_1 A_2 \dotsm A_n$ is $(-1)^{n-|I_\varpi|}$. Hence, we obtain the identity
    \[
        \psi_{\varpi(-n)}\psi_{\varpi(1-n)}   \dotsm \psi_{\varpi(n-1)}\psi_{\varpi(n)}
        = \sgn(\varpi)(-1)^{|I_\varpi|} \psi_0 A_1 A_2 \dotsm A_n.
    \]
    From \cref{mind} and \cref{mouse} we have $A_i x_{I_\varpi}=x_{I_\varpi}$ and $\psi_0 x_{I_\varpi}=\varepsilon (-1)^{|I_\varpi|}/\sqrt{2}$. This completes the proof of \cref{dragon} when $I=I_\varpi$.  On the other hand, when $I \ne I_\varpi$, we have $A_i x_I = 0$ for any $i \in (I \setminus I_\varpi) \cup (I_\varpi \setminus I)$.
\end{proof}

\subsection{The pin and spin groups\label{subsec:pin}}

Recall that $\Cl$ is $(\Z/2\Z)$-graded.  Define
\begin{equation}
    \GPin(V) := \{ g \in \Cl(V)^\times : g \text{ is homogeneous and } gVg^{-1} = V\},
\end{equation}
and let
\begin{equation}
    \iota \colon \Cl(V) \to \Cl(V)
\end{equation}
be the unique antiautomorphism of $\Cl(V)$ that is the identity on elements of $V$.  Then the \emph{spinor norm} on $\GPin(V)$ is the group homomorphism given by
\[
    \GPin(V) \to \C^\times,\quad
    g \mapsto g \iota(g).
\]
\details{
    We justify that the spinor norm takes values in $\C^\times$.  If $v \in V$ then $v \iota(v)\in \C$.  Next, note that $(xy)\iota(xy)=x(y\iota(y))\iota(x)$. So, if $y\iota(y)\in\C$ and $x\iota(x)\in \C$, then $(xy)\iota(xy)\in\C$.  Since $V$ generates $\Cl(V)$, we conclude that $x\iota(x)\in\C$ for all $x \in \Cl(V)$.  It also follows that the spinor norm is a group homomorphism, and hence takes values in $\C^\times$.
}
The \emph{pin group} associated to $V$, equipped with its nondegenerate symmetric bilinear form, is the subgroup of $\GPin(V)$ consisting of elements of spinor norm one:
\begin{equation}
    \Pin(V) := \{g \in \GPin(V) : g \iota(g) = 1\}.
\end{equation}
The corresponding \emph{spin group} is
\begin{equation}
    \Spin(V) := \Pin(V) \cap \Cl_\even,
\end{equation}
where $\Cl_\even$ is the even part of $\Cl$.

We will need the following analogue of the Cartan--Dieudonn\'e Theorem for the pin and spin groups. Note that if $v\in V$ satisfies $\Phi_V(v,v)=1$, then $v\in \Pin(V)$.

\begin{theo} \label{icts}
    Suppose $N\geq 1$. Then
    \begin{gather} \label{pinvgen}
        \Pin(V) = \{ v_1 v_2 \dotsm v_k : k \in \N,\ v_i \in V,\ \Phi_V(v_i,v_i)=1\ \forall\ 1 \le i \le k \}
        \subseteq \Cl(V)^\times,
        \\ \label{spinvgen}
        \Spin(V) = \{ v_1 v_2 \dotsm v_k : k \in 2\N,\ v_i \in V,\ \Phi_V(v_i,v_i)=1\ \forall\ 1 \le i \le k \}
        \subseteq \Pin(V).
    \end{gather}
\end{theo}

\begin{proof}
    First we prove \cref{pinvgen}.  We will deduce this from the usual Cartan--Dieudonn\'e Theorem, which states that the orthogonal group $\rO(V)$ is generated by reflections (linear transformations that act as $-1$ on a vector of nonzero length and fix its orthogonal complement).

    There is a homomorphism $p \colon \Pin(V)\to \rO(V)$ given by the following action of $\Pin(V)$ on $V$:
    \begin{equation} \label{cha}
        p(g)(v)=(-1)^{\deg g} g v g\inv,
    \end{equation}
    for $g\in \Pin(V)$ and $v\in V$.  Indeed, conjugate $uv+vu=2\Phi_V(u,v)$ by $g$ to get $\Phi_V(p(g)(u),p(g)(v)) = \Phi_V(u,v)$. This shows that the image of $p$ lies in $\rO(V)$.

    Since $v^2 = \Phi_V(v,v)$, we have
    \[
        v^{-1} = \Phi_V(v,v)^{-1} v
        \qquad \text{for } v \in V,\quad \Phi_V(v,v) \ne 0.
    \]
    Then, for $w \in V$, we have
    \[
        -vwv^{-1}
        \overset{\cref{Clifford}}{=} wvv^{-1} - 2 \Phi_V(v,w) v^{-1}
        = w - 2\Phi_V(v,w) v^{-1}
        = w - 2 \frac{\Phi_V(v,w)}{\Phi_V(v,v)} v.
    \]
    Thus, if $\Phi_V(v,v) \ne 0$, then $p(v)$ is reflection across the hyperplane orthogonal to $v$.

    Let $g\in \Pin (V)$. By the Cartan--Dieudonn\'e Theorem, there exist $v_1,v_2,\dotsc,v_k\in V$ with $\Phi_V(v_i,v_i)=1$, such that $p(g)=p(v_1 v_2 \dotsm v_k)$. Therefore $g\inv v_1v_2\cdots v_k\in \ker p$.  If $x\in \ker p$ then $xvx\inv=(-1)^{\deg x} v$ for all $v\in V$. Since $V$ generates $\Cl$, this implies $xyx\inv=(-1)^{(\deg x)(\deg y)}y$ for all homogeneous $y\in \Cl$. This is the condition that $x$ lies in the supercentre of $\Cl$ (where we consider $\Cl$ as a superalgebra via its $(\Z/2\Z)$-grading), which we claim consists only of scalars.  Indeed, for $I \subseteq [n]$, write $e_I=\prod_{i\in I}e_i$. (Pick one order of the product for each $I$; it does not matter which one.)  Suppose $\sum_{I \subseteq [n]}a_I e_I$ is in the supercentre. (In particular, it must be purely even or purely odd.)  Now suppose $I \subseteq [n]$ is nonempty, and pick $i \in I$. Comparing the coefficients of $e_{I \setminus \{i\}}$ on both sides of the equation
    \[
        \left( \sum_{I \subseteq [n]} a_I e_I \right)e_i
        = (-1)^{|I|} e_i \left( \sum_{I \subseteq [n]} a_I e_I \right)
    \]
    yields $a_I=0$.

    The only scalars lying in $\Pin(V)$ are $\pm 1$. Hence, $g\inv v_1 v_2 \dotsm v_k = \pm 1$. If $g\inv v_1 v_2 \dotsm v_k = 1$, we are done. If $g\inv v_1 v_2 \dotsm v_k = -1$, then pick $w \in V$ with $\Phi_V(w,w)=1$,
    and write $-1=w(-w)$.  Then we have $g = v_1 v_2 \dotsm v_{k+2}$ with $v_{k+1}=w$ and $v_{k+2}=-w$, and we are done.

    Finally, \cref{spinvgen} follows from \cref{pinvgen} by noting that each $v_i$ lies in $\Cl_\odd$.
\end{proof}

\begin{rem}[Low values of $N$] \label{lowNgroup}
    It will be important for some of the inductive arguments in the paper that we allow $N \in \{0,1,2\}$, even though, in some ways, these behave differently than the cases $N \ge 3$.  Let $C_2$ denote the cyclic group on $2$ elements.  Then we have the following:
    \begin{itemize}
        \item When $N=0$, $\Pin(V) = \Spin(V) = \{\pm 1\} \cong C_2$.
        \item When $N=1$, $\Pin(V) = \{\pm 1, \pm v\} \cong C_2 \times C_2$, where $v \in V$ satisfies $\Phi_V(v,v)=1$ (so that $v^2=1$), and $\Spin(V) = \{\pm 1\} \cong C_2$.
        \item When $N=2$, we have an isomorphism
            \begin{equation} \label{roti}
                \GG_m \xrightarrow{\cong} \Spin(V),\qquad
                t \mapsto t + (t^{-1}-t) \psi_1^\dagger \psi_1,
            \end{equation}
            where $\GG_m$ is the multiplicative group.  Next, note that $\Pin(V) = \Spin(V) \sqcup \Spin(V) e_1$ and conjugation by $e_1$ corresponds, under the above isomorphism, to inversion on $\GG_m$.  Thus, $\Pin(V) \cong \GG_m \rtimes C_2$, where the nontrivial element of $C_2$ acts on $\GG_m$ by inversion.
    \end{itemize}
\end{rem}

\details{
    The isomorphism described in the $N=2$ case above can be extended to describe the maximal torus of $\Spin(V)$ for $N \ge 2$.  For $i=1,2,\dotsc,n$, the map $t \mapsto t + (t^{-1}-t) \psi_i^\dagger \psi_i$ is a commuting family of 1-parameter subgroups of $\Spin(V)$.  This gives a homomorphism $\chi \colon \GG_m^n \to \Spin(V)$ whose image $H$ is a maximal torus of $\Spin(V)$.  We have
    \[
        \ker(\chi)
        = \left\{ (\zeta_1,\dotsc,\zeta_n) : \zeta_i \in \{\pm 1\},\ \prod_{i=1}^n \zeta_i=1 \right\}.
    \]
    If we identify $X^\ast(\GG_m^n) = \Hom(\GG_m^n,\GG_m) \cong \Z^n$ in the usual way, this identifies $X^\ast(H) \subseteq \Z^n$ as the set of $(\la_1,\dotsc,\la_n)$ such that either all $\la_i$ are odd or all $\la_i$ are even.  Since
    \[
        \left( t + (t^{-1}-t) \psi_i^\dagger \psi_i \right) x_I
        =
        \begin{cases}
            t x_I & \text{if } i \notin I, \\
            t\inv x_I & \text{if } i \in I,
        \end{cases}
    \]
    this gives another computation of the weights of the spin module.
}

Implicit in the proof of \cref{icts} is a short exact sequence (for all $N \in \N$)
\[
    \{1\} \to \{\pm 1\} \to \Pin(V)\to \rO(V) \to \{1\}.
\]
Restricting to $\Spin(V)$ yields another short exact sequence
\[
    \{1\} \to \{\pm 1\} \to \Spin(V)\to \SO(V) \to \{1\}.
\]
The group $\Spin(V)$ is connected for $N \ge 2$.  This explains why the image of the third map above lies in the connected group $\SO(V)$, and realises $\Spin(V)$ as the universal cover of $\SO(V)$ for $N\geq 3$.
\details{
    Here is an argument, by induction on $N$, that $\Spin(V)$ is connected for $N \geq 2$.  The action \cref{cha} induces an action of $\Spin(V)$ on the connected space $\{v \in V : \Phi_V(v,v)=1\}$.  This latter action factors through the surjective map $\Spin(V) \to \SO(V)$.  Since $\SO(V)$ acts transitively on the given space, so does $\Spin(V)$.  The stabiliser of $v$ is $\{\pm 1\}\times \Spin(v^\perp)$. By induction $\Spin(v^\perp)$ is connected and to show that $\{\pm 1\}$ lies in the identity component of $\Spin(V)$ is a simple $N=2$ computation.
}

\begin{rem} \label{hamster}
    If $N$ is odd, then the element $e_1 e_2 \dotsm e_N \in \Pin(V) \setminus \Spin(V)$ is central, and $\Pin(V)$ is generated by $\Spin(V)$ and this central element. In this case, the difference between the representation theory of $\Pin(V)$ and $\Spin(V)$ is not significant.  We will focus on $\Spin(V)$-modules when $N$ is odd; see \cref{gerbil}.
    \details{
        Every $\Spin(V)$-module has two $\Pin(V)$ lifts.  So $\Pin(V)$-mod is basically two copies of $\Spin(V)$-mod and we can get from one to the other by tensoring with $\triv^1$, to be defined in \cref{subsec:classirr}; see \cref{trivflip}.
    }
\end{rem}

\section{Special orthogonal Lie algebras\label{sec:Liealg}}

In this section we collect some basic facts about the special orthogonal Lie algebra $\fso(V)$.  Since $\fso(V)$ is the zero Lie algebra when $N \le 1$, we assume throughout this section that $N \ge 2$.

The Lie algebra $\Lie(\Cl^\times)$ is $\Cl$ with the commutator Lie bracket.  The inclusion $\Spin(V) \hookrightarrow \Cl^\times$ induces an inclusion $\Lie(\Spin(V)) \hookrightarrow \Cl$. The image of this inclusion is
\[
    \Cl^2 := \Span_\C \{ uv-vu : u,v \in V\} \subseteq \Cl.
\]
\details{
    Let $X$ be the image of the inclusion $\Lie(\Spin(V)) \hookrightarrow \Cl$.  Since $\Cl^2$ has basis $[e_i,e_j]$, $1 \le i < j \le N$, we have
    \[
        \dim_\C \Cl^2 = \tfrac{1}{2}{N(N-1)} = \dim_\C \Lie(\Spin(V)) = \dim_\C X.
    \]
    Hence, it suffices to show that $X \subseteq \Cl^2$. Furthermore, since the set of $(u,v) \in V \times V$ such that
    \begin{itemize}
        \item $u$ and $v$ are linearly independent and
        \item the restriction of $\Phi_V$ to $\Span_\C \{u,v\}$ is nondegenerate
    \end{itemize}
    is dense in $V \times V$, it suffices to show that $[u,v] \in X$ for all such pairs $(u,v)$.

    Suppose $u,v \in V$ are linearly independent and the restriction of $\Phi_V$ to $U = \Span_\C \{u,v\}$ is nondegenerate.  This implies that $V \cong U \oplus U^\perp$. Let $a,b\in \C$ and $g=a+buv$. We claim that $g \in \GPin(V)$ if $g$ is invertible.  To see this, it suffices to show that $g w \iota(g) \in V$ for all $w \in V$, since $g^{-1} = (g \iota(g))^{-1} \iota(g)$. If $w \in U^\perp$ then $w$ commutes with $g$ and $\iota(g)$ and we are done, since $g\iota(g) \in \C$.  By linearity, it remains to consider $w \in U$, i.e., we are reduced to a computation in the Clifford algebra of $U$. Then $g w \iota(g)\in \Cl(U)_\odd = U$, since $U$ is 2-dimensional.

    Next, since $\iota(g)=a+bvu$, we have
    \[
        g \iota(g)
        = a^2 + 2ab \Phi_V(u,v) + b^2 \Phi_V(u,u) \Phi_V(v,v).
    \]
    The condition that $a+buv\in \Pin(V)$ is therefore
    \[ \label{abcond} \tag{$\maltese$}
        a^2+2ab\Phi_V(u,v)+b^2\Phi_V(u,u)\Phi_V(v,v)=1.
    \]
    We will now compute the line in $\Lie(\Spin(V)) \subseteq \Lie(\Cl^\times)=\Cl$ corresponding to this 1-parameter subgroup.  To do so, set $a=1+\epsilon x$ and $b=\epsilon y$ where $\epsilon^2=0$.  The condition \cref{abcond} becomes $x+y\Phi_V(u,v)=0$, and so we have
    \[
        1 + \epsilon(x+yuv)
        = 1 + \epsilon \big( -y\Phi_V(u,v) + yuv \big)
        = 1 + \epsilon \tfrac{y}{2} [u,v].
    \]
    Hence, $[u,v]$ spans the desired line in $\Lie(\Spin(V))$, and so $[u,v] \in X$, as desired.
}
The group homomorphism $\Spin(V)\to \SO(V)$ induces an isomorphism of Lie algebras. Under the identification of $\Lie(\Spin(V))$ with $\Cl^2$ above, this isomorphism is
\begin{equation} \label{Canberra}
    \gamma \colon \Cl^2\to \fso(V), \qquad \gamma(uv-vu)=4M_{u,v},
\end{equation}
where $M_{u,v}\in \fso(V)$ is defined by
\begin{equation} \label{thai}
    M_{u,v} w = \Phi_V(v,w) u - \Phi_V(u,w) v.
\end{equation}

\details{
    We compute the image of $[u,v]$ under the map $\Lie(\Spin(V))\to \fso(V)$:
    \begin{gather*}
        (1+\epsilon[u,v])u(1-\epsilon[u,v])
        = u + 4\epsilon(\Phi_V(u,v)u-\Phi_V(u,u)v)
        = u + 4\epsilon M_{u,v}(u),
        \\
        (1+\epsilon[u,v])v(1-\epsilon[u,v])
        = v + 4\epsilon(\Phi_V(v,v)u-\Phi_V(u,v)v)
        = v + 4\epsilon M_{u,v}(v).
    \end{gather*}
    Furthermore, if $w$ is perpendicular to both $u$ and $v$, then $w$ commutes with $[u,v]$, and so
    \[
        (1+\epsilon[u,v])u(1-\epsilon[u,v])=w=w+4 \epsilon M_{u,v}(w).
    \]
    Then \cref{Canberra} follows.
}
\details{
    To see directly that $M_{u,v} \in \fso(V)$, we compute
    \begin{multline*}
        \Phi_V( M_{u,v} w, w' ) + \Phi_V( w, M_{u,v} w' )
        \\
        = \Phi_V(v,w) \Phi_V(u,w') - \Phi_V(u,w) \Phi_V(v, w')
        + \Phi_V(v,w') \Phi_V(w,u) - \Phi_V(u,w') \Phi_V(w,v)
        = 0.
    \end{multline*}
    The $M_{u,v}$ span $\fso(V)$, and we have an isomorphism of $\C$-modules
    \[
        \Lambda^2(V) \xrightarrow{\cong} \fso(V),\qquad
        u \wedge v \mapsto M_{u,v},\quad u,v \in V.
    \]
}

\subsection{Type $D$ (even $N$)\label{subsec:spinD}}

We suppose throughout this subsection that $N=2n$ is even, and we continue to assume that $N \ge 2$.  Note that
\begin{itemize}
    \item when $n=1$, $\fso(V)$ is a one-dimensional abelian Lie algebra, and
    \item when $n \ge 2$, $\fso(V)$ is the semisimple Lie algebra of type $D_n$.
\end{itemize}

For $A \in \Mat_n(\C)$, let $A'$ denote the flip of $A$ in the antidiagonal.  More precisely,
\begin{equation} \label{trip}
    \text{if } \quad
    A = (a_{ij})_{i,j=1}^n
    \quad \text{then} \quad
    A' = (a_{n-j+1,n-i+1})_{i,j=1}^n.
\end{equation}
Note that $(AB)' = B' A'$ for $A,B \in \Mat_n(\C)$.  In the ordered basis $\psi_1, \dotsc, \psi_n, \psi_n^\dagger, \dotsc, \psi_1^\dagger$ of $V$, the matrices of $\fso(V)$ are those of the form
\[
    \begin{pmatrix}
        A & B \\
        C & -A'
    \end{pmatrix}
    ,\qquad
    A,B,C \in \Mat_n(\C),\
    B' = -B,\
    C' = -C.
\]
The Cartan subalgebra $\fh$ consists of the diagonal matrices.  For $1 \le i \le n$, define
\[
    \epsilon_i \in \fh^*,\quad
    \epsilon_i( \diag(a_1,\dotsc,a_n, -a_n,\dotsc,-a_1) ) = a_i.
\]

Let $E_{ij}$, $1 \le i,j \le n$ denote the usual matrix units of $\Mat_n(\C)$, and define, for $1 \le i,j \le n$,
\begin{equation} \label{greenD}
    A_{ij} =
    \begin{pmatrix}
        E_{ij} & 0 \\
        0 & -E_{ij}'
    \end{pmatrix}
    ,\
    B_{ij} =
    \begin{pmatrix}
        0 & E_{i,n-j+1} - E_{j,n-i+1} \\
        0 & 0
    \end{pmatrix}
    ,\
    C_{ij} =
    \begin{pmatrix}
        0 & 0 \\
        E_{n-i+1,j} - E_{n-j+1,i} & 0
    \end{pmatrix}
    .
\end{equation}
Then the $A_{ij}$, $B_{ij}$, and $C_{ij}$ are the root vectors, with corresponding roots $\epsilon_i - \epsilon_j$, $\epsilon_i + \epsilon_j$, and $-\epsilon_i - \epsilon_j$, respectively.
\details{
    Let
    \[
        H = \diag(a_1,\dotsc,a_n,-a_n,\dotsc,-a_1)
        =
        \begin{pmatrix}
            \sum_{k=1}^n a_k E_{kk} & 0 \\
            0 & -\sum_{k=1}^n a_k E_{kk}'
        \end{pmatrix}
        .
    \]
    Then
    \[
        [H, A_{ij}] =
        \begin{pmatrix}
            \sum_{k=1}^n a_k [E_{kk}, E_{ij}] & 0 \\
            0 & \sum_{k=1}^n a_k [E_{kk}', E_{ij}']
        \end{pmatrix}
        = (a_i-a_j) A_{ij}
        = (\epsilon_i - \epsilon_j)(H) A_{ij}
    \]
    and
    \begin{multline*}
        [H, B_{ij}] =
        \begin{pmatrix}
            0 & \sum_{k=1}^n (a_k E_{kk} (E_{i,n-j+1}-E_{j,n-i+1}) + a_k (E_{i,n-j+1}-E_{j,n-i+1}) E_{kk}') \\
            0 & 0
        \end{pmatrix}
        \\
        =
        \begin{pmatrix}
            0 & a_i E_{i,n-j+1} - a_j E_{j,n-i+1} + a_j E_{i,n-j+1} - a_i E_{j,n-i+1} \\
            0 & 0
        \end{pmatrix}
        = (a_i+a_j) B_{ij}
        = (\epsilon_i + \epsilon_j)(H) B_{ij}
    \end{multline*}
    and
    \begin{multline*}
        [H, C_{ij}]
        =
        \begin{pmatrix}
            0 & 0 \\
            -\sum_{k=1}^n (a_k E_{kk}' (E_{n-i+1,j} - E_{n-j+1,i}) - a_k (E_{n-i+1,j} - E_{n-j+1,i}) E_{kk}) & 0
        \end{pmatrix}
        \\
        =
        \begin{pmatrix}
            0 & 0 \\
            -a_i E_{n-i+1,j} + a_j E_{n-j+1,i} - a_j E_{n-i+1,j} + a_i E_{n-j+1,i} & 0
        \end{pmatrix}
        = - (a_i+a_j) C_{ij}
        = -(\epsilon_i+\epsilon_j)(H) C_{ij}.
    \end{multline*}
}
We choose the positive system of roots given by
\[
    \epsilon_i \pm \epsilon_j,\quad 1 \le i < j \le n.
\]
Thus, the positive root spaces of $\fso(V)$ are spanned by
\[
    A_{ij},\quad B_{ij},\quad 1 \le i < j \le n.
\]
It is straightforward to verify, recalling the definition \cref{thai} of $M_{u,v}$, that
\[
    2 M_{\psi_i,\psi_j^\dagger} = A_{ij},\quad
    2 M_{\psi_i,\psi_j} = B_{ij},\quad
    2 M_{\psi_i^\dagger,\psi_j^\dagger} = C_{ij},\qquad
    1 \le i,j \le n.
\]
\details{
    For $1 \le i,j,k \le n$, we have
    \begin{gather*}
        2 M_{\psi_i,\psi_j^\dagger} \psi_k
        = 2 \Phi_V( \psi_j^\dagger, \psi_k ) \psi_i - 2 \Phi_V( \psi_i, \psi_k ) \psi_j^\dagger
        = \delta_{jk} \psi_i
        = A_{ij} \psi_k,
        \\
        2 M_{\psi_j,\psi_i^\dagger} \psi_k^\dagger
        = 2 \Phi_V( \psi_i^\dagger, \psi_k^\dagger ) \psi_j - 2 \Phi_V( \psi_j, \psi_k^\dagger ) \psi_i^\dagger
        = - \delta_{jk} \psi_i^\dagger
        = A_{ij} \psi_k^\dagger.
    \end{gather*}
    Similarly,
    \begin{gather*}
        2 M_{\psi_i,\psi_j} \psi_k
        = 2 \Phi_V( \psi_j, \psi_k ) \psi_i - 2 \Phi_V( \psi_i, \psi_k ) \psi_j
        = 0
        = B_{ij} \psi_k,
        \\
        M_{\psi_i,\psi_j} \psi_k^\dagger
        = 2 \Phi_V( \psi_j, \psi_k^\dagger ) \psi_i - 2 \Phi_V( \psi_i, \psi_k^\dagger ) \psi_j
        = \delta_{jk} \psi_i - \delta_{ik} \psi_j
        = B_{ij} \psi_k^\dagger,
    \end{gather*}
    and
    \begin{gather*}
        2 M_{\psi_i^\dagger, \psi_j^\dagger} \psi_k
        = 2 \Phi_V( \psi_j^\dagger, \psi_k ) \psi_i^\dagger - 2 \Phi_V( \psi_i^\dagger, \psi_k ) \psi_j^\dagger
        = \delta_{jk} \psi_i^\dagger - \delta_{ik} \psi_j^\dagger
        = C_{ij} \psi_k,
        \\
        2 M_{\psi_i^\dagger, \psi_j^\dagger} \psi_k^\dagger
        = 2 \Phi_V( \psi_j^\dagger, \psi_k^\dagger ) \psi_i^\dagger - 2 \Phi_V( \psi_i^\dagger, \psi_k^\dagger ) \psi_j^\dagger
        = 0
        = C_{ij} \psi_k.
    \end{gather*}
}
Thus, the positive root vectors are
\[
    2 M_{\psi_i,\psi_j^\dagger},\quad
    2 M_{\psi_i,\psi_j},\quad
    1 \le i < j \le n.
\]
The images under the isomorphism $\gamma\inv$, given in \cref{Canberra}, of these elements are
\begin{equation} \label{ramen}
    \psi_i \psi_j^\dagger,\quad
    \psi_i \psi_j,\quad
    1 \le i < j \le n.
\end{equation}
We also have that
\begin{equation} \label{bamboo}
    \gamma\inv(A_{ii})
    = \gamma\inv \left( 2 M_{\psi_i,\psi_i^\dagger} \right)
    = \tfrac{1}{2} (\psi_i \psi_i^\dagger - \psi_i^\dagger \psi_i)
    \overset{\cref{fermion}}{=} \psi_i \psi_i^\dagger - \tfrac{1}{2},
    \quad 1 \le i \le n.
\end{equation}
For $n \ge 2$, the dominant integral weights are those weights of the form
\begin{equation} \label{dominantD}
    \begin{gathered}
        \lambda = \sum_{i=1}^n \lambda_i \epsilon_i,\quad
        \lambda_1 \ge \lambda_2 \ge \dotsb \ge \lambda_{n-1} \ge |\lambda_n|,
        \\
        \text{such that} \quad (\lambda_i \in \tfrac{1}{2} + \Z \text{ for all } 1 \le i \le n) \quad \text{or} \quad
        (\lambda_i \in \Z \text{ for all } 1 \le i \le n).
    \end{gathered}
\end{equation}
For $n=1$, we adopt the convention that the set of dominant integral weights is $\frac{1}{2}\Z \epsilon_1$.

\begin{rem}[$N=2$]
    When $N=2$,
    \[
        \fso(V) =
        \left\{ \begin{pmatrix} a_1 & 0 \\ 0 & -a_1 \end{pmatrix} : a_1 \in \C \right\}
    \]
    is a one-dimensional abelian Lie algebra.  For $z \in \C$, we call the one-dimensional representation $z \epsilon_1 \colon \fso(V) \to \C \cong \End_\C(\C)$ the \emph{simple highest-weight $\fso(V)$-module with highest weight $z \epsilon_1$} since this will often allow us to make uniform statements for $N \ge 2$.
\end{rem}

\subsection{Type $B$ (odd $N$)\label{subsec:spinB}}

We suppose throughout this subsection that $N=2n+1$ is odd, so that $\fso(V)$ is the simple Lie algebra of type $B_n$.  We continue to assume that $N \ge 2$, that is, $n \ge 1$.

In the ordered basis $\psi_1,\dotsc,\psi_n,\frac{1}{\sqrt{2}} e_{2n+1},\psi_n^\dagger,\dotsc\psi_1^\dagger$, the matrices of $\fso(V)$ are those of the form
\[
    \begin{pmatrix}
        A & u & B & \\
        -v^\transpose & 0 & -u^\transpose \\
        C & v & -A'
    \end{pmatrix}
    , \qquad u,v \in \C^n,\ A,B,C \in \Mat_n(\C),\ B' = -B,\ C' = -C,
\]
where the notation $A'$ is defined in \cref{trip}.  The Cartan subalgebra $\fh$ consists of the diagonal matrices.  For $1 \le i \le n$, define
\[
    \epsilon_i \in \fh^*,\quad
    \epsilon_i( \diag(a_1,\dotsc,a_n,0,-a_n,\dotsc,-a_1) ) = a_i.
\]

Recall that $E_{ij}$, $1 \le i,j \le n$, denote the usual matrix units of $\Mat_n(\C)$, and let $u_i$ be the element of $\C^n$ with a $1$ in the $i$-th position and $0$ in all other positions.  Then define, for $1 \le  i,j \le n$,
\begin{equation} \label{greenB}
    \begin{gathered}
        A_{ij} =
        \begin{pmatrix}
            E_{ij} & 0 & 0 \\
            0 & 0 & 0 \\
            0 & 0 & -E_{ij}'
        \end{pmatrix}
        ,\quad
        X_i =
        \begin{pmatrix}
            0 & u_i & 0 \\
            0 & 0 & -u_{n-i+1}^\transpose\\
            0 & 0 & 0
        \end{pmatrix}
        ,\quad
        Y_i =
        \begin{pmatrix}
            0 & 0 & 0 \\
            -u_i^\transpose & 0 & 0 \\
            0 & u_{n-i+1} & 0
        \end{pmatrix}
        ,
        \\
        B_{ij} =
        \begin{pmatrix}
            0 & 0 & E_{i,n-j+1} - E_{j,n-i+1} \\
            0 & 0 & 0 \\
            0 & 0 & 0
        \end{pmatrix}
        ,\quad
        C_{ij} =
        \begin{pmatrix}
            0 & 0 & 0 \\
            0 & 0 &0 \\
            E_{n-i+1,j} - E_{n-j+1,i} & 0 & 0
        \end{pmatrix}
        .
    \end{gathered}
\end{equation}
Then the $A_{ij}$, $B_{ij}$, $C_{ij}$, $X_i$, and $Y_i$ are the root vectors, with corresponding roots $\epsilon_i - \epsilon_j$, $\epsilon_i + \epsilon_j$, $-\epsilon_i - \epsilon_j$, $\epsilon_i$, and $-\epsilon_i$, respectively.
\details{
    For $A_{ij}$, $B_{ij}$, and $C_{ij}$, the verification goes as in the type $D$ case.  Let
    \[
        H = \diag(a_1,\dotsc,a_n,-a_n,\dotsc,-a_n)
        =
        \begin{pmatrix}
            \sum_{k=1}^n a_k E_{kk} & 0 & 0 \\
            0 & 0 & 0 \\
            0 & 0 & -\sum_{k=1}^n a_k E_{kk}'
        \end{pmatrix}
        .
    \]
    Then
    \[
        [H, X_i]
        =
        \begin{pmatrix}
            0 & a_i u_i & 0 \\
            0 & 0 & - a_i u_{n-i+1}^\transpose \\
            0 & 0 & 0
        \end{pmatrix}
        = a_i X_i
        = \epsilon_i(H) X_i
    \]
    and
    \[
        [H, Y_i]
        =
        \begin{pmatrix}
            0 & 0 & 0 \\
            a_i u_i^\transpose & 0 & 0 \\
            0 & -a_i u_{n-i+1} & 0
        \end{pmatrix}
        = -a_i Y_i
        = -\epsilon_i(H) Y_i.
    \]
}
We choose the positive system of roots given by
\[
    \epsilon_i \pm \epsilon_j,\quad \epsilon_k, \qquad 1 \le i < j \le n,\quad 1 \le k \le n.
\]
Thus, the positive root spaces of $\fso(V)$ are spanned by
\[
    A_{ij},\qquad B_{ij},\qquad X_k,\qquad 1 \le i < j \le n,\quad 1 \le k \le n.
\]
It is straightforward to verify, recalling the definition \cref{thai} of $M_{u,v}$, that we have
\[
    2M_{\psi_i,\psi_j^\dagger} = A_{ij},\quad
    2M_{\psi_i,\psi_j} = B_{ij},\quad
    2M_{\psi_i^\dagger,\psi_j^\dagger} = C_{ij},\quad
    \sqrt{2} M_{\psi_i,e_{2n+1}} = X_i,\quad
    \sqrt{2} M_{\psi_i^\dagger,e_{2n+1}} = Y_i.
\]
\details{
    The first three equalities are verified as in type $D$ case.  For $1 \le i,j \le n$, we have
    \begin{gather*}
        M_{\psi_i,e_{2n+1}} e_{2n+1}
        = \Phi_V(e_{2n+1},e_{2n+1}) \psi_i - \Phi_V(\psi_i,e_{2n+1}) e_{2n+1}
        = \psi_i
        = \tfrac{1}{\sqrt{2}} X_i e_{2n+1},
        \\
        M_{\psi_i,e_{2n+1}} \psi_j^\dagger
        = \Phi_V(e_{2n+1}, \psi_j^\dagger) \psi_i - \Phi_V(\psi_i, \psi_j^\dagger) e_{2n+1}
        = - \tfrac{1}{2} \delta_{ij} e_{2n+1}
        = \tfrac{1}{\sqrt{2}} X_i \psi_j^\dagger,
        \\
        M_{\psi_i,e_{2n+1}} \psi_j = 0 = \tfrac{1}{\sqrt{2}} X_i \psi_j,
    \end{gather*}
    and
    \begin{gather*}
        M_{\psi_i^\dagger, e_{2n+1}} e_{2n+1}
        = \Phi_V(e_{2n+1}, e_{2n+1}) \psi_i^\dagger - \Phi_v(\psi_i^\dagger, e_{2n+1}) e_{2n+1}
        = \psi_i^\dagger
        = \tfrac{1}{\sqrt{2}} Y_i e_{2n+1},
        \\
        M_{\psi_i^\dagger, e_{2n+1}} \psi_j
        = \Phi_V(e_{2n+1},\psi_j) \psi_i^\dagger - \Phi_V(\psi_i^\dagger, \psi_j) e_{2n+1}
        = - \tfrac{1}{2} \delta_{ij} e_{2n+1}
        = \tfrac{1}{\sqrt{2}} Y_i \psi_j,
        \\
        M_{\psi_i^\dagger, e_{2n+1}} \psi_j^\dagger
        = \Phi_V(e_{2n+1}, \psi_j^\dagger) \psi_i^\dagger - \Phi_V(\psi_i^\dagger, \psi_j^\dagger) e_{2n+1}
        = 0
        = \tfrac{1}{\sqrt{2}} Y_i \psi_j^\dagger.
    \end{gather*}
}
Thus, the positive root vectors are
\[
    2M_{\psi_i, \psi_j^\dagger},\qquad
    2M_{\psi_i,\psi_j},\qquad
    \sqrt{2} M_{\psi_k,e_{2n+1}},\qquad
    1 \le i < j \le n,\quad 1 \le k \le n.
\]
The images under the isomorphism $\gamma\inv$, given in \cref{Canberra}, of these elements are
\[
    \psi_i \psi_j^\dagger,\quad
    \psi_i \psi_j,\quad
    \tfrac{1}{\sqrt{2}} \psi_k e_{2n+1},\qquad
    1 \le i < j \le n,\quad 1 \le k \le n.
\]
We also have that
\begin{equation} \label{bambooB}
    \gamma\inv(A_{ii})
    = \gamma\inv \left( 2 M_{\psi_i,\psi_i^\dagger} \right)
    = \tfrac{1}{2} (\psi_i \psi_i^\dagger - \psi_i^\dagger \psi_i)
    \overset{\cref{fermion}}{=} \psi_i \psi_i^\dagger - \tfrac{1}{2},
    \quad 1 \le i \le n.
\end{equation}
The dominant integral weights are those weights of the form
\begin{equation} \label{dominantB}
    \begin{gathered}
        \lambda = \sum_{i=1}^n \lambda_i \epsilon_i,\quad
        \lambda_1 \ge \lambda_2 \ge \dotsb \ge \lambda_{n-1} \ge \lambda_n \ge 0,
        \\
        \text{such that} \quad (\lambda_i \in \tfrac{1}{2} + \Z \text{ for all } 1 \le i \le n) \quad \text{or} \quad
        (\lambda_i \in \Z \text{ for all } 1 \le i \le n).
    \end{gathered}
\end{equation}

\section{Representations of pin and spin groups\label{sec:reps}}

In this section, we collect some facts about representations of the pin and spin groups that will be important for us.

\subsection{The spin and vector modules\label{suave}}

Recall the $\Cl$-module $S$ introduced in \cref{subsec:Cliffmod}.  By restriction, $S$ is a $\Pin(V)$-module and a $\Spin(V)$-module.  Passing to the Lie algebra, we obtain a $\fso(V)$-module structure on $S$, most conveniently computed via the isomorphism $\gamma$ obtained in \cref{Canberra}.  With respect to the Cartan subalgebras introduced in \cref{subsec:spinD,subsec:spinB}, the vectors $x_I$ for $I \subseteq [n]$ are all weight vectors.

First suppose that $N$ is even.  As $\fso(V)$-modules and as $\Spin(V)$-modules, we have a decomposition
\begin{equation}\label{defn:spm}
    S = S^+ \oplus S^-,\qquad
    S^+ = \Span_\C \{ x_I : |I| \text{ is even} \},\qquad
    S^- = \Span_\C \{ x_I : |I| \text{ is odd} \}.
\end{equation}
When $N \ge 2$, we also see that $S^+$ is a simple highest-weight $\fso(V)$-module with highest-weight vector $x_\varnothing$ of weight $\frac{1}{2}(\epsilon_1 + \dotsb + \epsilon_n)$ and that $S^-$ is a simple highest-weight $\fso(V)$-module with highest-weight vector $x_{\{n\}}$ of weight $\frac{1}{2}(\epsilon_1 + \dotsb + \epsilon_{n-1} - \epsilon_n)$.  As a $\Pin(V)$-module, $S$ remains simple; see \cref{hydro} below.

Now suppose that $N$ is odd.  In this case, there are two choices of a $\Cl$-module structure on $S$ depending on the choice of $\varepsilon\in \{\pm 1\}$, as in \cref{mouse}, but they give rise to isomorphic $\Spin(V)$-modules.
\details{
    The fact that they are isomorphic as $\Spin(V)$-modules follows from the fact that they are highest-weight modules of highest weight $\frac{1}{2}(\epsilon_1 + \dotsb + \epsilon_n)$.  However, they are \emph{not} isomorphic as $\Pin(V)$-modules.  To see this, let $w = e_1 e_2 \dotsm e_{2n}$. Then $w e_N$ is central, and so acts by a scalar.  We then compute $w e_{N} x_\varnothing = \varepsilon w x_\varnothing$.  Since $w x_\varnothing$ is independent of choice of $\varepsilon$, we see that $we_{N}$ acts by a different scalar in the two modules.
}
In this case, there is a unique highest-weight vector $x_\varnothing$, so the spin module $S$ is a simple $\fso(V)$-module of highest weight $\frac{1}{2}(\epsilon_1 + \dotsb + \epsilon_n)$.
\details{
    The vector $x_\varnothing$ is clearly killed by the positive root vectors.  Furthermore, for $1 \le i \le n$, we have
    \[
        \left( \psi_i \psi_i^\dagger - \tfrac{1}{2} \right) x_\varnothing
        = \left( 1 - \tfrac{1}{2} \right) x_\varnothing
        = \tfrac{1}{2} x_\varnothing.
    \]
}

We view $V$ as a $\Pin(V)$-module with action
\begin{equation} \label{PinVact}
    g \cdot v := gvg^{-1},\qquad g \in \Pin(V),\ v \in V.
\end{equation}
As a representation of $\fso(V)$, $V$ is simple with highest weight $\epsilon_1$ if $N\geq 3$.

\begin{rem}[Low values of $N$] \label{lowNrep}
    As noted in \cref{lowNgroup}, the cases $N \le 2$ behave differently than the cases $N \ge 3$.
    \begin{itemize}
        \item When $N=0$, we have $V=S^-=0$ and $S^+$ is the nontrivial one-dimensional module for $\Pin(V) \cong C_2$.  Of course, $S$ is the trivial module for the trivial group $\Spin(V)$ and the zero Lie algebra $\fso(V)$.
        \item When $N=1$, we have that $V$ is the trivial $\Pin(V)$-module.  We also have that $S$ is the nontrivial one-dimensional module for $\Spin(V) \cong C_2$.  The $\Pin(V)$-module structure on $S$ depends on the choice of $\varepsilon \in \{\pm 1\}$, as in \cref{mouse}.
        \item When $N=2$, recall the isomorphism $\GG_m \cong \Spin(V)$ of \cref{roti}.  Let $L_r$, $r \in \Z$, denote the one-dimensional $\GG_m$-module with action $t \cdot v = t^r v$, $t \in \GG_m$, $v \in L_r$.  Since
            \[
                \left( t + (t^{-1}-t) \psi_1^\dagger \psi_1 \right) x_\varnothing = t x_\varnothing,\qquad
                \left( t + (t^{-1}-t) \psi_1^\dagger \psi_1 \right) x_{\{1\}} = t^{-1} x_{\{1\}},
            \]
            we have $S^\pm \cong L_{\pm 1}$.  We also have $V \cong L_{-2} \oplus L_2$.
            \details{
                We have
                \begin{multline*}
                    \left( t + (t^{-1}-t) \psi_1^\dagger \psi_1 \right) \psi_1 \left( t + (t^{-1}-t) \psi_1^\dagger \psi_1 \right)^{-1}
                    \\
                    = \left( t + (t^{-1}-t) \psi_1^\dagger \psi_1 \right) \psi_1 \left( t + (t^{-1}-t) \psi_1 \psi_1^\dagger \right)
                    = t^2 \psi_1,
                \end{multline*}
                and so $\C \psi_1 \cong L_2$.  A similar computation shows that $\C \psi_1^\dagger \cong L_{-2}$.
            }
            As $\fso(V)$-modules, we have $L_r = L \left( \frac{r}{2} \epsilon_1 \right)$.  Both $V$ and $S$ are simple as modules for $\Pin(V) \cong \GG_m \rtimes C_2$, with the generator of $C_2$ interchanging the summands $L_r$ and $L_{-r}$.
    \end{itemize}
\end{rem}

\subsection{Classification of simple modules\label{subsec:classirr}}

When $N  \le 1$, the groups $\Spin(V)$ and $\Pin(V)$ are finite (see \cref{lowNgroup}), and their representation theory is straightforward.  Therefore, we assume in this subsection that $N \ge 2$.

We have an exact sequence of groups
\begin{equation} \label{SES}
    \{1\} \to \Spin(V) \to \Pin(V) \xrightarrow{\pi} \{\pm 1\} \to \{1\},
\end{equation}
where $\{\pm 1\}$ is the cyclic group of order $2$, written multiplicatively, and $\pi(g) = (-1)^{\deg g}$.  The finite-dimensional representation theory of $\Pin(V)$ can be described in terms of the representation theory of $\Spin(V)$ using Clifford theory.  Since $\Pin(V)$ and $\Spin(V)$ are reductive, their categories of finite-dimensional representations are both semisimple, and so it suffices to describe their simple modules.  We begin with a discussion of the representation theory of $\Spin(V)$.

The group $\Spin(V)$ is connected and reductive.  Let $H$ denote its abstract Cartan.  This is canonically isomorphic to the abelianisation of every Borel subgroup of $\Spin(V)$. Write $X^\ast(H)=\Hom(H,\GG_m)$ for the weight lattice of $\Spin(V)$. Dominance is defined in the usual way from any choice of Borel subgroup.  We write $X^\ast(H)^+$ for the subset of dominant weights.

The choice of Borel subalgebra of $\fso(V)$ spanned by the $A_{ij}$, $B_{ij}$, and $X_k$, for $1 \leq i < j \leq n$ and $1 \leq k \leq n$ (the $X_k$ only appearing in type $B$) induces an isomorphism $X^\ast(H)\otimes_{\Z} \C \cong \fh^\ast$, which we use to write down elements of $X^\ast(H)$ as linear combinations of $\epsilon_1,\epsilon_2,\dotsc,\epsilon_n$.

Write $\Irr(\Spin(V))$ for the set of isomorphism classes of finite-dimensional simple $\Spin(V)$-modules. These are classified by highest weight theory. Explicitly, there is an isomorphism $X^\ast(H)^+ \cong \Irr(\Spin(V))$, $\la\mapsto L(\la)$, characterised by the following fact: For all Borel subgroups $B$ of $\Spin(V)$, there exists nonzero $v\in L(\la)$ such that $bv=\la(b)v$ for all $b\in B$.

The group $\Pin(V)$ acts on $\Spin(V)$ by conjugation.  For $g \in \Pin(V)$ and $W$ a $\Spin(V)$-module, we let $W^g$ denote the $\Spin(V)$-module that is equal to $W$ as a vector space, but with the twisted action
\begin{equation} \label{twist}
    h \cdot w = \left( ghg^{-1} \right) w,\qquad h \in \Spin(V),\ w \in W^g,
\end{equation}
where the juxtaposition $hw$ denotes the action of $h \in \Spin(V)$ on $w \in W$.

The group $\Pin(V)$ also acts by conjugation on $H$ and hence by precomposition on $X^\ast(H)^+ \cong \Irr(\Spin(V))$.  We let $g \lambda$ denote the result of $g \in \Pin(V)$ acting on $\lambda \in X^\ast(H)$. The subgroup $\Spin(V)$ acts trivially, so this descends to an action of the quotient $\pi_0(\Pin(V)) \cong \{\pm 1\}$.  For $\lambda \in X^\ast(H)^+$ and $g \in \Pin(V)$, we have
\[
    L(\lambda)^g \cong L(g \lambda).
\]
In particular, up to isomorphism, $L(\lambda)^g$ depends only on $\lambda$ and the class of $g$ in $\pi_0(\Pin(V))$.

To pass between representations of $\Spin(V)$ and  $\Pin(V)$ we use the biadjoint pair of restriction and induction functors
\begin{equation} \label{cricket}
    \Res \colon \Pin(V)\md \to \Spin(V)\md
    \qquad \text{and} \qquad
    \Ind \colon \Spin(V)\md \to \Pin(V)\md,
\end{equation}
where $G\md$ denotes the category of finite-dimensional modules of an algebraic group $G$.  These satisfy
\begin{equation}\label{resind}
    \Res \circ \Ind(W) \cong W \oplus W^P,
\end{equation}
where $P$ is any element of $\Pin(V) \setminus \Spin(V)$.  In order to make explicit computations, we will fix
\begin{equation} \label{Pdef}
    P =
    \begin{cases}
        e_1 e_2 \dotsm e_N & \text{if $N$ is odd}, \\
        e_1 e_2 \dotsm e_{N-1} & \text{if $N$ is even}.
    \end{cases}
\end{equation}

\begin{prop} \label{sponge}
    Let $W$ be a simple $\Pin(V)$-module.  Then there exists a unique $\pi_0(\Pin(V))$-orbit $\mathcal{O}$ on $X^\ast(H)^+$ such that
    \begin{equation}\label{orbitres}
        \Res(W) \cong \bigoplus_{\la\in \mathcal{O}} L(\la).
    \end{equation}
    Furthermore, given an orbit $\mathcal{O}$, the number of nonisomorphic simple $\Pin(V)$-modules $W$ satisfying \cref{orbitres} is equal to the size of the stabiliser of $\pi_0(\Pin(V))$ acting on an element of $\mathcal{O}$.
\end{prop}

\begin{proof}
    By Frobenius reciprocity, every simple $\Pin(V)$-module is a simple summand of $\Ind(M)$ for some simple $\Spin(V)$-module $M$.
    \details{
        If $L$ is a simple $\Pin(V)$-module, then
        \[
            \Hom_{\Pin(V)}(L, \Ind \circ \Res(L))
            \cong \Hom_{\Spin(V)}(\Res(L), \Res(L)) \ne 0.
        \]
        Thus, $L$ is a simple summand of $\Ind \circ \Res(L)$.
    }
    Thus, it suffices to prove the result for such simple summands.

    Suppose $M$ is a simple $\Spin(V)$-module.  By Frobenius reciprocity,
    \[
        \dim \Hom_{\Pin(V)}(\Ind(M),\Ind(M))
        = \dim \Hom_{\Spin(V)}(M,\Res \circ \Ind(M))
    \]
    which, by \cref{resind}, is equal to two if $M \cong M^P$, and is equal to one otherwise.  In the former case, $\Ind M$ is of the form $W_1 \oplus W_2$ with $W_1$, $W_2$ nonisomorphic simple modules satisfying  $\Res(W_1) \cong \Res(W_2) \cong M$.  Thus $W_1$ and $W_2$ satisfy \cref{orbitres}, with the orbit $\mathcal{O}$ having one element, namely $M$.  In the latter case, $W=\Ind(M)$ is simple and also satisfies \cref{orbitres} by \cref{resind}.  The final statement of the proposition also follows from this discussion.
\end{proof}

\begin{rem} \label{hongmen}
    It follows from \cref{sponge} that the simple $\Pin(V)$-modules are:
    \begin{itemize}
        \item $\Ind(M)$ for a simple $\Spin(V)$-module $M$ with $M^P \not\cong M$,
        \item the two simple summands of $\Ind(M)$ for a simple $\Spin(V)$-module $M$ with $M^P \cong M$.
    \end{itemize}
    In particular, if $\mathcal{O}$ is an orbit of size two, then the unique simple $\Pin(V)$-module $W$ satisfying \cref{orbitres} is $\Ind (L(\la))$ where $\la$ is any element of $\mathcal{O}$.
\end{rem}

We let $\triv^0$ denote the trivial $\Pin(V)$-module and let $\triv^1$ be the one-dimensional module with action given by $g v = (-1)^{\deg g} v$, $v \in \triv^1$.  If $M$ is a simple $\Spin(V)$-module fixed under the $\Pin(V)$-action (i.e., $M^g \cong M$ as $\Spin(V)$-modules for $g \in \Pin(V)$), and $M'$ and $M''$ are its two lifts to a $\Pin(V)$-module, then these are related by
\begin{equation} \label{trivflip}
    M' \otimes \triv^1 \cong M''.
\end{equation}

We now study the action of $\pi_0(\Pin(V))$ on $\Irr(\Spin(V))\cong X^\ast(H)^+$.  When $N$ is even, define
\begin{equation} \label{reflect}
    \tilde{\lambda} := \lambda_1 \epsilon_1 + \dotsb + \lambda_{n-1} \epsilon_{n-1} - \lambda_n \epsilon_n
    \qquad \text{for }
    \lambda = \lambda_1 \epsilon_1 + \dotsb + \lambda_{n-1} \epsilon_{n-1} + \lambda_n \epsilon_n.
\end{equation}

\begin{prop} \label{hydro}
    If $N$ is odd, then $\pi_0(\Pin(V))$ acts trivially on $\Irr(\Spin(V))\cong X^\ast(H)^+$. If $N$ is even, the action of the nontrivial element $P \in \pi_0(\Pin(V))$ is $P \lambda = \tilde{\lambda}$.
\end{prop}

\begin{proof}
    If $N$ is odd then $P$ is central and so there is nothing to do. From now on, suppose $N=2n$ is even. Then
    \begin{equation} \label{PactD}
        P e_i P^{-1}
        = (-1)^{\delta_{iN}} e_i
        ,\qquad 1 \le i \le N.
    \end{equation}
    It follows that
    \begin{equation} \label{desk}
        P \psi_n P^{-1} = \psi_n^\dagger,\quad
        P \psi_n^\dagger P^{-1} = \psi_n,\quad
        P \psi_i P^{-1} = \psi_i,\quad
        P \psi_i^\dagger P^{-1} = \psi_i^\dagger,\quad
        1 \le i < n.
    \end{equation}
    Hence conjugation by $P$ preserves the set of positive root vectors \cref{ramen} and acts on the elements \cref{bamboo} of the Cartan subalgebra of $\fso(V)$ as
    \begin{gather*}
        P \left( \psi_i \psi_i^\dagger - \tfrac{1}{2} \right) P^{-1}
        = \psi_i \psi_i^\dagger - \tfrac{1}{2},\qquad 1 \le i < n,
        \\
        P \left( \psi_n \psi_n^\dagger - \tfrac{1}{2} \right) P^{-1}
        = \psi_n^\dagger \psi_n - \tfrac{1}{2}
        \overset{\cref{fermion}}{=} - \left( \psi_n \psi_n^\dagger - \tfrac{1}{2} \right).
        \qedhere
    \end{gather*}
\end{proof}

For the remainder of this subsection, we assume that $N$ is even.  It follows from \cref{hydro} that, for any dominant integral weight $\lambda$, we have
\begin{equation}\label{pactweights}
    L(\lambda)^P \cong L(\tilde{\lambda})
    \quad \text{as $\Spin(V)$-modules}.
\end{equation}
In particular, for $N \ge 2$,
\begin{equation} \label{ice}
    (S^\pm)^P \cong S^\mp
    \quad \text{as $\Spin(V)$-modules}.
\end{equation}
For all even $N$,
\begin{equation} \label{fire}
    S^P \cong S,\quad V^P \cong V
    \quad \text{as $\Spin(V)$-modules},
\end{equation}
since $V = L(-\epsilon_1) \oplus L(\epsilon_1)$ when $N=2$ (see \cref{lowNrep}), and $V = L(\epsilon_1)$ for $N \ge 3$.  Note that, when $N \ge 4$,
\[
    S \cong \Ind \left( L \left( \tfrac{1}{2} \epsilon_1 + \tfrac{1}{2} \epsilon_2 + \dotsb + \tfrac{1}{2} \epsilon_{n-1} \pm \tfrac{1}{2} \epsilon_n \right) \right).
\]

\begin{lem} \label{echo}
    Suppose $N$ is even.  Let $M_1$ and $M_2$ be two simple $\Pin(V)$-modules whose restrictions to $\Spin(V)$ are isomorphic. Then, for all $r\geq 1$, the multiplicities of $M_1$ and $M_2$ in $S^{\otimes r}$ are equal.
\end{lem}

\begin{proof}
    Since $S \cong \Ind(S^\pm)$ is self-dual, we have, for $i \in \{1,2\}$,
    \begin{multline*}
        \Hom_{\Pin(V)}(S^{\otimes r},M_i)
        \cong \Hom_{\Pin(V)}(S, S^{\otimes (r-1)}\otimes M_i)
        \\
        \cong \Hom_{\Pin(V)} \left( \Ind(S^+), S^{\otimes (r-1)} \otimes M_i \right)
        \cong \Hom_{\Spin(V)} (S^+, S^{\otimes(r-1)} \otimes M_i),
    \end{multline*}
    where we used Frobenius reciprocity in the final isomorphism.  Since $M_1$ and $M_2$ are isomorphic upon restriction to $\Spin(V)$, the result follows.
\end{proof}

\subsection{Invariant bilinear form}

For a subset $I$ of $[n]$, we let $I^\complement = [n]\setminus I$ denote its complement.  Define a bilinear form on $S$ by
\begin{equation} \label{Sform}
    \Phi_S(x_I, x_J) =
    \begin{cases}
        (-1)^{\binom{|I|}{2} + nN|I| + |\{(i,j)\in I\times I^\complement : i>j\}|}  & \text{if } J = I^\complement, \\
        0 & \text{otherwise},
    \end{cases}
\end{equation}
and extending by bilinearity.

\begin{lem}
    We have
    \begin{equation} \label{bounce}
        \Phi_S(v x, y) = (-1)^{nN} \Phi_S(x, v y),
        \qquad x,y \in S,\quad v \in V \subseteq \Cl.
    \end{equation}
\end{lem}

\begin{proof}
    Since both sides of \cref{bounce} are linear in $v$, $x$ and $y$, it suffices to prove that
    \begin{equation} \label{lynel}
        \Phi_S(\psi_k^\dagger x_I, x_J) = (-1)^{nN} \Phi_S(x_I, \psi_k^\dagger x_J)
        \quad \text{and} \quad
        \Phi_S(\psi_k x_I, x_J) = (-1)^{nN} \Phi_S(x_I, \psi_k x_J),
    \end{equation}
    for all $1 \le k \le n$ and $I,J \subseteq [n]$, and, if $N$ is odd, that
    \begin{equation} \label{lynel2}
        \Phi_S(e_{2n+1} x_I, x_J) = (-1)^n \Phi_S(x_I, e_{2n+1} x_J),
    \end{equation}
    for all $I,J \subseteq [n]$.

    For $I,J \subseteq [n]$, define
    \begin{equation} \label{hulk}
        \sigma_{I,J}
        := (-1)^{|\{ (i,j) \in I \times J : i > j \}|}.
    \end{equation}
    Then, for $I,J,I_1,J_1,I_2,J_2 \subseteq [n]$, with $I \cap J = I_1 \cap I_2 = J_1 \cap J_2 = \varnothing$, we have
    \begin{equation} \label{sigmund}
        \sigma_{I,J} = (-1)^{|I||J|} \sigma_{J,I},\qquad
        \sigma_{I_1 \sqcup I_2, J} = \sigma_{I_1,J} \sigma_{I_2,J},\qquad
        \sigma_{I,J_1 \sqcup J_2} = \sigma_{I,J_1} \sigma_{I,J_2}.
    \end{equation}

    First note that both sides of the first equation in \cref{lynel} are zero unless $I \cap J = \varnothing$ and $I \cup J = \{1,\dotsc,k-1,k+1,\dotsc,n\}$.  Thus, we assume that $I$ and $J$ satisfy these two conditions.  Then
    \[
        \Phi_S(\psi_k^\dagger x_I, x_J)
        \overset{\cref{mind}}{=} \sigma_{\{k\},I} \Phi_S(x_{I \cup \{k\}}, x_J)
        \overset{\cref{Sform}}{\underset{\cref{sigmund}}{=}} (-1)^{\binom{|I|+1}{2} + nN(|I|+1)} \sigma_{\{k\},I} \sigma_{I,J} \sigma_{\{k\},J}
    \]
    and
    \[
        \Phi_S(x_I, \psi_k^\dagger x_J)
        \overset{\cref{mind}}{=} \sigma_{\{k\},J} \Phi_S(x_I, x_{J \cup \{k\}})
        \overset{\cref{Sform}}{\underset{\cref{sigmund}}{=}} (-1)^{\binom{|I|}{2} + nN|I|} \sigma_{\{k\},J} \sigma_{I,J} \sigma_{I,\{k\}}.
    \]
    Since
    \begin{equation} \label{ghost}
        \sigma_{I,\{k\}} \sigma_{\{k\},I} = (-1)^{|I|}
        \qquad \text{and} \qquad
        \binom{|I|}{2} + |I| = \binom{|I|+1}{2},
    \end{equation}
    the first equality in \cref{lynel} follows.

    Next, note that both sides of the second equality in \cref{lynel} are zero unless $I \cap J = \{k\}$ and $I \cup J = [n]$.  Thus, we assume that $I$ and $J$ satisfy these two conditions.  Then
    \begin{multline*}
        \Phi_S(\psi_k x_I, x_J)
        \overset{\cref{mind}}{=} \sigma_{\{k\},I} \Phi_S(x_{I \setminus \{k\}}, x_J)
        \\
        \overset{\cref{Sform}}{=} (-1)^{\binom{|I|-1}{2} + nN(|I|-1)} \sigma_{\{k\},I} \sigma_{I \setminus \{k\}, J}
        \overset{\cref{sigmund}}{=} (-1)^{\binom{|I|-1}{2} + nN(|I|-1)} \sigma_{\{k\},I} \sigma_{I,J} \sigma_{\{k\}, J}
    \end{multline*}
    and
    \begin{multline*}
        \Phi_S(x_I, \psi_k x_J)
        \overset{\cref{mind}}{=} \sigma_{\{k\}, J} \Phi_S(x_I, x_{J \setminus \{k\}})
        \\
        \overset{\cref{Sform}}{=} (-1)^{\binom{|I|}{2} + nN|I|} \sigma_{\{k\}, J} \sigma_{I,J \setminus \{k\}}
        \overset{\cref{sigmund}}{=} (-1)^{\binom{|I|}{2} + nN|I|} \sigma_{\{k\}, J} \sigma_{I,J} \sigma_{I,\{k\}}.
    \end{multline*}
    Using the second equality in \cref{ghost} with $|I|$ replaced by $|I|-1$ then implies the second equality in \cref{lynel}.

    Now suppose that $N$ is odd. To prove \cref{lynel2}, we assume that $I \cap J = \varnothing$ and $I \cup J = [n]$, since otherwise both sides are zero.  Then we have
    \[
        \Phi_S(e_{2n+1} x_I, x_J)
        \overset{\cref{mouse}}{=} \pm (-1)^{|I|} \Phi_S(x_I, x_J)
        = \pm (-1)^{n + |J|} \Phi_S(x_I, x_J)
        \overset{\cref{mouse}}{=} (-1)^n \Phi_S(x_I, e_{2n+1} x_J),
    \]
    as desired.
\end{proof}

As in the introduction, define
\begin{equation}\label{gvdefn}
    \Group(V) :=
    \begin{cases}
        \Pin(V) & \text{if $N$ is even}, \\
        \Spin(V) & \text{if $N$ is odd}.
    \end{cases}
\end{equation}

\begin{cor} \label{taian}
    We have
    \begin{gather}
       \label{darker}
        \Phi_S(gx,gy) = \Phi_S(x,y)
        \qquad \text{for all } g \in \Group(V),\ x,y \in S,
    \\
    \label{dark}
        \Phi_S(Xx,y) = - \Phi_S(x,Xy)
        \qquad \text{for all } X \in \fso(V),\ x,y \in S.
    \end{gather}
\end{cor}

\begin{proof}
    When $N<2$, the identity \cref{darker} is trivial, since $\Group(V)=\{\pm 1\}$, acting by the scalar $\{\pm 1\}$ on $S$. Now suppose $N \geq 2$.  Let $v_1,\dotsc,v_k \in V$ satisfy $\Phi_V(v_i,v_i) = 1$ for all $1 \le i \le k$.  Then, for $x,y \in S$, it follows from \cref{bounce} that
    \begin{align*}
        \Phi_S(v_1 v_2 \dotsm v_k x, v_1 v_2 \dotsm v_k y)
        &=
        \begin{cases}
            \Phi_S(x, v_k \dotsm v_1 v_1 \dotsm v_k y) & \text{if } N=2n,
            \\
            (-1)^{kn} \Phi_S(x, v_k \dotsm v_1 v_1 \dotsm v_k y) & \text{if } N=2n+1,
        \end{cases}
        \\
        &=
        \begin{cases}
            \Phi_S(x,y) & \text{if } N=2n,
            \\
            (-1)^{kn} \Phi_S(x,y) & \text{if } N=2n+1.
        \end{cases}
    \end{align*}
    Thus, \cref{darker} follows from \cref{pinvgen,spinvgen}.  The identity \cref{dark} follows from \cref{darker} by differentiating.
    \details{
        We give here a direct proof of \cref{dark}.  Since the expression $uv-vu$ is linear in both $u$ and $v$, to prove \cref{dark} it suffices, by \cref{Canberra}, to show that
        \[
            \Phi_S( (uv-vu) x, y ) = - \Phi_S( x, (uv-vu)y )
        \]
        for all $u,v \in \{e_1,\dotsc,e_N\}$ and $x,y \in S$.  This follows easily from \cref{bounce}.
    }
\end{proof}

\begin{prop} \label{spiral}
    We have
    \begin{equation} \label{formsym}
        \Phi_S(x,y) = (-1)^{\binom{n}{2} + nN} \Phi_S(y,x)
        \qquad \text{for all } x,y \in S.
    \end{equation}
\end{prop}

\begin{proof}
    Since $\Phi_S(x_I,x_J) = 0 = \Phi_S(x_J,x_I)$ unless $I \cup J = [n]$ and $I \cap J = \varnothing$, we assume that $I$ and $J$ satisfy these two conditions.  Then, defining $\sigma_{I,J}$ as in \cref{hulk}, we have
    \begin{multline*}
        \Phi_S(x_J, x_I)
        \overset{\cref{Sform}}{=} (-1)^{\binom{|J|}{2} + nN|J|} \sigma_{J,I}
        \overset{\cref{sigmund}}{=} (-1)^{\binom{|J|}{2} + nN|J| + |I||J|} \sigma_{I,J}
        \\
        \overset{\cref{Sform}}{=} (-1)^{\binom{|J|}{2}+\binom{|I|}{2}+nN(|I|+|J|)+|I||J|} \Phi_S(x_I,x_J).
    \end{multline*}
    Then the result follows from the fact that $nN(|I|+|J|) = n^2 N \equiv nN$ modulo $2$ and that $\binom{|I|}{2} + \binom{|J|}{2} + |I||J| = \binom{n}{2}$.
\end{proof}

\begin{rem}\label{gerbil}
    \Cref{taian} implies that $S$ is self-dual.  Since $S$ is also simple, as noted above, $\Phi_S$ is the \emph{unique} invariant bilinear form on $S$, up to scalar multiple.  On the other hand, if $N \equiv 3 \pmod{4}$, then $S$ is not self-dual as a $\Pin(V)$-module, and so there is no $\Pin(V)$-invariant bilinear form on $S$.
    \details{
        Suppose $N \equiv 3 \pmod{4}$. Then $P = e_1 e_2 \dotsb e_N$ is an element in the centre of $\Pin(V)$ with $P^2=-1$.  Since $-1$ acts on $S$ by $-1$, it follows from Schur's lemma that $P$ must act on $S$ by $\pm i$. Therefore, $P$ acts on $S^\ast$ by $\mp i$. Hence $S$ is not isomorphic to $S^\ast$. (Recall that there were two choices of $S$ depending on a choice of $\varepsilon$; see \cref{mouse}.  Taking duals switches these two choices in this case.)
    }
    This is our main motivation for defining $\Group(V)$ to be $\Spin(V)$ when $N$ is odd; see also \cref{hamster}.
\end{rem}

\subsection{Tensor product decompositions}

We now recall some tensor product decompositions that will be important for us.  For a weight $\Spin(V)$-module $M$, we let $\wt(M)$ denote its set of weights.  Thus, for example, when $N \ge 2$,
\[
    \wt(S) = \left\{ \left( \pm \tfrac{1}{2}, \dotsc, \pm \tfrac{1}{2} \right) \right\}.
\]
For the next result, recall, from \cref{subsec:spinD}, that the set of dominant integral weights is $\frac{1}{2}\Z\epsilon_1$ when $N=2$.

\begin{lem} \label{lem:pieri}
    Suppose $N \ge 2$, and let $\la$ be a dominant integral weight. Then
    \[
        S \otimes L(\la) \cong \bigoplus_{\epsilon \in \wt(S)} L(\la+\epsilon)
        \quad \text{as $\Spin(V)$-modules},
    \]
    where we define $L(\la+\epsilon)$ to be zero if $\la+\epsilon$ is not dominant.
\end{lem}

\begin{proof}
    For $N=2$, this is a straightforward direct computation using the description of $S$ in \cref{lowNrep}. For $N \ge 3$, it is a standard application of the Weyl character formula.
    \details{
        It suffices to prove an equality of characters.  Let $\Delta$ be the Weyl denominator. By the Weyl character formula, we find
        \begin{align*}
            \operatorname{ch}(L(\la)\otimes S)&= \frac{1}{\Delta} \sum_{w\in W} \sgn(w) e^{w(\la+\rho)} \sum_{\epsilon\in \wt(S)} e^{\epsilon} \\
            &=\frac{1}{\Delta} \sum_{w\in W}\sgn(w) e^{w(\la+\rho)} \sum_{\epsilon\in \wt(S)} e^{w\epsilon} \\
            &= \sum_{\epsilon\in \wt(S)} \frac{1}{\Delta} \sum_{w\in W}\sgn(w) e^{w(\la+\epsilon+\rho)},
        \end{align*}
        where, for the second equality, we use the fact that $\wt(S)$ is $W$-invariant.  Now suppose that $\epsilon\in \wt(S)$ is such that $\la+\epsilon$ is not dominant. Note that, for all $i$, $\alpha_i^\vee(\epsilon)\geq -1$.  Since $\la$ is dominant, we must have $\alpha_i^\vee(\la+\epsilon)=-1$ and therefore $\alpha_i^\vee(\la+\epsilon+\rho)=0$.  Hence $s_i(\la+\epsilon+\rho)=\la+\epsilon+\rho$, so in this case
        \[
            \sum_{w\in W}\sgn(w) e^{w(\la+\epsilon+\rho)}=0,
        \]
        and we may conclude using the Weyl character formula.
    }
\end{proof}

\begin{cor} \label{SV}
    \begin{enumerate}[wide]
        \item \label{SV1} If $N=2$, then
            \begin{equation}
                S \otimes V \cong S \oplus \Ind \left( L \left( \tfrac{3}{2} \epsilon_1 \right) \right)
                \quad \text{as $\Pin(V)$-modules}.
            \end{equation}

        \item \label{SV2} If $N=2n+1 \ge 3$ (type $B_n$), then
            \begin{equation} \label{SVB}
                S \otimes V
                \cong S \otimes L(\epsilon_1)
                \cong S \oplus L \left( \tfrac{3}{2} \epsilon_1 + \tfrac{1}{2} \epsilon_2 + \dotsb + \tfrac{1}{2} \epsilon_n \right)
                \quad \text{as $\Spin(V)$-modules}.
            \end{equation}

        \item \label{SV3} If $N=2n \ge 4$ (type $D_n$), then
            \begin{equation} \label{SVD}
                S \otimes V \cong S \otimes L(\epsilon_1)
                \cong S \oplus \Ind \left( L \left( \tfrac{3}{2} \epsilon_1 + \tfrac{1}{2} \epsilon_2 + \dotsb + \tfrac{1}{2} \epsilon_n \right) \right)
                \quad \text{as $\Pin(V)$-modules}.
            \end{equation}
    \end{enumerate}
\end{cor}

\begin{proof}
    Part~\cref{SV1} is a direct computation using \cref{lowNrep}.  Parts~\cref{SV2,SV3} follow from \cref{lem:pieri,dominantD,dominantB}, where the appearance of $\Ind$ in part~\cref{SV3} follows from \cref{hongmen}.
\end{proof}

\begin{prop} \label{lemon}
    \begin{enumerate}[wide]
        \item \label{lemonB} If $N=2n+1$ (type $B_n$), we have
            \begin{equation} \label{lemonB1}
                \Lambda^k(V) \cong \Lambda^{N-k}(V)
                \quad \text{as $\Pin(V)$-modules},
                \quad 0 \le k \le n,
            \end{equation}
            and $\Lambda^k(V)$ is simple for $0 \le k \le N$.  Furthermore, if $n \ge 1$, we have
            \begin{equation} \label{lemonB2}
                \Lambda^k(V) \cong \Lambda^{N-k}(V) \cong L(\epsilon_1 + \dotsb + \epsilon_k)
                \quad \text{as $\Spin(V)$-modules},
                \quad 0 \le k \le n.
            \end{equation}

        \item \label{lemonD} If $N=2n$ (type $D_n$), we have
            \begin{equation} \label{zest}
                \Lambda^k(V) \not\cong \Lambda^{N-k}(V)
                \quad \text{as $\Pin(V)$-modules},\quad
                0 \le k < n,
            \end{equation}
            and $\Lambda^k(V)$ is simple for $0 \le k \le N$.  Furthermore, if $n \ge 2$, we have
            \begin{gather} \label{lemonD1}
                \Lambda^k(V) \cong \Lambda^{N-k}(V) \cong L(\epsilon_1 + \dotsb + \epsilon_k),\qquad
                0 \le k < n,
                \\ \label{lemonD2}
                \Lambda^n(V) \cong L(\epsilon_1 + \dotsb + \epsilon_{n-1} + \epsilon_n) \oplus L(\epsilon_1 + \dotsb + \epsilon_{n-1} - \epsilon_n)
            \end{gather}
            as $\Spin(V)$-modules.
    \end{enumerate}
\end{prop}

\begin{proof}
    For $N \le 2$, the results follow from straightforward computations using the explicit descriptions of $V$ and $S$ given in \cref{lowNrep}.
    \details{
        Suppose $N=2$.  We have $V = L_{-2} \oplus L_2$ as a module for $\Spin(V) \cong \GG_m$.  This is nontrivial as a module for $\Pin(V) \cong \GG_m \rtimes C_2$ with the nontrivial element of $C_2$ interchanging the summands $L_{-2}$ and $L_2$.  Therefore, the nontrivial element of $C_2$ acts as $-1$ on $\Lambda^2(V)$, and so $\Lambda^2(V)$ is not isomorphic to the trivial module $\Lambda^0(V)$.  Furthermore, $\Lambda^0(V)$ and $\Lambda^1(V)$ are both simple since they are one-dimensional.  Finally, $\Lambda^1(V) \cong V$ is also simple as a $\Pin(V)$-module.  The cases $N \in \{0,1\}$ are completely straightforward.
    }

    Now suppose that $N \ge 3$.  A proof that $\Lambda^k(V) \cong L(\epsilon_1 + \dotsb + \epsilon_k)$ as $\fso(V)$-modules, and hence as $\Spin(V)$-modules, for the given ranges on $k$ can be found, for instance, in \cite[Th.~13.9, Th.~13.11]{Car05}.  (The ranges on $k$ are slightly more restrictive there, since those results relate exterior powers to fundamental modules, but the proofs give the isomorphisms for our ranges on $k$.)  To prove \cref{lemonD2}, one notes that $\epsilon_1 + \dotsb + \epsilon_{n-1} \pm \epsilon_n$ are both weights that appear in $\Lambda^n(V)$.  Furthermore, they are highest weights since adding any simple root produces a weight that does not appear in $\Lambda^n(V)$.  Hence, $\Lambda^n(V)$ contains a submodule isomorphic to the right-hand side of \cref{lemonD2}.  A straightforward application of the Weyl dimension formula then shows that this submodule is all of $\Lambda^n(V)$.

    Next, note that we have a pairing of $\Lambda^k(V)$ with $\Lambda^{N-k}(V)$ given by the composition
    \[
        \Lambda^k(V) \otimes \Lambda^{N-k}(V)
        \xrightarrow{\wedge} \Lambda^{N}(V) \xrightarrow{\cong} \C.
    \]
    This is a $\Spin(V)$-module homomorphism, and so identifies $\Lambda^{N-k}(V)$ with the dual of $\Lambda^k(V)$.  Since $\Lambda^k(V)$ is self-dual, this yields an isomorphism of $\Lambda^{N-k}(V)$ with $\Lambda^k(V)$ as $\Spin(V)$-modules.

    In type $B_n$, the element $P$, defined in \cref{Pdef}, acts trivially on $V$, and so the actions of $P$ on $\Lambda^k(V)$ and $\Lambda^{N-k}(V)$ are also trivial.  This completes the proof of \cref{lemonB1}.

    In type $D_n$, the highest-weight spaces of $\Lambda^k(V)$ and $\Lambda^{N-k}(V)$, $1 \le k < n$, are spanned, respectively, by
    \[
        v_k := \psi_1 \wedge \psi_2 \wedge \dotsb \wedge \psi_k
        \quad \text{and} \quad
        w_k := \psi_1 \wedge \psi_2 \wedge \dotsb \wedge \psi_n \wedge \psi_n^\dagger \wedge \dotsb \wedge \psi_{k+1}^\dagger.
    \]
    By \cref{desk}, the action of $P$ on these highest-weight vectors is given by
    \[
        P \cdot v_k = v_k,\qquad
        P \cdot w_k = - w_k.
    \]
    Thus, $\Lambda^k(V) \not \cong \Lambda^{N-k}(V)$ as $\Pin(V)$-modules.
\end{proof}

\begin{cor} \label{Sdub}
    \begin{enumerate}
        \item When $N=2n+1$ (type $B_n$), we have
        \begin{equation} \label{SdubB}
            S^{\otimes 2} \cong \bigoplus_{k=0}^n \Lambda^k(V)
            \qquad \text{as $\Spin(V)$-modules}.
        \end{equation}

        \item When $N=2n$ (type $D_n$), we have
        \begin{equation} \label{SdubD}
            S^{\otimes 2} \cong \bigoplus_{k=0}^{2n} \Lambda^k(V)
            \qquad \text{as $\Pin(V)$-modules}.
        \end{equation}
    \end{enumerate}
\end{cor}

\begin{proof}
    \begin{enumerate}[wide]
        \item When $N=1$, it follows immediately from the descriptions of $S$ and $V$ given in \cref{lowNrep} that $S^{\otimes 2} \cong \triv^0 \cong \Lambda^0(V)$.  For $N \ge 3$, it follows from \cref{lem:pieri} that
            \[
                S^{\otimes 2} \cong \bigoplus_{k=0}^n L(\epsilon_1 + \dotsb + \epsilon_k)
                \cong \bigoplus_{k=0}^n \Lambda^k(V)
                \qquad \text{as $\Spin(V)$-modules}.
            \]

        \item When $N=0$, it follows immediately from the descriptions of $S$ and $V$ given in \cref{lowNrep} that $S^{\otimes 2} \cong \triv^0 \cong \Lambda^0(V)$.  Now suppose $N=2$.  Then, in the notation of \cref{lowNrep}, we have
            \[
                S^{\otimes 2}
                \cong L_{-2} \oplus L_2 \oplus L_0^{\oplus 2}
            \]
            as modules for $\Spin(V) \cong \GG_m$.  As $\Pin(V)$-modules, the summand $L_{-2} \oplus L_2$ is isomorphic to $\Lambda^1(V) \cong V$.  It remains to show that the summand $L_0^{\oplus 2}$ contains the trivial $\Pin(V)$-module $\Lambda^0(V)$ and the nontrivial $\Pin(V)$-module $\Lambda^2(V)$.  As modules for the subgroup $C_2 \subseteq \GG_m \rtimes C_2 \cong \Pin(V)$, $S$ decomposes as a sum of the trivial module and the nontrivial $C_2$-module.  Hence, $S^{\otimes 2}$ contains two copies of the trivial $C_2$-module and two copies of the nontrivial $C_2$-module.  Since the summand $L_{-2} \oplus L_2$ contains one of each, we are done.

            For $N \ge 4$, it follows from \cref{lem:pieri} and \cref{lemon}\cref{lemonD} that
            \begin{align*}
                S^{\otimes 2}
                &\cong (S^+ \otimes S) \oplus (S^- \otimes S)
                \\
                &\cong \left( \bigoplus_{k=0}^n L(\epsilon_1 + \dotsb + \epsilon_k) \right) \oplus \left( L(\epsilon_1 + \dotsb + \epsilon_{n-1} - \epsilon_n) \oplus \bigoplus_{k=0}^{n-1} L(\epsilon_1 + \dotsb + \epsilon_k) \right)
                \\
                &\cong \bigoplus_{k=0}^{2n} \Lambda^k(V),
            \end{align*}
            as $\Spin(V)$-modules.  The result then follows from \cref{echo} and \cref{lemon}\cref{lemonD}. \qedhere
    \end{enumerate}
\end{proof}

\section{The spin Brauer category\label{sec:spinBrauer}}

In this section, we introduce our main category of interest.   We work over an arbitrary commutative ring $\kk$ in which $2$ is invertible.  (Note, however, that \cref{SBdef} below does not require that $2$ is invertible.)

\begin{defin} \label{SBdef}
    For $d,D \in \kk$ and $\kappa \in \{\pm 1\}$, the \emph{spin Brauer category} $\SB(d,D;\kappa)$ is the strict $\kk$-linear monoidal category presented as follows.  The generating objects are $\Sgo$ and $\Vgo$, whose identity morphisms we depict by a black strand and a dotted blue strand:
    \[
        \idstrand{spin} := 1_\Sgo,\qquad
        \idstrand{vec} := 1_\Vgo.
    \]
    The generating morphisms are
    \begin{gather*}
        \capmor{spin} \colon \Sgo \otimes \Sgo \to \one,\qquad
        \cupmor{spin} \colon \one \to \Sgo \otimes \Sgo,\qquad
        \capmor{vec} \colon \Vgo \otimes \Vgo \to \one,\qquad
        \cupmor{vec} \colon \one \to \Vgo \otimes \Vgo,
        \\
        \crossmor{spin}{spin} \colon \Sgo \otimes \Sgo \to \Sgo \otimes \Sgo,\qquad
        \crossmor{vec}{spin} \colon \Vgo \otimes \Sgo \to \Sgo \otimes \Vgo,\qquad
        \crossmor{spin}{vec} \colon \Sgo \otimes \Vgo \to \Vgo \otimes \Sgo,\qquad
        \crossmor{vec}{vec} \colon \Vgo \otimes \Vgo \to \Vgo \otimes \Vgo,
        \\
        \mergemor{vec}{spin}{spin} \colon \Vgo \otimes \Sgo \to \Sgo.
    \end{gather*}
    To state the defining relations, we will use the convention that a relation involving $r \ge 1$ dashed red strands (as in \cref{brauer}) means we impose the $2^r$ relations obtained from replacing each such strand with either a black strand or a dotted blue strand.  The defining relations on morphisms are then as follows:
    \begin{gather} \label{brauer}
        \begin{tikzpicture}[centerzero]
            \draw[any] (0.2,-0.4) to[out=135,in=down] (-0.15,0) to[out=up,in=225] (0.2,0.4);
            \draw[any] (-0.2,-0.4) to[out=45,in=down] (0.15,0) to[out=up,in=-45] (-0.2,0.4);
        \end{tikzpicture}
        \ =\
        \begin{tikzpicture}[centerzero]
            \draw[any] (-0.2,-0.4) -- (-0.2,0.4);
            \draw[any] (0.2,-0.4) -- (0.2,0.4);
        \end{tikzpicture}
        \ ,\quad
        \begin{tikzpicture}[centerzero]
            \draw[any] (0.4,-0.4) -- (-0.4,0.4);
            \draw[any] (0,-0.4) to[out=135,in=down] (-0.32,0) to[out=up,in=225] (0,0.4);
            \draw[any] (-0.4,-0.4) -- (0.4,0.4);
        \end{tikzpicture}
        \ =\
        \begin{tikzpicture}[centerzero]
            \draw[any] (0.4,-0.4) -- (-0.4,0.4);
            \draw[any] (0,-0.4) to[out=45,in=down] (0.32,0) to[out=up,in=-45] (0,0.4);
            \draw[any] (-0.4,-0.4) -- (0.4,0.4);
        \end{tikzpicture}
        \ ,\quad
        \begin{tikzpicture}[centerzero]
            \draw[any] (-0.3,0.4) -- (-0.3,0) arc(180:360:0.15) arc(180:0:0.15) -- (0.3,-0.4);
        \end{tikzpicture}
        =
        \begin{tikzpicture}[centerzero]
            \draw[any] (0,-0.4) -- (0,0.4);
        \end{tikzpicture}
        = \
        \begin{tikzpicture}[centerzero]
            \draw[any] (-0.3,-0.4) -- (-0.3,0) arc(180:0:0.15) arc(180:360:0.15) -- (0.3,0.4);
        \end{tikzpicture}
        \ ,\quad
        \begin{tikzpicture}[anchorbase]
            \draw[any] (-0.15,-0.4) to[out=60,in=-90] (0.15,0) arc(0:180:0.15) to[out=-90,in=120] (0.15,-0.4);
        \end{tikzpicture}
        = \,
        \capmor{any}
        \ ,\quad
        \begin{tikzpicture}[centerzero]
            \draw[any] (-0.2,-0.3) -- (-0.2,-0.1) arc(180:0:0.2) -- (0.2,-0.3);
            \draw[any] (-0.3,0.3) \braiddown (0,-0.3);
        \end{tikzpicture}
        =
        \begin{tikzpicture}[centerzero]
            \draw[any] (-0.2,-0.3) -- (-0.2,-0.1) arc(180:0:0.2) -- (0.2,-0.3);
            \draw[any] (0.3,0.3) \braiddown (0,-0.3);
        \end{tikzpicture}
        \ ,
        \\ \label{typhoon}
        \begin{tikzpicture}[centerzero]
            \draw[vec] (-0.4,-0.4) -- (-0.1,-0.1);
            \draw[spin] (0,-0.4) -- (-0.1,-0.1) -- (0.4,0.4);
            \draw[spin] (0.4,-0.4) to[out=110,in=-20] (-0.4,0.4);
        \end{tikzpicture}
        =
        \begin{tikzpicture}[centerzero]
            \draw[vec] (-0.4,-0.4) -- (0.2,0.2);
            \draw[spin] (0,-0.4) to[out=45,in=-90] (0.2,0.2) -- (0.4,0.4);
            \draw[spin] (0.4,-0.4) to[out=150,in=-60] (-0.4,0.4);
        \end{tikzpicture}
         \ ,\qquad
          \begin{tikzpicture}[centerzero]
            \draw[vec] (-0.4,-0.4) -- (-0.1,-0.1);
            \draw[spin] (0,-0.4) -- (-0.1,-0.1) -- (0.4,0.4);
            \draw[vec] (0.4,-0.4) to[out=110,in=-20] (-0.4,0.4);
        \end{tikzpicture}
        =
        \begin{tikzpicture}[centerzero]
            \draw[vec] (-0.4,-0.4) -- (0.2,0.2);
            \draw[spin] (0,-0.4) to[out=45,in=-90] (0.2,0.2) -- (0.4,0.4);
            \draw[vec] (0.4,-0.4) to[out=150,in=-60] (-0.4,0.4);
        \end{tikzpicture}
        \\ \label{swishy}
        \begin{tikzpicture}[anchorbase]
            \draw[spin] (0,0) -- (0,0.1) arc (0:180:0.2) -- (-0.4,-0.6);
            \draw[spin] (0,0) to[out=-45, in=left] (0.2,-0.2) arc(-90:0:0.15) -- (0.35,0.4);
            \draw[vec] (0,0) to[out=-135,in=up] (-0.2,-0.3) to[out=down,in=down] (0.6,-0.3) -- (0.6,0.4);
        \end{tikzpicture}
        =
        \begin{tikzpicture}[anchorbase,xscale=-1]
            \draw[spin] (0,0) -- (0,0.1) arc (0:180:0.2) -- (-0.4,-0.6);
            \draw[vec] (0,0) to[out=-45, in=left] (0.2,-0.2) arc(-90:0:0.15) -- (0.35,0.4);
            \draw[spin] (0,0) to[out=-135,in=up] (-0.2,-0.3) to[out=down,in=down] (0.6,-0.3) -- (0.6,0.4);
        \end{tikzpicture}
        \ ,
        \\ \label{fishy}
        \begin{tikzpicture}[anchorbase]
            \draw[vec] (0.2,-0.5) to [out=135,in=down] (-0.15,-0.2) to[out=up,in=-135] (0,0);
            \draw[spin] (-0.2,-0.5) to[out=45,in=down] (0.15,-0.2) to[out=up,in=-45] (0,0) -- (0,0.2);
        \end{tikzpicture}
        = \kappa\
        \begin{tikzpicture}[anchorbase]
            \draw[spin] (0,0) -- (0,0.1) arc (0:180:0.2) -- (-0.4,-0.4);
            \draw[spin] (0,0) to[out=-45, in=left] (0.2,-0.2) arc(-90:0:0.15) -- (0.35,0.4);
            \draw[vec] (0,0) to[out=-135,in=up] (-0.2,-0.4);
        \end{tikzpicture}
        \ ,
        \\ \label{oist}
        \begin{tikzpicture}[centerzero]
            \draw[vec] (-0.4,-0.4) -- (0,0.1);
            \draw[vec] (0,-0.4) -- (0.2,-0.15);
            \draw[spin] (0.4,-0.4) --(0,0.1) -- (0,0.4);
        \end{tikzpicture}
        +
        \begin{tikzpicture}[centerzero]
            \draw[vec] (-0.4,-0.4) -- (0.2,-0.15);
            \draw[vec] (0,-0.4) to[out=135,in=-135] (0,0.1);
            \draw[spin] (0.4,-0.4) -- (0,0.1) -- (0,0.4);
        \end{tikzpicture}
        = 2\
        \begin{tikzpicture}[centerzero]
            \draw[vec] (-0.4,-0.4) -- (-0.4,-0.3) arc(180:0:0.2) -- (0,-0.4);
            \draw[spin] (0.4,-0.4) \braidup (0,0.4);
        \end{tikzpicture}
        \ ,
        \\ \label{dimrel}
        \bubble{vec} = d 1_\one,\qquad
        \bubble{spin} = D 1_\one.
    \end{gather}
    This concludes the definition of $\SB(d,D;\kappa)$.
\end{defin}

The third and fourth relations in \cref{brauer} imply that $\SB(d,D;\kappa)$ is a rigid monoidal category, with the objects $\Sgo$ and $\Vgo$ being self-dual.  The first, second, and sixth relations in \cref{brauer}, together with \cref{typhoon} imply that $\SB(d,D;\kappa)$ is symmetric monoidal, with symmetry given by the crossings.  Then \cref{swishy} implies that $\SB(d,D;\kappa)$ is strict pivotal, with duality given by rotating diagrams through $180\degree$.  This means that diagrams are isotopy invariant, and so rotated versions of all the defining relations hold.  For example, we have
\[
    \begin{tikzpicture}[centerzero]
        \draw[any] (-0.2,0.3) -- (-0.2,0.1) arc(180:360:0.2) -- (0.2,0.3);
        \draw[any] (-0.3,-0.3) to[out=up,in=down] (0,0.3);
    \end{tikzpicture}
    =
    \begin{tikzpicture}[centerzero]
        \draw[any] (-0.2,0.3) -- (-0.2,0.1) arc(180:360:0.2) -- (0.2,0.3);
        \draw[any] (0.3,-0.3) to[out=up,in=down] (0,0.3);
    \end{tikzpicture}
    \ ,\qquad
    \begin{tikzpicture}[anchorbase]
        \draw[any] (-0.15,0.4) to[out=-45,in=up] (0.15,0) arc(0:-180:0.15) to[out=up,in=-135] (0.15,0.4);
    \end{tikzpicture}
    = \cupmor{any}
    \ ,\qquad
    \begin{tikzpicture}[anchorbase]
        \draw[any] (-0.4,0.2) to[out=down,in=225,looseness=2] (0,0) to[out=45,in=up,looseness=2] (0.4,-0.2);
        \draw[any] (-0.2,0.2) -- (0.2,-0.2);
    \end{tikzpicture}
    =
    \crossmor{any}{any}
    =
    \begin{tikzpicture}[anchorbase]
        \draw[any] (0.2,0.2) -- (-0.2,-0.2);
        \draw[any] (0.4,0.2) to[out=down,in=-45,looseness=2] (0,0) to[out=135,in=up,looseness=2] (-0.4,-0.2);
    \end{tikzpicture}
    \ .
\]
Throughout this document, we will refer to a relation by its equation number even when we are, in fact, using a rotated version of that relation.

We introduce other trivalent morphisms by successive clockwise rotation:
\begin{equation} \label{vortex}
    \splitmor{spin}{vec}{spin}
    :=
    \begin{tikzpicture}[anchorbase]
        \draw[vec] (-0.4,0.2) to[out=down,in=180] (-0.2,-0.2) to[out=0,in=225] (0,0);
        \draw[spin] (0,0) -- (0,0.2);
        \draw[spin] (0.3,-0.3) -- (0,0);
    \end{tikzpicture}
    \ ,\qquad
    \mergemor{spin}{spin}{vec}
    :=
    \begin{tikzpicture}[anchorbase]
        \draw[spin] (0.4,-0.2) to[out=up,in=0] (0.2,0.2) to[out=180,in=45] (0,0);
        \draw[spin] (0,0) -- (0,-0.2);
        \draw[vec] (-0.3,0.3) -- (0,0);
    \end{tikzpicture}
    \ ,\qquad
    \splitmor{spin}{spin}{vec}
    :=
    \begin{tikzpicture}[anchorbase]
        \draw[spin] (-0.4,0.2) to[out=down,in=180] (-0.2,-0.2) to[out=0,in=225] (0,0);
        \draw[vec] (0,0) -- (0,0.2);
        \draw[spin] (0.3,-0.3) -- (0,0);
    \end{tikzpicture}
    \ ,\qquad
    \mergemor{spin}{vec}{spin}
    :=
    \begin{tikzpicture}[anchorbase]
        \draw[vec] (0.4,-0.2) to[out=up,in=0] (0.2,0.2) to[out=180,in=45] (0,0);
        \draw[spin] (0,0) -- (0,-0.2);
        \draw[spin] (-0.3,0.3) -- (0,0);
    \end{tikzpicture}
    \ ,\qquad
    \splitmor{vec}{spin}{spin}
    :=
    \begin{tikzpicture}[anchorbase]
        \draw[spin] (-0.4,0.2) to[out=down,in=180] (-0.2,-0.2) to[out=0,in=225] (0,0);
        \draw[spin] (0,0) -- (0,0.2);
        \draw[vec] (0.3,-0.3) -- (0,0);
    \end{tikzpicture}
    \ .
\end{equation}
Since $\SB(d,D;\kappa)$ is strict pivotal, the trivalent morphisms are also related in the natural way by counterclockwise rotation:
\begin{equation} \label{vortex2}
    \begin{tikzpicture}[anchorbase]
        \draw[spin] (0.4,0.2) to[out=down,in=0] (0.2,-0.2) to[out=180,in=-45] (0,0);
        \draw[spin] (0,0) -- (0,0.2);
        \draw[vec] (-0.3,-0.3) -- (0,0);
    \end{tikzpicture}
    =
    \splitmor{vec}{spin}{spin}
    \ ,\quad
    \begin{tikzpicture}[anchorbase]
        \draw[spin] (-0.4,-0.2) to[out=up,in=180] (-0.2,0.2) to[out=0,in=-225] (0,0);
        \draw[vec] (0,0) -- (0,-0.2);
        \draw[spin] (0.3,0.3) -- (0,0);
    \end{tikzpicture}
    =
    \mergemor{spin}{vec}{spin}
    \ ,\quad
    \begin{tikzpicture}[anchorbase]
        \draw[vec] (0.4,0.2) to[out=down,in=0] (0.2,-0.2) to[out=180,in=-45] (0,0);
        \draw[spin] (0,0) -- (0,0.2);
        \draw[spin] (-0.3,-0.3) -- (0,0);
    \end{tikzpicture}
    =
    \splitmor{spin}{spin}{vec}
    \ ,\quad
    \begin{tikzpicture}[anchorbase]
        \draw[spin] (-0.4,-0.2) to[out=up,in=180] (-0.2,0.2) to[out=0,in=-225] (0,0);
        \draw[spin] (0,0) -- (0,-0.2);
        \draw[vec] (0.3,0.3) -- (0,0);
    \end{tikzpicture}
    =
    \mergemor{spin}{spin}{vec}
    \ ,\quad
    \begin{tikzpicture}[anchorbase]
        \draw[spin] (0.4,0.2) to[out=down,in=0] (0.2,-0.2) to[out=180,in=-45] (0,0);
        \draw[vec] (0,0) -- (0,0.2);
        \draw[spin] (-0.3,-0.3) -- (0,0);
    \end{tikzpicture}
    =
    \splitmor{spin}{vec}{spin}
    \ ,\quad
    \begin{tikzpicture}[anchorbase]
        \draw[vec] (-0.4,-0.2) to[out=up,in=180] (-0.2,0.2) to[out=0,in=-225] (0,0);
        \draw[spin] (0,0) -- (0,-0.2);
        \draw[spin] (0.3,0.3) -- (0,0);
    \end{tikzpicture}
    =
    \mergemor{vec}{spin}{spin}
    \ .
\end{equation}

\begin{lem}
    We have
    \begin{equation} \label{lobster}
        \begin{tikzpicture}[anchorbase]
            \draw[vec] (0.2,-0.5) to [out=135,in=down] (-0.15,-0.2) to[out=up,in=-135] (0,0);
            \draw[spin] (-0.2,-0.5) to[out=45,in=down] (0.15,-0.2) to[out=up,in=-45] (0,0) -- (0,0.2);
        \end{tikzpicture}
        =
        \kappa\, \mergemor{spin}{vec}{spin}
        \ ,\qquad
        \begin{tikzpicture}[anchorbase]
            \draw[spin] (0.2,-0.5) to [out=135,in=down] (-0.15,-0.2) to[out=up,in=-135] (0,0) -- (0,0.2);
            \draw[vec] (-0.2,-0.5) to[out=45,in=down] (0.15,-0.2) to[out=up,in=-45] (0,0);
        \end{tikzpicture}
        =
        \kappa\, \mergemor{vec}{spin}{spin}
        \ ,\qquad
        \begin{tikzpicture}[anchorbase]
            \draw[spin] (0.2,-0.5) to [out=135,in=down] (-0.15,-0.2) to[out=up,in=-135] (0,0);
            \draw[spin] (-0.2,-0.5) to[out=45,in=down] (0.15,-0.2) to[out=up,in=-45] (0,0);
            \draw[vec] (0,0) -- (0,0.2);
        \end{tikzpicture}
        = \kappa\,
        \mergemor{spin}{spin}{vec}
        \ .
    \end{equation}
\end{lem}

\begin{proof}
    The first relation in \cref{lobster} is simply a rewriting of \cref{fishy}, using \cref{vortex2}.  Then, composing on the bottom of the first relation in \cref{lobster} with $\crossmor{vec}{spin}$ and using the first relation in \cref{brauer} gives the second relation in \cref{lobster}.  For the third relation in \cref{lobster}, we compute
    \[
        \begin{tikzpicture}[anchorbase]
            \draw[spin] (0.2,-0.5) to [out=135,in=down] (-0.15,-0.2) to[out=up,in=-135] (0,0);
            \draw[spin] (-0.2,-0.5) to[out=45,in=down] (0.15,-0.2) to[out=up,in=-45] (0,0);
            \draw[vec] (0,0) -- (0,0.2);
        \end{tikzpicture}
        \overset{\cref{typhoon}}{=}
        \begin{tikzpicture}[anchorbase]
            \draw[spin] (0,0) -- (-0.3,-0.5);
            \draw[vec] (0,0) -- (0,0.5);
            \draw[spin] (0,0) to[out=135,in=right] (-0.25,0.15) to[out=left,in=up] (-0.4,0) to[out=down,in=left] (-0.3,-0.15) to[out=right,in=left] (0,0.3) to[out=right,in=up] (0.3,-0.5);
        \end{tikzpicture}
        \overset{\cref{brauer}}{=}
        \begin{tikzpicture}[anchorbase]
            \draw[spin] (0,0) -- (-0.3,-0.5);
            \draw[vec] (0,0) -- (0,0.5);
            \draw[spin] (0,0) to[out=left,in=left,looseness=2] (0,0.3) to[out=right,in=up] (0.3,-0.5);
        \end{tikzpicture}
        \overset{\cref{fishy}}{=}
        \kappa\,
        \mergemor{spin}{spin}{vec}
        \ .
        \qedhere
    \]
\end{proof}

\begin{rem}
    The need for the choice of $\kappa \in \{\pm 1 \}$ in the definition of the spin Brauer category arises from that fact that, under the incarnation functor to be defined in \cref{sec:incarnation}, the objects $\Sgo$ and $\Vgo$ will be sent to the spin and vector representations, respectively, of $\Pin(N)$ or $\Spin(N)$.  For some values of $N$, the trivial representation and the vector representation either both live in the symmetric square or both live in the exterior square of the spin representation.  In this case, we can take $\kappa=1$.  However, for other values of $N$, one of the trivial or vector representations lives in the symmetric square while the other lives in the exterior square.  In this case, we need to take $\kappa=-1$.  See \cref{crazy,incarnation} for details.
\end{rem}

It will also sometimes be convenient to draw horizontal strands.  Since $\SB(d,D;\kappa)$ is strict pivotal, the meaning of diagrams containing such strands is unambiguous.  For example,
\begin{equation} \label{barbell}
    \begin{tikzpicture}[centerzero]
        \draw[spin] (-0.2,-0.3) -- (-0.2,0.3);
        \draw[spin] (0.2,-0.3) -- (0.2,0.3);
        \draw[vec] (-0.2,0) -- (0.2,0);
    \end{tikzpicture}
    \ =\
    \begin{tikzpicture}[centerzero]
        \draw[spin] (-0.2,-0.3) -- (-0.2,0.3);
        \draw[spin] (0.2,-0.3) -- (0.2,0.3);
        \draw[vec] (-0.2,0) to[out=-45,in=-135,looseness=1.8] (0.2,0);
    \end{tikzpicture}
    \ =\
    \begin{tikzpicture}[centerzero]
        \draw[spin] (-0.2,-0.3) -- (-0.2,0.3);
        \draw[spin] (0.2,-0.3) -- (0.2,0.3);
        \draw[vec] (-0.2,0) to[out=45,in=135,looseness=1.8] (0.2,0);
    \end{tikzpicture}
    \ =\
    \begin{tikzpicture}[centerzero]
        \draw[spin] (-0.2,-0.3) -- (-0.2,0.3);
        \draw[spin] (0.2,-0.3) -- (0.2,0.3);
        \draw[vec] (-0.2,0.15) -- (0.2,-0.15);
    \end{tikzpicture}
    \ =\
    \begin{tikzpicture}[centerzero]
        \draw[spin] (-0.2,-0.3) -- (-0.2,0.3);
        \draw[spin] (0.2,-0.3) -- (0.2,0.3);
        \draw[vec] (-0.2,-0.15) -- (0.2,0.15);
    \end{tikzpicture}
    \ .
\end{equation}

\begin{lem}
    We have
    \begin{equation} \label{bump}
        \begin{tikzpicture}[centerzero]
            \draw[spin] (0,-0.4) -- (0,0.4);
            \draw[vec] (0,-0.2) -- (-0.1,-0.2) arc(270:90:0.2) -- (0,0.2);
        \end{tikzpicture}
        = d\
        \begin{tikzpicture}[centerzero]
            \draw[spin] (0,-0.4) -- (0,0.4);
        \end{tikzpicture}
        \ .
    \end{equation}
\end{lem}

\begin{proof}
    We compute
    \[
        2\
        \begin{tikzpicture}[centerzero]
            \draw[spin] (0,-0.4) -- (0,0.4);
            \draw[vec] (0,-0.2) -- (-0.1,-0.2) arc(270:90:0.2) -- (0,0.2);
        \end{tikzpicture}
        \overset{\cref{brauer}}{=}
        \begin{tikzpicture}[centerzero]
            \draw[spin] (0,-0.4) -- (0,0.4);
            \draw[vec] (0,-0.2) -- (-0.1,-0.2) arc(270:90:0.2) -- (0,0.2);
        \end{tikzpicture}
        +
        \begin{tikzpicture}[centerzero]
            \draw[spin] (0,-0.4) -- (0,0.4);
            \draw[vec] (0,-0.2) to[out=left,in=right] (-0.4,0.2) arc(90:270:0.2) to[out=right,in=left] (0,0.2);
        \end{tikzpicture}
        \overset{\cref{oist}}{=} 2\
        \begin{tikzpicture}[centerzero]
            \bub{vec}{-0.4,0};
            \draw[spin] (0,-0.4) -- (0,0.4);
        \end{tikzpicture}
        \overset{\cref{dimrel}}{=}
        2d\
        \begin{tikzpicture}[centerzero]
            \draw[spin] (0,-0.4) -- (0,0.4);
        \end{tikzpicture}
        \ .
    \]
    Since we have assumed that $2$ is invertible in the ground ring $\kk$, the result follows.
\end{proof}

Let $\cC^\op$ denote the opposite of a category $\cC$, and let $\cC^\rev$ denote the reverse of a monoidal category $\cC$, where we reverse the order of the tensor product.  We have an isomorphism of monoidal categories
\begin{equation} \label{opflip}
    \SB(d,D;\kappa) \to \SB(d,D;\kappa)^\op
\end{equation}
that is the identity on objects and reflects morphisms in the horizontal axis.  We also have an isomorphism of monoidal categories
\begin{equation} \label{revflip}
    \SB(d,D;\kappa) \to \SB(d,D;\kappa)^\rev
\end{equation}
that is the identity on objects and reflects morphisms in the vertical axis.

\begin{lem}
    We have
    \begin{equation} \label{zombie}
        2 \kappa\
        \begin{tikzpicture}[centerzero]
            \draw[spin] (-0.35,0.35) -- (0.35,-0.35);
            \draw[vec] (-0.35,-0.35) -- (0.35,0.35);
        \end{tikzpicture}
        =
        \begin{tikzpicture}[centerzero]
            \draw[spin] (-0.35,0.35) -- (0,0.15) -- (0,-0.15) -- (0.35,-0.35);
            \draw[vec] (0,0.15) -- (0.35,0.35);
            \draw[vec] (0,-0.15) -- (-0.35,-0.35);
        \end{tikzpicture}
        +
        \begin{tikzpicture}[centerzero]
            \draw[spin] (-0.35,0.35) -- (-0.15,0) -- (0.15,0) -- (0.35,-0.35);
            \draw[vec] (-0.35,-0.35) -- (-0.15,0);
            \draw[vec] (0.35,0.35) -- (0.15,0);
        \end{tikzpicture}
        .
    \end{equation}
\end{lem}

\begin{proof}
    By \cref{oist}, we have
    \[
        2\
        \begin{tikzpicture}[centerzero]
            \draw[spin] (0,-0.4) -- (0,0.4);
            \draw[vec] (-0.3,-0.4) -- (-0.3,0.4);
        \end{tikzpicture}
        =
        \begin{tikzpicture}[centerzero]
            \draw[spin] (0.35,0.35) -- (0,0.15) -- (0,-0.15) -- (0.35,-0.35);
            \draw[vec] (0,0.15) -- (-0.35,0.35);
            \draw[vec] (0,-0.15) -- (-0.35,-0.35);
        \end{tikzpicture}
        +
        \begin{tikzpicture}[centerzero]
            \draw[spin] (0,-0.4) -- (0,0.4);
            \draw[vec] (-0.3,-0.4) -- (0,0.25);
            \draw[vec] (-0.3,0.4) -- (0,-0.25);
        \end{tikzpicture}
        \ .
    \]
    Then \cref{zombie} follows after composing on top with $\crossmor{vec}{spin}$ and using \cref{typhoon,brauer,lobster}.
\end{proof}

It will be convenient to introduce a shorthand for multiple strands:
\begin{equation} \label{multistrand}
    \begin{tikzpicture}[centerzero]
        \draw[vec] (0,-0.3) -- (0,0.3) node[midway,anchor=east] {\strandlabel{0}};
    \end{tikzpicture}
    := 1_\one
    \ ,\qquad
    \begin{tikzpicture}[centerzero]
        \draw[spin] (0,-0.3) -- (0,0.3) node[midway,anchor=east] {\strandlabel{0}};
    \end{tikzpicture}
    := 1_\one
    \ ,\qquad
    \begin{tikzpicture}[centerzero]
        \draw[vec] (0,-0.3) -- (0,0.3) node[midway,anchor=east] {\strandlabel{r}};
    \end{tikzpicture}
    :=
    \begin{tikzpicture}[centerzero]
        \draw[vec] (-0.3,-0.3) -- (-0.3,0.3);
        \draw[vec] (0.3,-0.3) -- (0.3,0.3);
        \node at (0.03,0) {$\cdots$};
        \draw[decorate, decoration={brace,mirror}] (-0.35,-0.35) -- (0.35,-0.35) node[midway,anchor=north] {\strandlabel{r}};
    \end{tikzpicture}
    \ ,\qquad
    \begin{tikzpicture}[centerzero]
        \draw[spin] (0,-0.3) -- (0,0.3) node[midway,anchor=east] {\strandlabel{r}};
    \end{tikzpicture}
    :=
    \begin{tikzpicture}[centerzero]
        \draw[spin] (-0.3,-0.3) -- (-0.3,0.3);
        \draw[spin] (0.3,-0.3) -- (0.3,0.3);
        \node at (0.03,0) {$\cdots$};
        \draw[decorate, decoration={brace,mirror}] (-0.35,-0.35) -- (0.35,-0.35) node[midway,anchor=north] {\strandlabel{r}};
    \end{tikzpicture}
    \ ,\quad r \ge 1.
\end{equation}
The first two relations in \cref{brauer} imply that we can interpret any element of the symmetric group $\fS_r$ on $r$ strands as a morphism in $\End_{\SB(d,D;\kappa)}(\Vgo^{\otimes r})$.  For a permutation $g \in \fS_r$, let $\sgn(g)$ denote its sign.  Then define the \emph{antisymmetrizer}
\begin{equation} \label{altbox}
    \begin{tikzpicture}[centerzero]
        \draw[vec] (0,-0.6) -- (0,0.6);
        \altbox{-0.5,-0.15}{0.5,0.15}{r};
    \end{tikzpicture}
    :=
    \begin{tikzpicture}[centerzero]
        \draw[vec] (0,-0.6) -- (0,-0.15) node[midway,anchor=west] {\strandlabel{r}};
        \draw[vec] (0,0.15) -- (0,0.6) node[midway,anchor=west] {\strandlabel{r}};
        \altbox{-0.5,-0.15}{0.5,0.15}{r};
    \end{tikzpicture}
    :=
    \sum_{g \in \fS_r} \sgn(g)\
    \begin{tikzpicture}[centerzero]
        \draw[vec] (0,-0.6) -- (0,-0.15) node[midway,anchor=west] {\strandlabel{r}};
        \draw[vec] (0,0.15) -- (0,0.6) node[midway,anchor=west] {\strandlabel{r}};
        \genbox{-0.5,-0.15}{0.5,0.15}{g};
    \end{tikzpicture}
    \ ,\quad r \ge 0,
\end{equation}
where we label the strands by $r$ when we wish to emphasize how many there are.  Thus, for example,
\[
    \begin{tikzpicture}[centerzero]
        \draw[vec] (0,-0.5) -- (0,0.5);
        \altbox{-0.25,-0.15}{0.25,0.15}{3};
    \end{tikzpicture}
    \ =\
    \begin{tikzpicture}[centerzero]
        \draw[vec] (-0.15,-0.5) -- (-0.15,0.5);
        \draw[vec] (0,-0.5) -- (0,0.5);
        \draw[vec] (0.15,-0.5) -- (0.15,0.5);
        \altbox{-0.25,-0.15}{0.25,0.15}{3};
    \end{tikzpicture}
    \ =\
    \begin{tikzpicture}[centerzero]
        \draw[vec] (-0.4,-0.4) -- (-0.4,0.4);
        \draw[vec] (0,-0.4) -- (0,0.4);
        \draw[vec] (0.4,-0.4) -- (0.4,0.4);
    \end{tikzpicture}
    \ -\
    \begin{tikzpicture}[centerzero]
        \draw[vec] (-0.4,-0.4) -- (0,0.4);
        \draw[vec] (0,-0.4) -- (-0.4,0.4);
        \draw[vec] (0.4,-0.4) -- (0.4,0.4);
    \end{tikzpicture}
    \ -\
    \begin{tikzpicture}[centerzero]
        \draw[vec] (-0.4,-0.4) -- (-0.4,0.4);
        \draw[vec] (0,-0.4) -- (0.4,0.4);
        \draw[vec] (0.4,-0.4) -- (0,0.4);
    \end{tikzpicture}
    \ +\
    \begin{tikzpicture}[centerzero]
        \draw[vec] (-0.4,-0.4) -- (0,0.4);
        \draw[vec] (0,-0.4) -- (0.4,0.4);
        \draw[vec] (0.4,-0.4) -- (-0.4,0.4);
    \end{tikzpicture}
    \ +\
    \begin{tikzpicture}[centerzero]
        \draw[vec] (-0.4,-0.4) -- (0.4,0.4);
        \draw[vec] (0,-0.4) -- (-0.4,0.4);
        \draw[vec] (0.4,-0.4) -- (0,0.4);
    \end{tikzpicture}
    \ -\
    \begin{tikzpicture}[centerzero]
        \draw[vec] (0.4,-0.4) -- (-0.4,0.4);
        \draw[vec] (0,-0.4) to[out=135,in=down] (-0.32,0) to[out=up,in=225] (0,0.4);
        \draw[vec] (-0.4,-0.4) -- (0.4,0.4);
    \end{tikzpicture}
    \ .
\]
It follows from \cref{multistrand} that the antisymmetrizer \cref{altbox} is $1_\one$ when $r=0$.  It also follows directly from the definition that
\begin{equation} \label{absorb}
    \begin{tikzpicture}[anchorbase]
        \draw[vec] (0,-1.2) -- (0,-0.8) node[midway,anchor=west] {\strandlabel{r}};
        \draw[vec] (-0.3,-0.5) -- (-0.3,0) node[midway,anchor=east] {\strandlabel{s}};
        \draw[vec] (0.3,-0.5) -- (0.3,0) node[midway,anchor=west] {\strandlabel{r-s-2}};
        \draw[vec] (-0.1,-0.5) -- (0.1,0);
        \draw[vec] (0.1,-0.5) -- (-0.1,0);
        \altbox{-0.5,-0.8}{0.5,-0.5}{r};
    \end{tikzpicture}
    = -\
    \begin{tikzpicture}[centerzero]
        \draw[vec] (0,-0.6) -- (0,0.6);
        \altbox{-0.5,-0.15}{0.5,0.15}{r};
    \end{tikzpicture}
    \qquad \text{for all } 0 \le s \le r-2.
\end{equation}

\begin{prop} \label{Delprop}
    Suppose that $r$ is a positive integer such that either
    \begin{itemize}
        \item $r$ is even and invertible, or
        \item $r$ is odd and $r-d$ is invertible.
    \end{itemize}
    Then
    \begin{equation} \label{Deligne}
        \begin{tikzpicture}[anchorbase]
            \draw[spin] (-0.1,0) -- (0.1,0) arc(-90:90:0.15) -- (-0.1,0.3) arc(90:270:0.15);
            \draw[vec] (0,-1) -- (0,0);
            \altbox{-0.2,-0.65}{0.2,-0.35}{r};
        \end{tikzpicture}
        = 0.
    \end{equation}
\end{prop}

\begin{proof}
    We have
    \begin{multline*}
        d\
        \begin{tikzpicture}[anchorbase]
            \draw[vec] (0,-1.2) -- (0,-0.8) node[midway,anchor=east] {\strandlabel{r}};
            \draw[vec] (0,-0.5) -- (0,0) node[midway,anchor=east] {\strandlabel{r}};
            \draw[spin] (-0.4,0) -- (0.4,0) arc(-90:90:0.15) -- (-0.4,0.3) arc(90:270:0.15);
            \altbox{-0.5,-0.8}{0.5,-0.5}{r};
        \end{tikzpicture}
        \overset{\cref{bump}}{=}
        \begin{tikzpicture}[anchorbase]
            \draw[vec] (0,-1.2) -- (0,-0.8) node[midway,anchor=east] {\strandlabel{r}};
            \draw[vec] (0,-0.5) -- (0,0) node[midway,anchor=east] {\strandlabel{r}};
            \draw[vec] (0.2,0) -- (0.2,0.3);
            \draw[spin] (-0.4,0) -- (0.4,0) arc(-90:90:0.15) -- (-0.4,0.3) arc(90:270:0.15);
            \altbox{-0.5,-0.8}{0.5,-0.5}{r};
        \end{tikzpicture}
        \overset{\cref{zombie}}{=} -
        \begin{tikzpicture}[anchorbase]
            \draw[vec] (0,-1.2) -- (0,-0.8) node[midway,anchor=east] {\strandlabel{r}};
            \draw[vec] (-0.2,-0.5) -- (-0.2,0) node[midway,anchor=east] {\strandlabel{r-1}};
            \draw[vec] (0.2,-0.5) -- (0.2,0);
            \draw[spin] (-0.4,0) -- (0.4,0) arc(-90:90:0.15) -- (-0.4,0.3) arc(90:270:0.15);
            \altbox{-0.5,-0.8}{0.5,-0.5}{r};
            \draw[vec] (0,0) -- (0,0.3);
        \end{tikzpicture}
        + 2\kappa
        \begin{tikzpicture}[anchorbase]
            \draw[vec] (0,-1.2) -- (0,-0.8) node[midway,anchor=east] {\strandlabel{r}};
            \draw[vec] (-0.2,-0.5) -- (-0.2,0) node[midway,anchor=east] {\strandlabel{r-1}};
            \draw[vec] (0.2,-0.5) -- (0.2,0.3);
            \draw[spin] (-0.4,0) -- (0.4,0) arc(-90:90:0.15) -- (-0.4,0.3) arc(90:270:0.15);
            \altbox{-0.5,-0.8}{0.5,-0.5}{r};
        \end{tikzpicture}
        \\
        \overset{\cref{zombie}}{=}
        \begin{tikzpicture}[anchorbase]
            \draw[vec] (0,-1.2) -- (0,-0.8) node[midway,anchor=east] {\strandlabel{r}};
            \draw[vec] (-0.2,-0.5) -- (-0.2,0) node[midway,anchor=east] {\strandlabel{r-2}};
            \draw[vec] (0.2,-0.5) -- (0.2,0) node[midway,anchor=west] {\strandlabel{2}};
            \draw[spin] (-0.4,0) -- (0.4,0) arc(-90:90:0.15) -- (-0.4,0.3) arc(90:270:0.15);
            \altbox{-0.5,-0.8}{0.5,-0.5}{r};
            \draw[vec] (0,0) -- (0,0.3);
        \end{tikzpicture}
        - 2\kappa
        \begin{tikzpicture}[anchorbase]
            \draw[vec] (0,-1.2) -- (0,-0.8) node[midway,anchor=east] {\strandlabel{r}};
            \draw[vec] (-0.2,-0.5) -- (-0.2,0) node[midway,anchor=east] {\strandlabel{r-2}};
            \draw[vec] (0,-0.5) -- (0,0.3);
            \draw[vec] (0.2,-0.5) -- (0.2,0);
            \draw[spin] (-0.4,0) -- (0.4,0) arc(-90:90:0.15) -- (-0.4,0.3) arc(90:270:0.15);
            \altbox{-0.5,-0.8}{0.5,-0.5}{r};
        \end{tikzpicture}
        + 2\kappa
        \begin{tikzpicture}[anchorbase]
            \draw[vec] (0,-1.2) -- (0,-0.8) node[midway,anchor=east] {\strandlabel{r}};
            \draw[vec] (-0.2,-0.5) -- (-0.2,0) node[midway,anchor=east] {\strandlabel{r-1}};
            \draw[vec] (0.2,-0.5) -- (0.2,0.3);
            \draw[spin] (-0.4,0) -- (0.4,0) arc(-90:90:0.15) -- (-0.4,0.3) arc(90:270:0.15);
            \altbox{-0.5,-0.8}{0.5,-0.5}{r};
        \end{tikzpicture}
        = \dotsb = (-1)^r\
        \begin{tikzpicture}[anchorbase]
            \draw[vec] (0,-1.2) -- (0,-0.8) node[midway,anchor=east] {\strandlabel{r}};
            \draw[vec] (0,-0.5) -- (0,0) node[midway,anchor=west] {\strandlabel{r}};
            \draw[spin] (-0.4,0) -- (0.4,0) arc(-90:90:0.15) -- (-0.4,0.3) arc(90:270:0.15);
            \altbox{-0.5,-0.8}{0.5,-0.5}{r};
            \draw[vec] (-0.2,0) -- (-0.2,0.3);
        \end{tikzpicture}
        + 2\kappa \sum_{s=0}^{r-1} (-1)^s\
        \begin{tikzpicture}[anchorbase]
            \draw[vec] (0,-1.2) -- (0,-0.8) node[midway,anchor=east] {\strandlabel{r}};
            \draw[vec] (-0.2,-0.5) -- (-0.2,0) node[midway,anchor=east] {\strandlabel{r-s-1}};
            \draw[vec] (0,-0.5) -- (0,0.3);
            \draw[vec] (0.2,-0.5) -- (0.2,0) node[midway,anchor=west] {\strandlabel{s}};
            \draw[spin] (-0.4,0) -- (0.4,0) arc(-90:90:0.15) -- (-0.4,0.3) arc(90:270:0.15);
            \altbox{-0.5,-0.8}{0.5,-0.5}{r};
        \end{tikzpicture}
        \\
        \overset{\cref{absorb}}{\underset{\cref{typhoon}}{=}} (-1)^r
        \begin{tikzpicture}[anchorbase]
            \draw[vec] (0,-1.2) -- (0,-0.8) node[midway,anchor=east] {\strandlabel{r}};
            \draw[vec] (0,-0.5) -- (0,0) node[midway,anchor=west] {\strandlabel{r}};
            \draw[spin] (-0.4,0) -- (0.4,0) arc(-90:90:0.15) -- (-0.4,0.3) arc(90:270:0.15);
            \altbox{-0.5,-0.8}{0.5,-0.5}{r};
            \draw[vec] (-0.2,0) -- (-0.2,0.3);
        \end{tikzpicture}
        + 2\kappa \sum_{s=0}^{r-1}\
        \begin{tikzpicture}[anchorbase]
            \draw[vec] (0,-1.2) -- (0,-0.8) node[midway,anchor=east] {\strandlabel{r}};
            \draw[vec] (-0.2,-0.5) -- (-0.2,0) node[midway,anchor=east] {\strandlabel{r-1}};
            \draw[vec] (0.2,-0.5) -- (0.2,0.3);
            \draw[spin] (-0.4,0) -- (0.4,0) arc(-90:90:0.15) -- (-0.4,0.3) arc(90:270:0.15);
            \altbox{-0.5,-0.8}{0.5,-0.5}{r};
        \end{tikzpicture}
        \overset{\cref{bump}}{\underset{\cref{lobster}}{=}}
        \left( 2 r + (-1)^r d \right)\
        \begin{tikzpicture}[anchorbase]
            \draw[vec] (0,-1.2) -- (0,-0.8) node[midway,anchor=east] {\strandlabel{r}};
            \draw[vec] (0,-0.5) -- (0,0) node[midway,anchor=east] {\strandlabel{r}};
            \draw[spin] (-0.4,0) -- (0.4,0) arc(-90:90:0.15) -- (-0.4,0.3) arc(90:270:0.15);
            \altbox{-0.5,-0.8}{0.5,-0.5}{r};
        \end{tikzpicture}
        \ .
    \end{multline*}
    Thus
    \[
        \big( 2r + ((-1)^r-1)d \big)
        \begin{tikzpicture}[anchorbase]
            \draw[vec] (0,-1.2) -- (0,-0.8) node[midway,anchor=east] {\strandlabel{r}};
            \draw[vec] (0,-0.5) -- (0,0) node[midway,anchor=east] {\strandlabel{r}};
            \draw[spin] (-0.4,0) -- (0.4,0) arc(-90:90:0.15) -- (-0.4,0.3) arc(90:270:0.15);
            \altbox{-0.5,-0.8}{0.5,-0.5}{r};
        \end{tikzpicture}
        \ = 0.
    \]
    If $r$ is even, then the coefficient on the left-hand side above is $2r$ and so \cref{Deligne} follows when $r$ is invertible.  If $r$ is odd, then the coefficient is $2(r-d)$ and \cref{Deligne} follows as long as $r - d$ is invertible.
\end{proof}

\begin{rem} \label{sofa}
    Note that the case $r=d \in 2 \N +1$ is not covered by \cref{Delprop}.  In fact, the diagram in \cref{Deligne} is \emph{not} zero in this case; see \cref{Delignefail}.
\end{rem}

\begin{defin}
    For $d \in 2\N + 1$, let $\overline{\SB}(d,D;\kappa)$ denote the quotient of $\SB(d,D;\kappa)$ by the relation
    \begin{equation} \label{extra}
        \begin{tikzpicture}[centerzero]
            \draw[spin] (0,0.25) circle(0.15);
            \draw[spin] (0,-0.25) circle(0.15);
            \draw[vec] (0,0.4) -- (0,1.1);
            \draw[vec] (0,-0.4) -- (0,-1.1);
            \altbox{-0.2,-0.9}{0.2,-0.6}{d};
            \altbox{-0.2,0.6}{0.2,0.9}{d};
        \end{tikzpicture}
        = D^2 (d!)^2\
        \begin{tikzpicture}[centerzero]
            \draw[vec] (0,-1.1) -- (0,1.1);
            \altbox{-0.2,-0.15}{0.2,0.15}{d};
        \end{tikzpicture}
        \ .
    \end{equation}
    For $d \notin 2\N+1$, let $\overline{\SB}(d,D;\kappa) = \SB(d,D;\kappa)$.
\end{defin}

The next result gives sufficient conditions under which all closed diagrams in $\overline{\SB}(d,D;\kappa)$ can be reduced to a multiple of the empty diagram $1_\one$.

\begin{prop} \label{popping}
   Suppose that $\kk$ is a $\mathbb{Q}$-algebra and $r-d$ is invertible for all $r \in (2\N + 1) \setminus \{d\}$. (For instance, this is satisfied when $\kk$ is a field of characteristic zero.)  Then $\End_{\overline{\SB}(d,D;\kappa)}(\one) = \kk 1_\one$.
\end{prop}

\begin{proof}
    We give an algorithm to simplify any diagram in $\End_{\overline{\SB}(d,D;\kappa)}(\one)$ to a scalar multiple of the identity. Throughout, we freely use our observations about isotopy invariance of diagrams in $\End_{\SB(d,D;\kappa)}(\one)$, as discussed immediately after the definition of $\SB(d,D;\kappa)$.  We proceed by induction on the number of trivalent vertices in the diagram.

    Suppose we have a closed diagram with at least one trivalent vertex. Consider the black curve that is part of that trivalent vertex. Since every trivalent vertex has exactly two black strings incident to it, this curve is part of a loop. Our first goal is to remove all self-intersections of this loop, and make the interior of this loop empty. We can separate all other black strands from this loop and remove self-intersections of this loop using \cref{brauer} and \cref{typhoon}. We separate all other dotted blue strands that do not have any trivalent vertices on this loop in the same manner. We can then use the same techniques, in addition to \cref{fishy}, to ensure the interior of this loop is empty. Let $r$ be the number of trivalent vertices on this loop.

    Unless $r=d$ and $d$ is an odd number, we have
    \[
        0
        \overset{\cref{Deligne}}{=}
        \begin{tikzpicture}[anchorbase]
            \draw[spin] (-0.1,0) -- (0.1,0) arc(-90:90:0.15) -- (-0.1,0.3) arc(90:270:0.15);
            \draw[vec] (0,-1) -- (0,0);
            \altbox{-0.2,-0.65}{0.2,-0.35}{r};
        \end{tikzpicture}
        \overset{\cref{oist}}{=} r!
        \begin{tikzpicture}[anchorbase]
            \draw[spin] (-0.1,0) -- (0.1,0) arc(-90:90:0.15) -- (-0.1,0.3) arc(90:270:0.15);
            \draw[vec] (0,-1) -- (0,0) node[midway,anchor=west] {\strandlabel{r}};
        \end{tikzpicture}
        + A,
    \]
    where $A$ is a linear combination of diagrams with fewer than $r$ dotted blue strands attached to the black circle.  Since $\kk$ is a $\Q$-algebra, $r!$ is invertible in $\kk$. We can then use this relation to write our diagram as a linear combination of diagrams with fewer trivalent vertices, as is our inductive goal.

    Suppose instead that $r=d$ is an odd positive integer. Since the total number of trivalent vertices is even, there must be another black loop with a trivalent vertex. We can repeat the process discussed above with that loop, and can either rewrite in terms of diagrams with fewer trivalent vertices, or that other loop also has $r$ trivalent vertices, in which case we end up with a subdiagram of the form
    \[
        \begin{tikzpicture}[anchorbase]
            \draw[spin] (-0.1,0) -- (0.1,0) arc(-90:90:0.15) -- (-0.1,0.3) arc(90:270:0.15);
            \draw[vec] (0,-0.5) -- (0,0) node[midway,anchor=west] {\strandlabel{r}};
        \end{tikzpicture}
        \begin{tikzpicture}[anchorbase]
            \draw[spin] (-0.1,0) -- (0.1,0) arc(-90:90:0.15) -- (-0.1,0.3) arc(90:270:0.15);
            \draw[vec] (0,-0.5) -- (0,0) node[midway,anchor=west] {\strandlabel{r}};
        \end{tikzpicture}
        \ .
    \]
    Then we have
    \[
        D^2 (r!)^2\
        \begin{tikzpicture}[anchorbase]
            \draw[vec] (-0.2,0) -- (-0.2,1) arc(180:0:0.2) -- (0.2,0);
            \altbox{-0.4,0.3}{0,0.6}{r};
        \end{tikzpicture}
        \overset{\cref{extra}}{=}
        \begin{tikzpicture}[anchorbase]
            \draw[spin] (-0.1,0) -- (0.1,0) arc(-90:90:0.15) -- (-0.1,0.3) arc(90:270:0.15);
            \draw[vec] (0,-1) -- (0,0);
            \altbox{-0.2,-0.65}{0.2,-0.35}{r};
        \end{tikzpicture}
        \
        \begin{tikzpicture}[anchorbase]
            \draw[spin] (-0.1,0) -- (0.1,0) arc(-90:90:0.15) -- (-0.1,0.3) arc(90:270:0.15);
            \draw[vec] (0,-1) -- (0,0);
            \altbox{-0.2,-0.65}{0.2,-0.35}{r};
        \end{tikzpicture}
        \overset{\cref{oist}}{=} (r!)^2
        \begin{tikzpicture}[anchorbase]
            \draw[spin] (-0.1,0) -- (0.1,0) arc(-90:90:0.15) -- (-0.1,0.3) arc(90:270:0.15);
            \draw[vec] (0,-1) -- (0,0) node[midway,anchor=west] {\strandlabel{r}};
        \end{tikzpicture}
        \begin{tikzpicture}[anchorbase]
            \draw[spin] (-0.1,0) -- (0.1,0) arc(-90:90:0.15) -- (-0.1,0.3) arc(90:270:0.15);
            \draw[vec] (0,-1) -- (0,0) node[midway,anchor=west] {\strandlabel{r}};
        \end{tikzpicture}
        + A,
    \]
    where again $A$ is a linear combination of diagrams with fewer trivalent vertices, and we proceed as before.

    This completes the inductive step and reduces us to considering the case where there are zero trivalent vertices. In this case, the relations \cref{brauer} (which are the same as in the Brauer category) suffice to rewrite our diagrams as a disjoint union of circles, which are evaluated as scalars by \cref{dimrel}.
\end{proof}

\begin{rem}
    \begin{enumerate}[wide]
        \item Note that \cref{popping} does not imply that $\End_{\overline{\SB}(d,D;\kappa)}(\one)$ is a free $\kk$-module of rank one.  Rather, it states that it is spanned by $1_\one$.  A priori, this endomorphism algebra could be a quotient of $\kk$.  However, see \cref{lunch} for conditions that insure it is free of rank one.

        \item When $d$ is an odd positive integer, the authors believe that \cref{popping} is false when $\overline{\SB}(d,D;\kappa)$ is replaced by $\SB(d,D;\kappa)$, because of the condition on $r$ in \cref{Delprop} (see \cref{sofa}).  For instance, if $d=3$, the authors do not now how to reduce the diagrams
            \[
                \begin{tikzpicture}[centerzero]
                    \draw[vec] (-0.5,0.5) -- (0.5,0.5);
                    \draw[vec] (-0.5,0.5) -- (-0.5,-0.5);
                    \draw[vec] (-0.5,0.5) -- (0.5,-0.5);
                    \draw[vec] (0.5,0.5) -- (-0.5,-0.5);
                    \draw[vec] (-0.5,-0.5) --(0.5,-0.5);
                    \draw[vec] (0.5,0.5) -- (0.5,-0.5);
                    \filldraw[spin,fill=white] (-0.5,0.5) circle(0.2);
                    \filldraw[spin,fill=white] (0.5,0.5) circle(0.2);
                    \filldraw[spin,fill=white] (-0.5,-0.5) circle(0.2);
                    \filldraw[spin,fill=white] (0.5,-0.5) circle(0.2);
                \end{tikzpicture}
                \qquad \text{or} \qquad
                \begin{tikzpicture}[centerzero]
                    \draw[vec] (-0.5,0.6) -- (0.5,0.6);
                    \draw[vec] (-0.5,0.4) -- (0.5,0.4);
                    \draw[vec] (-0.5,0.5) -- (-0.5,-0.5);
                    \draw[vec] (-0.5,-0.4) --(0.5,-0.4);
                    \draw[vec] (-0.5,-0.6) -- (0.5,-0.6);
                    \draw[vec] (0.5,0.5) -- (0.5,-0.5);
                    \filldraw[spin,fill=white] (-0.5,0.5) circle(0.2);
                    \filldraw[spin,fill=white] (0.5,0.5) circle(0.2);
                    \filldraw[spin,fill=white] (-0.5,-0.5) circle(0.2);
                    \filldraw[spin,fill=white] (0.5,-0.5) circle(0.2);
                \end{tikzpicture}
            \]
            to a scalar multiple of the empty diagram $1_\one$ without the additional relation \cref{extra}.
    \end{enumerate}
\end{rem}

\begin{lem}
    We have
    \begin{equation} \label{grapes}
        \begin{tikzpicture}[centerzero]
            \draw[spin] (-0.4,-0.4) -- (-0.4,0.4);
            \draw[spin] (0,-0.4) -- (0,0.4);
            \draw[spin] (0.4,-0.4) -- (0.4,0.4);
            \draw[vec] (-0.4,0.2) -- (0,0.2);
            \draw[vec] (-0.4,0) -- (0,0);
            \draw[vec] (0,-0.2) -- (0.4,-0.2);
        \end{tikzpicture}
        \ + 2\
        \begin{tikzpicture}[centerzero]
            \draw[spin] (-0.4,-0.4) -- (-0.4,0.4);
            \draw[spin] (0,-0.4) -- (0,0.4);
            \draw[spin] (0.4,-0.4) -- (0.4,0.4);
            \draw[vec] (-0.4,0.2) -- (0,0.2);
            \draw[vec] (-0.4,-0.2) -- (0,-0.2);
            \draw[vec] (0,0) -- (0.4,0);
        \end{tikzpicture}
        \ +\
        \begin{tikzpicture}[centerzero]
            \draw[spin] (-0.4,-0.4) -- (-0.4,0.4);
            \draw[spin] (0,-0.4) -- (0,0.4);
            \draw[spin] (0.4,-0.4) -- (0.4,0.4);
            \draw[vec] (-0.4,-0.2) -- (0,-0.2);
            \draw[vec] (-0.4,0) -- (0,0);
            \draw[vec] (0,0.2) -- (0.4,0.2);
        \end{tikzpicture}
        \ = 4\
        \begin{tikzpicture}[centerzero]
            \draw[spin] (-0.4,-0.4) -- (-0.4,0.4);
            \draw[spin] (0,-0.4) -- (0,0.4);
            \draw[spin] (0.4,-0.4) -- (0.4,0.4);
            \draw[vec] (0,0) -- (0.4,0);
        \end{tikzpicture}
        \ .
    \end{equation}
\end{lem}

\begin{proof}
    We have
    \[
        \begin{tikzpicture}[centerzero]
            \draw[spin] (-0.4,-0.4) -- (-0.4,0.4);
            \draw[spin] (0,-0.4) -- (0,0.4);
            \draw[spin] (0.4,-0.4) -- (0.4,0.4);
            \draw[vec] (-0.4,0.2) -- (0,0.2);
            \draw[vec] (-0.4,0) -- (0,0);
            \draw[vec] (0,-0.2) -- (0.4,-0.2);
        \end{tikzpicture}
        \ + 2\
        \begin{tikzpicture}[centerzero]
            \draw[spin] (-0.4,-0.4) -- (-0.4,0.4);
            \draw[spin] (0,-0.4) -- (0,0.4);
            \draw[spin] (0.4,-0.4) -- (0.4,0.4);
            \draw[vec] (-0.4,0.2) -- (0,0.2);
            \draw[vec] (-0.4,-0.2) -- (0,-0.2);
            \draw[vec] (0,0) -- (0.4,0);
        \end{tikzpicture}
        \ +\
        \begin{tikzpicture}[centerzero]
            \draw[spin] (-0.4,-0.4) -- (-0.4,0.4);
            \draw[spin] (0,-0.4) -- (0,0.4);
            \draw[spin] (0.4,-0.4) -- (0.4,0.4);
            \draw[vec] (-0.4,-0.2) -- (0,-0.2);
            \draw[vec] (-0.4,0) -- (0,0);
            \draw[vec] (0,0.2) -- (0.4,0.2);
        \end{tikzpicture}
        \overset{\cref{zombie}}{=}
        2\kappa\
        \begin{tikzpicture}[centerzero]
            \draw[spin] (-0.4,-0.4) -- (-0.4,0.4);
            \draw[spin] (0,-0.4) -- (0,0.4);
            \draw[spin] (0.4,-0.4) -- (0.4,0.4);
            \draw[vec] (-0.4,0.15) -- (0,0.15);
            \draw[vec] (-0.4,-0.15) -- (0.4,-0.15);
        \end{tikzpicture}
        \ +2\kappa\
        \begin{tikzpicture}[centerzero]
            \draw[spin] (-0.4,-0.4) -- (-0.4,0.4);
            \draw[spin] (0,-0.4) -- (0,0.4);
            \draw[spin] (0.4,-0.4) -- (0.4,0.4);
            \draw[vec] (-0.4,-0.15) -- (0,-0.15);
            \draw[vec] (-0.4,0.15) -- (0.4,0.15);
        \end{tikzpicture}
        \overset{\substack{\cref{typhoon} \\ \cref{oist}}}{\underset{\cref{lobster}}{=}}
        4\
        \begin{tikzpicture}[centerzero]
            \draw[spin] (-0.4,-0.4) -- (-0.4,0.4);
            \draw[spin] (0,-0.4) -- (0,0.4);
            \draw[spin] (0.4,-0.4) -- (0.4,0.4);
            \draw[vec] (0,0) -- (0.4,0);
        \end{tikzpicture}
        \ . \qedhere
    \]
\end{proof}

\begin{rem}
    The image of the relation \cref{grapes} under the incarnation functor to be defined in \cref{incarnation} corresponds to \cite[Lem.~1.3]{Wen20}, which plays a key role in the arguments of that paper.  Note that our $\beta$, defined in \cref{baton}, is equal to $2C$, where $C$ is defined in \cite[\S1.4]{Wen20}.
\end{rem}

We conclude this section with two lemmas that will be needed in the sequel.

\begin{lem}
    We have
    \begin{equation} \label{monkey2}
        \begin{tikzpicture}[anchorbase]
            \draw[vec] (0,-0.15) arc(180:360:0.2) -- (0.4,0.15) arc(0:180:0.2);
            \altbox{-0.2,-0.15}{0.2,0.15}{r};
        \end{tikzpicture}
        = d(d-1) \dotsm (d-r+1) 1_\one,\qquad r \in \N,
    \end{equation}
    where we interpret the right-hand side as $1_\one$ when $r=0$.
\end{lem}

\begin{proof}
    We prove the result by induction on $r$.  The base case $r=0$ is immediate.  (The case $r=1$ is the first relation in \cref{dimrel}.)  For the inductive step, note that, by the standard decomposition of $S_{r+1}$ as a union of right $S_r$-cosets, we have
    \begin{equation} \label{hoff}
        \begin{tikzpicture}[centerzero]
            \draw[vec] (0,-0.5) -- (0,0.5);
            \altbox{-0.35,-0.15}{0.35,0.16}{r+1};
        \end{tikzpicture}
        = \sum_{i=0}^r (-1)^i\
        \begin{tikzpicture}[anchorbase]
            \draw[vec] (0,-0.5) -- (0,-0.3);
            \draw[vec] (-0.1,0) -- (-0.1,0.5) node[midway,anchor=east] {\strandlabel{r-i}};
            \draw[vec] (0.4,0.4) node[anchor=west] {\strandlabel{i}};
            \draw[vec] (0.1,0) \braidup (0.4,0.5);
            \draw[vec] (0.4,-0.5) -- (0.4,0) \braidup (0.1,0.5);
            \altbox{-0.2,-0.3}{0.2,0}{r};
        \end{tikzpicture}
        \ .
    \end{equation}
    Thus,
    \[
        \begin{tikzpicture}[centerzero]
            \draw[vec] (0,-0.15) -- (0,0.15) arc(180:0:0.3) -- (0.6,-0.15) arc(360:180:0.3);
            \altbox{-0.35,-0.15}{0.35,0.16}{r+1};
        \end{tikzpicture}
        \overset{\cref{hoff}}{\underset{\cref{brauer}}{=}}
        \begin{tikzpicture}[centerzero]
            \draw[vec] (0,0.15) to[out=up,in=up] (1,0.15) -- (1,-0.15) to[out=down,in=down] (0,-0.15);
            \draw[vec] (0.4,0) arc(180:-180:0.2);
            \altbox{-0.2,-0.15}{0.2,0.15}{r};
        \end{tikzpicture}
        + \sum_{i=1}^r (-1)^i
        \begin{tikzpicture}[anchorbase]
            \altbox{-0.4,-0.3}{0.4,0}{r};
            \draw[vec] (-0.2,0) -- (-0.2,0.6) node[midway,anchor=east] {\strandlabel{r-i}} to[out=up,in=up] (1,0.6) -- (1,-0.3) to[out=down,in=down] (-0.2,-0.3);
            \draw[vec] (0,0) to[out=up,in=up] (0.6,0) -- (0.6,-0.3) to[out=down,in=down] (0.2,-0.3);
            \node[vec] at (0.45,0.35) {\strandlabel{i-1}};
            \draw[vec] (0.2,0) \braidup (0,0.6) to[out=up,in=up] (0.8,0.6) -- (0.8,-0.3) to[out=down,in=down] (0,-0.3);
        \end{tikzpicture}
        \overset{\cref{absorb}}{=}
        (d-r)\
        \begin{tikzpicture}[anchorbase]
            \draw[vec] (0,-0.15) arc(180:360:0.2) -- (0.4,0.15) arc(0:180:0.2);
            \altbox{-0.2,-0.15}{0.2,0.15}{r};
        \end{tikzpicture}
        = d(d-1) \dotsm (d-r)
    \]
    by the inductive hypothesis.
\end{proof}

\begin{lem}
    We have
    \begin{equation} \label{monkey1}
        \begin{tikzpicture}[anchorbase]
            \draw[spin] (-0.1,0) -- (0.1,0) arc(-90:90:0.15) -- (-0.1,0.3) arc(90:270:0.15);
            \draw[vec] (0,0) -- (0,-0.65) arc(180:360:0.2) -- (0.4,0.3) arc(0:180:0.2);
            \altbox{-0.2,-0.65}{0.2,-0.35}{r};
        \end{tikzpicture}
        = r! D d (d-1) \dotsm (d-r+1) 1_\one,
        \qquad r \in \N,
    \end{equation}
    where we interpret the right-hand side as $D 1_\one$ when $r=0$.
\end{lem}

\begin{proof}
    We prove the result by induction on $r$.  The base case $r=0$ is precisely the second relation in \cref{dimrel}.  Now suppose the result holds for some $r \ge 0$.  Using \cref{hoff}, we have
    \begin{equation} \label{duck1}
        \begin{tikzpicture}[anchorbase]
            \draw[spin] (-0.1,0) -- (0.1,0) arc(-90:90:0.15) -- (-0.1,0.3) arc(90:270:0.15);
            \draw[vec] (0,0) -- (0,-0.65) arc(180:360:0.3) -- (0.6,0.3) arc(0:180:0.3);
            \altbox{-0.3,-0.65}{0.3,-0.34}{r+1};
        \end{tikzpicture}
        =
        \sum_{i=0}^r A_i,\qquad \text{where} \quad
        A_i
        = (-1)^i\
        \begin{tikzpicture}[anchorbase]
            \draw[vec] (0,-0.3) to[out=down,in=down] (1,-0.3) -- (1,0.8) to[out=up,in=up,looseness=1.5] (0,0.8);
            \draw[vec] (-0.1,0) -- (-0.1,0.5) node[midway,anchor=east] {\strandlabel{r-i}};
            \draw[vec] (0.3,0.36) node[anchor=west] {\strandlabel{i}};
            \draw[vec] (0.1,0) \braidup (0.4,0.5);
            \draw[vec] (0.1,0.5) \braiddown (0.4,0) arc(180:360:0.2) -- (0.8,0.8) arc(0:180:0.2);
            \altbox{-0.2,-0.3}{0.2,0}{r};
            \draw[spin] (-0.2,0.5) arc(270:90:0.15) -- (0.5,0.8) arc(90:-90:0.15) -- (-0.2,0.5);
        \end{tikzpicture}
        \ .
    \end{equation}
    Now,
    \begin{equation} \label{duck2}
        A_i
        \overset{\cref{oist}}{=}
        A_{i-1} + 2(-1)^i\
        \begin{tikzpicture}[anchorbase]
            \draw[vec] (-0.2,-0.3) -- (-0.2,0.3) node[midway,anchor=east] {\strandlabel{r-i}};
            \draw[vec] (0.2,-0.3) -- (0.2,0.3);
            \draw[vec] (0.2,0.2) node[anchor=west] {\strandlabel{i-1}};
            \draw[vec] (0,-0.3) \braidup (1.2,0) -- (1.2,0.6) to[out=up,in=up] (0.2,0.6);
            \draw[vec] (0,-0.6) to[out=down,in=down] (1.4,-0.6) -- (1.4,0.6) to[out=up,in=up,looseness=1.5] (-0.2,0.6);
            \altbox{-0.3,-0.6}{0.3,-0.3}{r};
            \draw[spin] (-0.2,0.3) arc(270:90:0.15) -- (0.2,0.6) arc(90:-90:0.15) -- (-0.2,0.3);
        \end{tikzpicture}
        \overset{\cref{absorb}}{=}
        A_{i-1} - 2\
        \begin{tikzpicture}[anchorbase]
            \draw[spin] (-0.1,0) -- (0.1,0) arc(-90:90:0.15) -- (-0.1,0.3) arc(90:270:0.15);
            \draw[vec] (0,0) -- (0,-0.65) arc(180:360:0.2) -- (0.4,0.3) arc(0:180:0.2);
            \altbox{-0.2,-0.65}{0.2,-0.35}{r};
        \end{tikzpicture}
        \ .
    \end{equation}
    Starting from \cref{duck1}, repeated use of \cref{duck2} gives
    \[
        \begin{tikzpicture}[anchorbase]
            \draw[spin] (-0.1,0) -- (0.1,0) arc(-90:90:0.15) -- (-0.1,0.3) arc(90:270:0.15);
            \draw[vec] (0,0) -- (0,-0.65) arc(180:360:0.3) -- (0.6,0.3) arc(0:180:0.3);
            \altbox{-0.3,-0.65}{0.3,-0.34}{r+1};
        \end{tikzpicture}
        = 2 A_{r-1} + \sum_{i=0}^{r-1} A_i - 2\
        \begin{tikzpicture}[anchorbase]
            \draw[spin] (-0.1,0) -- (0.1,0) arc(-90:90:0.15) -- (-0.1,0.3) arc(90:270:0.15);
            \draw[vec] (0,0) -- (0,-0.65) arc(180:360:0.2) -- (0.4,0.3) arc(0:180:0.2);
            \altbox{-0.2,-0.65}{0.2,-0.35}{r};
        \end{tikzpicture}
        = \dotsb
        = (r+1) A_0 - r(r+1)\
        \begin{tikzpicture}[anchorbase]
            \draw[spin] (-0.1,0) -- (0.1,0) arc(-90:90:0.15) -- (-0.1,0.3) arc(90:270:0.15);
            \draw[vec] (0,0) -- (0,-0.65) arc(180:360:0.2) -- (0.4,0.3) arc(0:180:0.2);
            \altbox{-0.2,-0.65}{0.2,-0.35}{r};
        \end{tikzpicture}
        \ .
    \]
    Since
    \[
        A_0 =
        \begin{tikzpicture}[anchorbase]
            \draw[vec] (0,-0.3) to[out=down,in=down] (1,-0.3) -- (1,0.8) to[out=up,in=up,looseness=1.5] (0,0.8);
            \draw[vec] (0,0) -- (0,0.5) node[midway,anchor=east] {\strandlabel{r}};
            \draw[vec] (0.4,0.5) arc(180:360:0.2) -- (0.8,0.8) arc(0:180:0.2);
            \altbox{-0.2,-0.3}{0.2,0}{r};
            \draw[spin] (-0.2,0.5) arc(270:90:0.15) -- (0.5,0.8) arc(90:-90:0.15) -- (-0.2,0.5);
        \end{tikzpicture}
        \overset{\cref{bump}}{=}
        d\
        \begin{tikzpicture}[anchorbase]
            \draw[spin] (-0.1,0) -- (0.1,0) arc(-90:90:0.15) -- (-0.1,0.3) arc(90:270:0.15);
            \draw[vec] (0,0) -- (0,-0.65) arc(180:360:0.2) -- (0.4,0.3) arc(0:180:0.2);
            \altbox{-0.2,-0.65}{0.2,-0.35}{r};
        \end{tikzpicture}
        \ ,
    \]
    it follows that
    \[
        \begin{tikzpicture}[anchorbase]
            \draw[spin] (-0.1,0) -- (0.1,0) arc(-90:90:0.15) -- (-0.1,0.3) arc(90:270:0.15);
            \draw[vec] (0,0) -- (0,-0.65) arc(180:360:0.3) -- (0.6,0.3) arc(0:180:0.3);
            \altbox{-0.3,-0.65}{0.3,-0.34}{r+1};
        \end{tikzpicture}
        = (d-r)(r+1)\
        \begin{tikzpicture}[anchorbase]
            \draw[spin] (-0.1,0) -- (0.1,0) arc(-90:90:0.15) -- (-0.1,0.3) arc(90:270:0.15);
            \draw[vec] (0,0) -- (0,-0.65) arc(180:360:0.2) -- (0.4,0.3) arc(0:180:0.2);
            \altbox{-0.2,-0.65}{0.2,-0.35}{r};
        \end{tikzpicture}
        = (r+1)! Dd(d-1) \dotsm (d-r) 1_\one
    \]
    by the inductive hypothesis.
\end{proof}

\section{The incarnation functor\label{sec:incarnation}}

In this section we relate the spin Brauer category to the representation theory of the spin and pin groups.  Throughout this section, we assume $\kk = \C$.

Fix an vector space $V$ of finite dimension $N$, equipped with a nondegenerate symmetric bilinear form $\Phi_V$, and let $n = \left\lfloor \frac{N}{2} \right\rfloor$. Recall the definition of $\Group(V)$ from \cref{gvdefn}, the spin $\Group(V)$-module $S$ and the vector $\Group(V)$-module $V$ from \cref{suave}, and the bilinear form $\Phi_S$ on $S$ defined in \cref{Sform}.  Let
\begin{gather} \label{crazy}
    \sigma_N
    := (-1)^{\binom{n}{2} + nN},
    \qquad
    \kappa_N := (-1)^{nN},
    \\
    \SB(V) := \SB(N,\sigma_N 2^n; \kappa_N),
    \qquad \overline{\SB}(V) = \overline{\SB}(N,\sigma_N 2^n; \kappa_N).
\end{gather}
(Recall that $\sigma_N$ is the sign appearing in \cref{formsym}, describing the symmetry of the form $\Phi_S$.)

Fix a basis $\bB_S$ of $S$, and let $\bB_S^\vee = \{x^\vee : x \in \bB_S\}$ denote the left dual basis with respect to $\Phi_S$, defined by
\[
    \Phi_S(x^\vee, y) = \delta_{x,y},\qquad x,y \in \bB_S.
\]
We fix a basis $\bB_V$ of $V$ and define the left dual basis $\bB_V^\vee = \{v^\vee : v \in V\}$ similarly.  Then we have $\Group(V)$-module homomorphisms
\begin{align}
    \Phi_S^\vee &\colon \kk \to S \otimes S,&
    \lambda &\mapsto \lambda \sum_{x \in \bB_S} x \otimes x^\vee,\quad
    \lambda \in \kk,
    \\
    \Phi_V^\vee &\colon \kk \to V \otimes V,&
    \lambda &\mapsto \lambda \sum_{v \in \bB_V} v \otimes v^\vee,\quad
    \lambda \in \kk.
\end{align}
These are independent of the choices of bases.

It follows from \cref{spiral} and the fact that the form $\Phi_V$ is symmetric that the left dual bases of $\bB_V^\vee$ and $\bB_S^\vee$ are given by
\begin{equation} \label{falafel}
    (v^\vee)^\vee = v,\quad v \in \bB_V,
    \qquad \text{and} \qquad
    (x^\vee)^\vee = \sigma_N x,\quad x \in \bB_S,
\end{equation}
respectively.

For any $\kk$-modules $U$ and $W$, we define the linear map
\begin{equation} \label{flip}
    \flip = \flip_{U,W} \colon U \otimes W \to W \otimes U,\qquad
    u \otimes w \mapsto w \otimes u,
\end{equation}
extended by linearity.  If $U$ and $W$ are $\Group(V)$-modules, then $\flip$ is a homomorphism of $\Group(V)$-modules.  We also let
\begin{equation} \label{triaction}
    \tau \colon V \otimes S \to S,\quad v \otimes x \mapsto vx,
\end{equation}
denote the homomorphism of $\Group(V)$-modules induced by multiplication in the Clifford algebra $\Cl(V)$; see \cref{snow}.
\details{
    This is a homomorphism of $\Group(V)$-modules since, for $g \in \Group(V)$, $v \in V$, and $x \in S$, we have
    \[
        \tau(g(v \otimes x))
        = \tau(gvg^{-1} \otimes gx)
        = gvx
        = g \tau(v \otimes x).
    \]
}

\begin{theo} \label{incarnation}
    There is a unique monoidal functor
    \[
        \bF \colon \SB(V) \to \Group(V)\md
    \]
    given on objects by $\Sgo \mapsto S$, $\Vgo \mapsto V$, and on morphisms by
    \begin{gather} \label{incarnate1}
        \capmor{spin} \mapsto \Phi_S,\qquad
        \capmor{vec} \mapsto \Phi_V,\qquad
        \mergemor{vec}{spin}{spin} \mapsto \tau,
        \\ \label{incarnate2}
        \crossmor{spin}{spin} \mapsto \sigma_N \flip_{S,S},\qquad
        \crossmor{spin}{vec} \mapsto \flip_{S,V},\qquad
        \crossmor{vec}{spin} \mapsto \flip_{V,S},\qquad
        \crossmor{vec}{vec} \mapsto \flip_{V,V}.
    \end{gather}
    Furthermore, we have
    \begin{equation} \label{incarnate3}
        \cupmor{spin} \mapsto \Phi_S^\vee,\qquad
        \cupmor{vec} \mapsto \Phi_V^\vee.
    \end{equation}
\end{theo}

We call $\bF$ the \emph{incarnation functor}.

\begin{proof}
    We first show that \cref{incarnate1,incarnate2,incarnate3} indeed yield a functor $\bF$.  We must show that $\bF$ respects the relations of \cref{SBdef}.  The fifth relation in \cref{brauer} follows from \cref{spiral} and the fact that $\Phi_V$ is symmetric.  The remaining relations in \cref{brauer} are straightforward.  Relation \cref{typhoon} is also straightforward.

    The image under $\bF$ of the left-hand side of \cref{swishy} is the map $S \to S \otimes V$ given by
    \begin{multline*}
        x \mapsto \sum_{\substack{v \in \bB_V \\ y \in \bB_S}} x \otimes v \otimes y \otimes y^\vee \otimes v^\vee
        \mapsto \sum_{\substack{v \in \bB_V \\ y \in \bB_S}} x \otimes vy \otimes y^\vee \otimes v^\vee
        \mapsto \sum_{\substack{v \in \bB_V \\ y \in \bB_S}} \Phi_S(x,vy) y^\vee \otimes v^\vee
        \\
        \overset{\cref{bounce}}{=} (-1)^{nN} \sum_{\substack{v \in \bB_V \\ y \in \bB_S}} \Phi_S(vx,y) y^\vee \otimes v^\vee
        = (-1)^{nN} \sum_{v \in \bB_V} vx \otimes v^\vee.
    \end{multline*}
    On the other hand, the image under $\bF$ of the diagram in the right-hand side of \cref{swishy} is the map given by
    \begin{multline*}
        x \mapsto \sum_{\substack{v \in \bB_V \\ y \in \bB_S}} y \otimes v \otimes v^\vee \otimes y^\vee \otimes x
        \mapsto \sum_{\substack{v \in \bB_V \\ y \in \bB_S}} y \otimes v \otimes v^\vee y^\vee \otimes x
        \mapsto \sum_{\substack{v \in \bB_V \\ y \in \bB_S}} \Phi_S(v^\vee y^\vee, x) y \otimes v
        \\
        \overset{\cref{bounce}}{=} (-1)^{nN} \sum_{\substack{v \in \bB_V \\ y \in \bB_S}} \Phi_S(y^\vee, v^\vee x) y \otimes v
        = (-1)^{nN} \sum_{v \in \bB_V} v^\vee x \otimes v
        \overset{\cref{falafel}}{=} (-1)^{nN} \sum_{v \in \bB_V} v x \otimes v^\vee.
    \end{multline*}
    Since $\kappa_N=(-1)^{nN}$, $\bF$ respects relation \cref{swishy}.

    The image under $\bF$ of the left-hand side of \cref{fishy} is the map $S \otimes V \mapsto S$ given by
    \[
        x \otimes v \mapsto v \otimes x
        \mapsto vx.
    \]
    On the other hand, the image under $\bF$ of the right-hand side of \cref{fishy} is the map given by
    \begin{multline*}
        x \otimes v
        \mapsto \sum_{y \in \bB_S} x \otimes v \otimes y \otimes y^\vee
        \mapsto \sum_{y \in \bB_S} x \otimes vy \otimes y^\vee
        \mapsto \sum_{y \in \bB_S} \Phi_S(x,vy) y^\vee
        \\
        \overset{\cref{bounce}}{=} (-1)^{nN} \sum_{y \in \bB_S} \Phi_S(vx,y) y^\vee
        = (-1)^{nN} vx.
    \end{multline*}
    Thus, $\bF$ respects \cref{fishy}.

    The image under $\bF$ of the left-hand side of relation \cref{oist} is the map $V \otimes V \otimes S \to S$ given by
    \[
        v \otimes w \otimes x
        \mapsto (vw+wv)x
        \overset{\cref{Clifford}}{=} 2 \Phi_V(v,w) x,
    \]
    which agrees with the image under $\bF$ of the right-hand side of \cref{oist}.

    For the first relation in \cref{dimrel}, we use the fact that $\Phi_V$ is symmetric and that $\dim_\kk(V) = N$ to compute
    \[
        \bF
        \left(
            \begin{tikzpicture}[centerzero]
                \draw[vec] (-0.2,0) arc(180:-180:0.2);
            \end{tikzpicture}
        \right)
        (1)
        = \sum_{v \in \bB_V} \Phi_V(v, v^\vee)
        = \sum_{v \in \bB_V} \Phi_V(v^\vee, v)
        = N.
    \]
    Finally, for the second relation in \cref{dimrel}, we use \cref{spiral} and the fact that $\dim_\kk(S) = 2^n$ to compute
    \[
        \bF
        \left(
            \begin{tikzpicture}[centerzero]
                \draw[spin] (-0.2,0) arc(180:-180:0.2);
            \end{tikzpicture}
        \right)
        (1)
        = \sum_{x \in \bB_S} \Phi_S(x, x^\vee)
        = \sum_{x \in \bB_S} \sigma_N \Phi_S(x^\vee, x)
        = \sigma_N 2^n.
    \]

    It remains to prove that, for any functor as in the first sentence of the theorem, we have \cref{incarnate3}.  Suppose that
    \[
        \bF(\cupmor{spin}) \colon 1 \mapsto \sum_{x,y \in \bB_S} a_{xy} x \otimes y,\qquad
        a_{xy} \in \kk.
    \]
    Then, for all $z \in \bB_S$, we have
    \[
        z = \bF
        \left(\
            \begin{tikzpicture}[centerzero]
                \draw[spin] (0,-0.4) -- (0,0.4);
            \end{tikzpicture}
        \ \right)
        (z)
        =
        \left(
            \begin{tikzpicture}[centerzero]
                \draw[spin] (-0.3,0.4) -- (-0.3,0) arc(180:360:0.15) arc(180:0:0.15) -- (0.3,-0.4);
            \end{tikzpicture}
        \right)
        (z)
        = \sum_{x,y \in \bB_S} a_{xy} \Phi_S(y,z) x
        = \sum_{x \in \bB_S} a_{xz} z.
    \]
    It follows that $a_{xz} = \delta_{xz}$ for all $x,z \in \bB_S$, and so $\bF(\cupmor{spin}) = \Phi_S^\vee$.  The proof that $\bF(\cupmor{vec}) = \Phi_V^\vee$ is analogous.
\end{proof}

\begin{cor} \label{lunch}
    Let $\kk_0=\Q[d,D][\frac{1}{d-1},\frac{1}{d-3},\frac{1}{d-5},\ldots]$, and suppose that $\kk$ is a commutative $\kk_0$-algebra.  (In particular, this holds when $\kk$ is a field of characteristic zero and $d \notin 2\N+1$.)  Then
    \[
        \End_{\SB(d,D;1)}(\one)\cong \kk.
    \]
\end{cor}

\begin{proof}
    It suffices to prove the result when $\kk=\kk_0$, since the general result then follows after extending scalars from $\kk_0$ to $\kk$.  By \cref{popping}, we have $\End_{\SB(d,D,1)}(\one)\cong \kk_0/I$ for some ideal $I$ of $\kk_0$. (This is where we use our assumption that $d-1, d-3, \dotsc$ are invertible.) Suppose there exists a nonzero element $f(d,D)\in I$. Then there exists a positive integer $n$ such that $f \left( 2n,(-1)^{\binom{n}{2}} 2^{2n} \right) \neq 0$.
    \details{
        As $n \to \infty$, one of the monomials appearing in $f$ will grow faster than all the others when evaluated at these points.
    }
    Viewing $\C$ as a $\kk_0$-module via the map $d \mapsto 2n$, $D \mapsto (-1)^{\binom{n}{2}} 2^{2n}$, we can extend scalars in $\SB(d,D;1)$ and we then have an incarnation functor
    \[
        \SB(d,D;1) \otimes_{\kk_0} \C \to \Group(V)\md.
    \]
    This functor sends $f 1_\one$ to a nonzero element of $\End_{\Group(V)}(\triv^0)$, which is a contradiction.
\end{proof}

Our next goal is to show that the incarnation functor factors through $\overline{\SB}(V)$.

\begin{lem} \label{cheese}
    When $N$ is an odd positive integer, we have
    \begin{equation} \label{eggs}
        \bF
        \left(
            \begin{tikzpicture}[anchorbase]
                \draw[spin] (-0.1,0) -- (0.1,0) arc(-90:90:0.15) -- (-0.1,0.3) arc(90:270:0.15);
                \draw[spin] (-0.1,-1.3) -- (0.1,-1.3) arc(-90:90:0.15) -- (-0.1,-1) arc(90:270:0.15);
                \draw[vec] (0,0) -- (0,-1);
                \altbox{-0.2,-0.65}{0.2,-0.35}{N};
            \end{tikzpicture}
        \right)
        = 2^{N-1}(N!)^2.
    \end{equation}
\end{lem}

\begin{proof}
    Let
    \[
        Y = \{-n,1-n,\dotsc,n\}.
    \]
    In what follows, we will use the fact that
    \begin{equation} \label{beaver}
        \Phi_S(x_I, x_I^\vee)
        \overset{\cref{formsym}}{=} (-1)^{\binom{n}{2} + n(2n+1)} \Phi_S(x_I^\vee, x_I)
        = (-1)^{\binom{n+1}{2}}
        \qquad \text{for all } I \subseteq [n].
    \end{equation}

    Recall the definition of $\psi_i$ for $i \le 0$ from \cref{spiky}.  The dual of the ordered basis $\{\psi_{-n},\psi_{1-n},\ldots,\psi_{n}\}$ of $V$ is $\{2\psi_{n},2\psi_{n-1},\dotsc,2\psi_{-n}\}$.  Therefore, by \cref{vortex,incarnate3},
    \[
        \bF
        \left(
            \splitmor{spin}{vec}{spin}
        \right)
        \colon x \mapsto 2 \sum_{i=-n}^n \psi_{-i} \otimes \psi_i x.
    \]
    Thus,
    \[
        \bF
        \left(
            \begin{tikzpicture}[anchorbase]
                \draw[spin] (-0.1,-1.3) -- (0.1,-1.3) arc(-90:90:0.15) -- (-0.1,-1) arc(90:270:0.15);
                \draw[vec] (0,-0.5) -- (0,-1) node[midway,anchor=west] {\strandlabel{N}};
            \end{tikzpicture}
        \right)
        =
        \bF
        \left(
            \begin{tikzpicture}[anchorbase,rotate=45]
                \draw[spin] (-0.1,-1.3) -- (0.1,-1.3) arc(-90:90:0.15) -- (-0.1,-1) arc(90:270:0.15);
                \draw[vec] (0,-0.5) node[anchor=west] {\strandlabel{N}} -- (0,-1);
            \end{tikzpicture}
        \right)
    \]
    is the map
    \begin{align*}
        1
        &\mapsto \sum_{I \subseteq [n]} x_I \otimes x_I^\vee
        \\
        &\mapsto 2^N \sum_{I \subseteq [n]} \sum_{i_{-n},i_{1-n},\dotsc,i_n=-n}^n \psi_{-i_n}\otimes \psi_{-i_{n-1}} \otimes \dotsb \otimes \psi_{-i_{-n}} \otimes \psi_{i_{-n}} \dotsm \psi_{i_{n-1}} \psi_{i_n} x_I \otimes x_I^\vee
        \\
        &\mapsto 2^N \sum_{I \subseteq [n]} \sum_{i_{-n},i_{1-n},\dotsc,i_n=-n}^n \Phi_S(\psi_{i_{-n}} \dotsm \psi_{i_{n-1}} \psi_{i_n} x_I, x_I^\vee) \psi_{-i_n}\otimes \psi_{-i_{n-1}} \otimes \dotsb \otimes \psi_{-i_{-n}}.
    \end{align*}
    Applying the antisymmetrizer
    \(
        \bF
        \left(
            \begin{tikzpicture}[anchorbase]
                \draw[vec] (0,-0.3) -- (0,0.3);
                \altbox{-0.2,-0.15}{0.2,0.15}{N};
            \end{tikzpicture}
        \right)
    \),
    which annihilates any terms for which the map $j \mapsto i_j$ is not some permutation $\varpi \in \fS_Y$, we obtain
    \begin{align*}
        2^N  &\sum_{\varpi,\varpi' \in \fS_Y} \sgn(\varpi') \sum_{I \subseteq [n]} \Phi_S(\psi_{\varpi(-n)} \dotsm \psi_{\varpi(n-1)} \psi_{\varpi(n)} x_I, x_I^\vee) \psi_{-\varpi \varpi'(n)} \otimes \psi_{-\varpi \varpi'(n-1)} \otimes \dotsb \otimes \psi_{-\varpi \varpi'(-n)}
        \\
        &\overset{\mathclap{\cref{dragon}}}{=}\ \, 2^{N-1/2} \varepsilon \sum_{\varpi,\varpi' \in \fS_Y} \sgn(\varpi\varpi') \Phi_S(x_{I_\varpi}, x_{I_\varpi}^\vee) \psi_{-\varpi\varpi'(n)} \otimes \psi_{-\varpi\varpi'(n-1)} \otimes \dotsb \otimes \psi_{-\varpi\varpi'(-n)}
        \\
        &\overset{\mathclap{\cref{beaver}}}{=}\ \, 2^{N-1/2} \varepsilon (-1)^{\binom{n+1}{2}} \sum_{\varpi,\varpi' \in \fS_Y} \sgn(\varpi\varpi') \psi_{-\varpi\varpi'(n)} \otimes \psi_{-\varpi\varpi'(n-1)} \otimes \dotsb \otimes \psi_{-\varpi\varpi'(-n)}
        \\
        &= 2^{N-1/2} \varepsilon (-1)^{\binom{n+1}{2}} N! \sum_{\varpi \in \fS_Y} \sgn(\varpi) \psi_{\varpi(-n)} \otimes \psi_{\varpi(1-n)} \otimes \dotsb \otimes \psi_{\varpi(n)},
    \end{align*}
    where, in the last equality, we re-indexed the summation, noting that the sign of the permutation $\varpi(-j) \mapsto -\varpi \varpi'(j)$ is $\sgn(\omega')$.

    We now apply
    \[
        \bF
        \left(
            \begin{tikzpicture}[anchorbase]
                \draw[spin] (-0.1,1.3) -- (0.1,1.3) arc(90:-90:0.15) -- (-0.1,1) arc(270:90:0.15);
                \draw[vec] (0,0.5) -- (0,1) node[midway,anchor=west] {\strandlabel{N}};
            \end{tikzpicture}
        \right)
        =
        \bF
        \left(
            \begin{tikzpicture}[anchorbase,rotate=-45]
                \draw[spin] (-0.1,1.3) -- (0.1,1.3) arc(90:-90:0.15) -- (-0.1,1) arc(270:90:0.15);
                \draw[vec] (0,0.5) node[anchor=south] {\strandlabel{N}} -- (0,1);
            \end{tikzpicture}
        \right)
    \]
    to obtain
    \begin{multline*}
        2^{N-1/2} \varepsilon (-1)^{\binom{n+1}{2}} N! \sum_{\varpi \in \fS_Y} \sum_{I \subseteq [n]} \sgn(\varpi) \Phi_S(\psi_{\varpi(-n)} \psi_{\varpi(1-n)} \psi_{\varpi(n)} x_I, x_I^\vee)
        \\
        \overset{\cref{dragon}}{=} 2^{N-1} (-1)^{\binom{n+1}{2}} N! \sum_{\varpi \in \fS_Y} \Phi_S(x_{I_\varpi}, x_{I_\varpi}^\vee)
        \overset{\cref{beaver}}{=} 2^{N-1} (N!)^2.
        \qedhere
    \end{multline*}
\end{proof}

\begin{theo} \label{beacon}
    The incarnation functor $\bF$ of \cref{incarnation} factors through $\overline{\SB}(V)$.
\end{theo}

\begin{proof}
    If $N$ is even, there is nothing to prove, since $\overline{\SB}(V) = \SB(V)$ in this case.  Therefore, we suppose that $N$ is odd.  The images under $\bF$ of the two sides of \cref{extra} are endomorphisms of $\End_{\Group(V)}(\Lambda^N(V))$, which is one dimensional.  Therefore, there exists a scalar $a \in \kk$ such \[
        \bF
        \left(
            \begin{tikzpicture}[centerzero]
                \draw[spin] (0,0.25) circle(0.15);
                \draw[spin] (0,-0.25) circle(0.15);
                \draw[vec] (0,0.4) -- (0,1.1);
                \draw[vec] (0,-0.4) -- (0,-1.1);
                \altbox{-0.2,-0.9}{0.2,-0.6}{N};
                \altbox{-0.2,0.6}{0.2,0.9}{N};
            \end{tikzpicture}
        \right)
        = a \bF
        \left(
            \begin{tikzpicture}[centerzero]
                \draw[vec] (0,-1.1) -- (0,1.1);
                \altbox{-0.2,-0.15}{0.2,0.15}{N};
            \end{tikzpicture}
        \right)
        \ .
    \]
    We then have
    \[
        2^{N-1} (N!)^2
        \overset{\cref{eggs}}{=}
        \bF
        \left(
            \begin{tikzpicture}[anchorbase]
                \draw[spin] (-0.1,0) -- (0.1,0) arc(-90:90:0.15) -- (-0.1,0.3) arc(90:270:0.15);
                \draw[spin] (-0.1,-1.3) -- (0.1,-1.3) arc(-90:90:0.15) -- (-0.1,-1) arc(90:270:0.15);
                \draw[vec] (0,0) -- (0,-1);
                \altbox{-0.2,-0.65}{0.2,-0.35}{N};
            \end{tikzpicture}
        \right)
        \overset{\cref{absorb}}{=} \frac{1}{N!} \bF
        \left(
            \begin{tikzpicture}[centerzero]
                \draw[spin] (0,0.25) circle(0.15);
                \draw[spin] (0,-0.25) circle(0.15);
                \draw[vec] (0,0.4) -- (0,0.9) arc (180:0:0.2) -- (0.4,-0.9) arc (360:180:0.2) -- (0,-0.4);
                \altbox{-0.2,-0.9}{0.2,-0.6}{N};
                \altbox{-0.2,0.6}{0.2,0.9}{N};
            \end{tikzpicture}
            \,
        \right)
        = \frac{a}{N!} \bF
        \left(
            \begin{tikzpicture}[anchorbase]
                \draw[vec] (0,-0.15) arc(180:360:0.2) -- (0.4,0.15) arc(0:180:0.2);
                \altbox{-0.2,-0.15}{0.2,0.15}{N};
            \end{tikzpicture}
        \right)
        \overset{\cref{monkey2}}{=}
        a.
    \]
    It follows that $a = 2^{N-1} (N!)^2$, as desired.
\end{proof}

\begin{lem}
    We have
    \begin{align} \label{black}
        \bF \left( \mergemor{spin}{vec}{spin} \right) &\colon S \otimes V \to S,&
        x \otimes v &\mapsto (-1)^{nN} vx,
        \\ \label{baction}
        \bF \left(
            \begin{tikzpicture}[centerzero]
                \draw[spin] (-0.2,-0.3) -- (-0.2,0.3);
                \draw[spin] (0.2,-0.3) -- (0.2,0.3);
                \draw[vec] (-0.2,0) -- (0.2,0);
            \end{tikzpicture}
        \right)
        &\colon S \otimes S \to S \otimes S,&
        x \otimes y &\mapsto (-1)^{nN} \beta(x \otimes y),
    \end{align}
    where
    \begin{equation} \label{baton}
        \beta := \sum_{i=1}^N e_i \otimes e_i \in \Cl^{\otimes 2}.
    \end{equation}
\end{lem}

\begin{proof}
    Using \cref{lobster}, we see that \cref{black} follows from the part of the proof of \cref{incarnation} where we verified \cref{fishy}.  Then we have
    \[
        \bF \left(
            \begin{tikzpicture}[centerzero]
                \draw[spin] (-0.2,-0.3) -- (-0.2,0.3);
                \draw[spin] (0.2,-0.3) -- (0.2,0.3);
                \draw[vec] (-0.2,0) -- (0.2,0);
            \end{tikzpicture}
        \right)
        \overset{\cref{barbell}}{=}
        \bF \left(
            \begin{tikzpicture}[centerzero]
                \draw[spin] (-0.2,-0.3) -- (-0.2,0.3);
                \draw[spin] (0.2,-0.3) -- (0.2,0.3);
                \draw[vec] (-0.2,0) to[out=-45,in=-135,looseness=1.8] (0.2,0);
            \end{tikzpicture}
        \right)
        \colon x \otimes y
        \mapsto (-1)^{nN} \sum_{i=1}^n e_i x \otimes e_i y
        = (-1)^{nN} \beta(x \otimes y).
        \qedhere
    \]
\end{proof}

\begin{rem} \label{vampire}
    There are other possible incarnation functors.  In particular, for $m,k \in \N$, let $\OSp(m|2k)$ be the corresponding orthosymplectic supergroup, defined to be the supergroup preserving a nondegenerate supersymmetric bilinear form $\Phi_W$ on a vector superspace $W$ whose even part has dimension $m$ and odd part has dimension $2k$.  Then there is a unique monoidal functor
    \[
        \SB(N,\sigma_N (m-2k) 2^n;\kappa_N) \to (\Group(V) \times \OSp(m|2k))\md
    \]
    given on objects by $\Sgo \mapsto S \otimes W$, $\Vgo \mapsto V$, and on morphisms by
    \begin{gather*}
        \capmor{spin} \mapsto \Phi_S \otimes \Phi_W,\qquad
        \capmor{vec} \mapsto \Phi_V,\qquad
        \mergemor{vec}{spin}{spin} \mapsto \tau \otimes \id_W,
        \\
        \crossmor{spin}{spin} \mapsto \sigma_N \flip_{S \otimes W, S \otimes W},\qquad
        \crossmor{spin}{vec} \mapsto \flip_{S \otimes W, V},\qquad
        \crossmor{vec}{spin} \mapsto \flip_{V, S \otimes W},\qquad
        \crossmor{vec}{vec} \mapsto \flip_{V,V},
    \end{gather*}
    where we now use the super analogue of the map $\flip$ of \cref{flip}, given by $u \otimes w \mapsto (-1)^{\bar{u} \bar{w}} w \otimes u$, where $\bar{v}$ is the parity of a homogeneous vector $v$.  The proof of the existence and uniqueness of this functor is similar to that of \cref{incarnation}, as is the proof that it factors through $\overline{\SB}(N,\sigma_N(m-2k)2^n; \kappa_N)$.
    \details{
        As an example of the proof of existence, we verify that this functor respects the relation \cref{fishy}.  The image under the functor of the left-hand side of \cref{fishy} is the map $S \otimes V \mapsto S$ given by
        \[
            x \otimes w \otimes v
            \mapsto v \otimes x \otimes w
            \mapsto vx \otimes w.
        \]
        On the other hand, the image under the functor of the right-hand side of the first relation in \cref{fishy} is the map given by
        \begin{multline*}
            x \otimes w \otimes v
            \mapsto \sum_{\substack{y \in \bB_S \\ u \in \bB_W}} x \otimes w \otimes v \otimes y \otimes u \otimes y^\vee \otimes u^\vee
            \mapsto \sum_{\substack{y \in \bB_S \\ u \in \bB_W}} x \otimes w \otimes vy \otimes u \otimes y^\vee \otimes u^\vee
            \\
            \mapsto \sum_{\substack{y \in \bB_S \\ u \in \bB_W}} \Phi_S(x,vy) \Phi_W(w,u) y^\vee \otimes u^\vee
            \overset{\cref{bounce}}{=} (-1)^{nN} \Phi_S(vx,y) y^\vee \otimes w
            = (-1)^{nN} \Phi(w,u) vx \otimes w.
        \end{multline*}
        Thus, $\bF$ respects \cref{fishy}.
    }
    \details{
        To show that the modified incarnation functors of \cref{vampire} also respect the relation \cref{extra} when $d$ is an odd integer, we compute the analogue of \cref{eggs} under these functors.

        Let $\bB_W$ be a basis of $W$, and let $w^\vee$, $w \in \bB_W$, denote the dual basis, defined by
        \[
            \Phi_W(w^\vee, u) = \delta_{wu},\qquad w,u \in \bB_W.
        \]
        Then we have
        \[
            \Phi_W(w,w^\vee) = (-1)^{\bar{w}} \Phi(w^\vee, w) = (-1)^{\bar{w}}.
        \]

        By \cref{vortex},
        \[
            \bF
            \left(
                \splitmor{spin}{vec}{spin}
            \right)
            \colon S \otimes W \to V \otimes S \otimes W,\qquad
            x \otimes w \mapsto 2 \sum_{i=-n}^n \psi_{-i} \otimes \psi_i x \otimes w.
        \]
        Thus,
        \[
            \bF
            \left(
                \begin{tikzpicture}[anchorbase]
                    \draw[spin] (-0.1,-1.3) -- (0.1,-1.3) arc(-90:90:0.15) -- (-0.1,-1) arc(90:270:0.15);
                    \draw[vec] (0,-0.5) -- (0,-1) node[midway,anchor=west] {\strandlabel{N}};
                \end{tikzpicture}
            \right)
            =
            \bF
            \left(
                \begin{tikzpicture}[anchorbase,rotate=45]
                    \draw[spin] (-0.1,-1.3) -- (0.1,-1.3) arc(-90:90:0.15) -- (-0.1,-1) arc(90:270:0.15);
                    \draw[vec] (0,-0.5) node[anchor=west] {\strandlabel{N}} -- (0,-1);
                \end{tikzpicture}
            \right)
        \]
        is the map
        \begin{align*}
            1
            &\mapsto \sum_{\substack{I \subseteq [n] \\ w \in \bB_W}} x_I \otimes w \otimes x_I^\vee \otimes w^\vee
            \\
            &\mapsto 2^N \sum_{\substack{I \subseteq [n] \\ w \in \bB_W}} \sum_{i_{-n},i_{1-n},\dotsc,i_n=-n}^n \psi_{-i_n}\otimes \psi_{-i_{n-1}} \otimes \dotsb \otimes \psi_{-i_{-n}} \otimes \psi_{i_{-n}} \dotsm \psi_{i_{n-1}} \psi_{i_n} x_I \otimes w \otimes x_I^\vee \otimes w^\vee
            \\
            &\mapsto 2^N \sum_{\substack{I \subseteq [n] \\ w \in \bB_W}} \sum_{i_{-n},i_{1-n},\dotsc,i_n=-n}^n \Phi_S(\psi_{i_{-n}} \dotsm \psi_{i_{n-1}} \psi_{i_n} x_I, x_I^\vee) \Phi_W(w,w^\vee) \psi_{-i_n}\otimes \psi_{-i_{n-1}} \otimes \dotsb \otimes \psi_{-i_{-n}}
            \\
            &= (m-2k) 2^N \sum_{\substack{I \subseteq [n] \\ w \in \bB_W}} \sum_{i_{-n},i_{1-n},\dotsc,i_n=-n}^n \Phi_S(\psi_{i_{-n}} \dotsm \psi_{i_{n-1}} \psi_{i_n} x_I, x_I^\vee) \psi_{-i_n}\otimes \psi_{-i_{n-1}} \otimes \dotsb \otimes \psi_{-i_{-n}}
        \end{align*}
        Applying the antisymmetrizer
        \(
            \bF
            \left(
                \begin{tikzpicture}[anchorbase]
                    \draw[vec] (0,-0.3) -- (0,0.3);
                    \altbox{-0.2,-0.15}{0.2,0.15}{N};
                \end{tikzpicture}
            \right)
        \),
        we obtain, as in the proof of \cref{cheese},
        \[
            (m-2k) 2^{N-1/2} \varepsilon (-1)^{\binom{n+1}{2}} N! \sum_{\varpi \in \fS_Y} \sgn(\varpi) \psi_{\varpi(-n)} \otimes \psi_{\varpi(1-n)} \otimes \dotsb \otimes \psi_{\varpi(n)}.
        \]
        We now apply
        \[
            \bF
            \left(
                \begin{tikzpicture}[anchorbase]
                    \draw[spin] (-0.1,1.3) -- (0.1,1.3) arc(90:-90:0.15) -- (-0.1,1) arc(270:90:0.15);
                    \draw[vec] (0,0.5) -- (0,1) node[midway,anchor=west] {\strandlabel{N}};
                \end{tikzpicture}
            \right)
            =
            \bF
            \left(
                \begin{tikzpicture}[anchorbase,rotate=-45]
                    \draw[spin] (-0.1,1.3) -- (0.1,1.3) arc(90:-90:0.15) -- (-0.1,1) arc(270:90:0.15);
                    \draw[vec] (0,0.5) node[anchor=south] {\strandlabel{N}} -- (0,1);
                \end{tikzpicture}
            \right)
        \]
        to obtain
        \begin{multline*}
            (m-2k) 2^{N-1/2} \varepsilon (-1)^{\binom{n+1}{2}} N! \sum_{\varpi \in \fS_Y} \sum_{\substack{I \subseteq [n] \\ w \in \bB_W}} \sgn(\varpi) \Phi_S(\psi_{\varpi(-n)} \psi_{\varpi(1-n)} \psi_{\varpi(n)} x_I, x_I^\vee) \Phi_W(w,w^\vee)
            \\
            = (m-2k)^2 2^{N-1} (-1)^{\binom{n+1}{2}} N! \sum_{\varpi \in \fS_Y} \Phi_S(x_{I_\varpi}, x_{I_\varpi}^\vee)
            \overset{\cref{beaver}}{=} (m-2k)^2 2^{N-1} (N!)^2.
            \qedhere
        \end{multline*}

        We can then repeat the argument of \cref{beacon} to show that
        \[
            \bF
            \left(
                \begin{tikzpicture}[centerzero]
                    \draw[spin] (0,0.25) circle(0.15);
                    \draw[spin] (0,-0.25) circle(0.15);
                    \draw[vec] (0,0.4) -- (0,1.1);
                    \draw[vec] (0,-0.4) -- (0,-1.1);
                    \altbox{-0.2,-0.9}{0.2,-0.6}{N};
                    \altbox{-0.2,0.6}{0.2,0.9}{N};
                \end{tikzpicture}
            \right)
            = (m-2k)^2 2^{N-1} (N!)^2 \bF
            \left(
                \begin{tikzpicture}[centerzero]
                    \draw[vec] (0,-1.1) -- (0,1.1);
                    \altbox{-0.2,-0.15}{0.2,0.15}{N};
                \end{tikzpicture}
            \right)
            \ .
        \]
    }
\end{rem}

\begin{cor} \label{carrot}
    Suppose that $\kk$ is a $\Q$-algebra and $d$ is an odd positive integer. Then
    \[
        \End_{\overline{\SB}(d,D,\kappa_d)}(\one)\cong \kk.
    \]
\end{cor}

\begin{proof}
    The proof is analogous to that of \cref{lunch}, using the extra incarnation functors of \cref{vampire}.
    \details{
        In the proof of \cref{lunch}, we already had a Zariski dense set of incarnation functors. Here, we have fixed where we specialise $d$, and so we need a Zariski-dense choice of parameters to which we can specialise $D$, with $d$ fixed.  This is why we need the extra incarnation functors of \cref{vampire}.
    }
\end{proof}

\begin{rem}
    When $\kk$ is a field of characteristic not equal to two, we have an incarnation functor from $\SB(V)$ to the category of tilting modules for the group $\Group(V)$, given in an analogous manner to \cref{incarnation}. This functor exists since our constructions can be carried out over $\Z[\frac{1}{2}]$, the defining and spin representations are tilting away from characteristic two, and the category of tilting modules is closed under tensor products. The restriction on the characteristic is necessary since the module $V$ is not tilting in characteristic two.  We expect that this incarnation functor is full.
\end{rem}

\section{Fullness of the incarnation functor\label{sec:full}}

In the current section, we prove that the incarnation functor of \cref{incarnation} is full.  Until the statement of \cref{surly}, we assume that $N \ge 2$.

Recall the element $\beta \in \Cl^{\otimes 2}$ from \cref{baton}.  We define a \emph{barbell} to be any element of the form $1^{\otimes t} \otimes \beta \otimes 1^{r-t-2} \in \Cl^{\otimes r}$, $0 \le t \le r-2$, $r \ge 2$.  The action of a barbell yields an element of $\End_{\Group(V)}(S^{\otimes r})$.

\begin{lem}\label{barbellgen}
    The action of the barbell $\beta$ generates $\End_{\Group(V)}(S\otimes S)$.
\end{lem}

\begin{proof}
    By \cref{Sdub,lemon}, the $\Group(V)$-module $S \otimes S$ is multiplicity free and, by \cite[Lem.~1.2]{Wen20}, the action of $\beta$ has a different eigenvalue on each summand. (There is a typo in \cite[Lem.~1.2(b)]{Wen20}; the index $j$ should run from $0$ to $k$ inclusive.)
\end{proof}

\begin{lem} \label{calculator}
    For all $k \ge 0$, the morphism
    \[
        \bF
        \left(
            \begin{tikzpicture}[centerzero]
                \draw[multispin] (-0.2,0.5) -- (-0.2,0.3) arc(180:360:0.2) -- (0.2,0.5);
                \draw (0.2,0.35) node[anchor=east] {\strandlabel{k}};
                \draw[multispin] (-0.2,-0.5) -- (-0.2,-0.3) arc(180:0:0.2) -- (0.2,-0.5);
                \draw (0.2,-0.35) node[anchor=east] {\strandlabel{k}};
            \end{tikzpicture}
        \right)
    \]
    lies in the subalgebra of $\End_{\Group(V)}(S^{\otimes 2k})$ generated by barbells, where the thick cup and cap labelled by $k$ denote $k$ nested cups and caps, respectively.
\end{lem}

\begin{proof}
    We have
    \[
        \bF
        \left(
            \begin{tikzpicture}[centerzero]
                \draw[multispin] (-0.2,0.5) -- (-0.2,0.3) arc(180:360:0.2) -- (0.2,0.5);
                \draw (0.2,0.35) node[anchor=east] {\strandlabel{k}};
                \draw[multispin] (-0.2,-0.5) -- (-0.2,-0.3) arc(180:0:0.2) -- (0.2,-0.5);
                \draw (0.2,-0.35) node[anchor=east] {\strandlabel{k}};
            \end{tikzpicture}
        \right)
        =
        \bF
        \left(\,
            \begin{tikzpicture}[centerzero]
                \draw[multispin] (-0.4,0.7) -- (-0.4,0.5) to[out=down,in=down] (0.4,0.5) -- (0.4,0.7);
                \node at (0,0.5) {\strandlabel{k-1}};
                \draw[spin] (-0.6,0.7) -- (-0.6,0.5) to[out=down,in=down,looseness=1.2] (0.6,0.5) -- (0.6,0.7);
                \draw[multispin] (-0.4,-0.7) -- (-0.4,-0.5) to[out=up,in=up] (0.4,-0.5) -- (0.4,-0.7);
                \draw[spin] (-0.6,-0.7) -- (-0.6,-0.5) to[out=up,in=up,looseness=1.2] (0.6,-0.5) -- (0.6,-0.7);
                \node at (0,-0.5) {\strandlabel{k-1}};
            \end{tikzpicture}
        \, \right)
        \ .
    \]
    By \cref{barbellgen}, the innermost $\bF(\hourglass)$ can be written as a polynomial in
    \[
        \bF
        \left(
            \begin{tikzpicture}[centerzero]
                \draw[spin] (-0.2,-0.3) -- (-0.2,0.3);
                \draw[spin] (0.2,-0.3) -- (0.2,0.3);
                \draw[vec] (-0.2,0) -- (0.2,0);
            \end{tikzpicture}
        \right).
    \]
    Next, note that
    \[
        \begin{tikzpicture}[centerzero]
            \draw[multispin] (-0.4,0.6) -- (-0.4,0.4) to[out=down,in=down] (0.4,0.4) -- (0.4,0.6);
            \node at (0,0.4) {\strandlabel{k-1}};
            \draw[multispin] (-0.4,-0.6) -- (-0.4,-0.4) to[out=up,in=up] (0.4,-0.4) -- (0.4,-0.6);
            \draw[spin] (-0.6,-0.6) -- (-0.6,0.6);
            \draw[spin] (0.6,-0.6) -- (0.6,0.6);
            \node at (0,-0.4) {\strandlabel{k-1}};
            \draw[vec] (-0.6,0) -- (0.6,0);
        \end{tikzpicture}
        =
        \begin{tikzpicture}[centerzero]
            \draw[multispin] (-0.4,0.6) -- (-0.4,0.2) to[out=down,in=down] (0.4,0.2) -- (0.4,0.6);
            \node at (0,0.2) {\strandlabel{k-1}};
            \draw[multispin] (-0.4,-0.6) -- (-0.4,-0.4) to[out=up,in=up] (0.4,-0.4) -- (0.4,-0.6);
            \draw[spin] (-0.6,-0.6) -- (-0.6,0.6);
            \draw[spin] (0.6,-0.6) -- (0.6,0.6);
            \node at (0,-0.4) {\strandlabel{k-1}};
            \draw[vec] (-0.6,0.4) -- (0.6,0.4);
        \end{tikzpicture}
        \ .
    \]
    A straightforward proof by induction, using \cref{zombie}, shows that
    \[
        \begin{tikzpicture}[centerzero]
            \draw[spin] (-1,-0.3) -- (-1,0.3);
            \draw[multispin] (-0.5,-0.3) \botlabel{k-1} -- (-0.5,0.3);
            \draw[multispin] (0.5,-0.3) \botlabel{k-1} -- (0.5,0.3);
            \draw[spin] (1,-0.3) -- (1,0.3);
            \draw[vec] (-1,0) -- (1,0);
        \end{tikzpicture}
    \]
    is in the image of the barbells.  Thus, the lemma follows by induction.
\end{proof}

\begin{lem} \label{summand}
    Let $r\in \N$, and let $\la=(\la_1,\la_2,\ldots,\la_n)$ be a dominant integral weight with $\la_1 = \frac{r}{2}$. Then $L(\la)$ is a direct summand of the $\Spin(V)$-module $S^{\otimes r}$.
\end{lem}

\begin{proof}
    We prove this by induction on $r$, the base case $r=0$ being trivial.  Suppose $r \geq 1$ and let $\la=(\la_1,\la_2,\ldots,\la_n)$ with $\la_1=\frac{r}{2}$.   Let $k$ be the largest index for which $\la_k> 0$.  Let $\epsilon = \left( \frac{1}{2}, \dotsc, \frac{1}{2}, -\frac{1}{2}, \dotsc, -\frac{1}{2} \right)$, where there are $k$ occurrences of $\frac{1}{2}$. It then follows directly from the characterisations \cref{dominantD,dominantB} of dominant integral weights that $\la-\epsilon$ is dominant integral.  By the inductive hypothesis, $L(\la-\epsilon)$ is a direct summand of $S^{\otimes (r-1)}$. \Cref{lem:pieri} implies that $L(\la)$ is a summand of $S \otimes L(\la-\epsilon)$, hence is a direct summand of $S^{\otimes r}$, as required.
\end{proof}

\begin{cor}\label{pinsummand}
    Suppose $N$ is even, and let $r$ be a positive integer.  Let $\la=(\la_1,\la_2,\dotsc,\la_n)$ be a dominant integral weight with $\la_1=\frac{r}{2}$, and let $M$ be a simple $\Pin(V)$-module whose restriction to $\Spin(V)$ contains $L(\la)$ as submodule.  Then $M$ is a summand of the $\Pin(V)$-module $S^{\otimes r}$.
\end{cor}

\begin{proof}
    This follows from \cref{summand,echo}.
\end{proof}

\begin{rem} \label{mitre}
    Note that the condition $r>0$ in \cref{pinsummand} is necessary since $\triv^1$ is not a summand of the trivial $\Pin(V)$-module $S^{\otimes 0}$.  However, when $N$ is even, $\triv^1$ is a summand of $S^{\otimes r}$ for $r \in 2\N$, $r > 0$, by \cref{echo}.
\end{rem}

Let $r \in \N$.  For $N > 2$, let $X_r$ be the $\Spin(V)$-submodule of $S^{\otimes r}$ that is the sum of all simple summands of highest weight $\lambda$ with $\lambda_1 = \frac{r}{2}$.  For $N=2$, let $X_r$ be the $\Spin(V)$-submodule of $S^{\otimes r}$ that is the sum of all simple summands of highest weight $\lambda_1 = \pm \frac{r}{2}$. It follows from \cref{hongmen,hydro} that $X_r$ is also $\Group(V)$-submodule.

Recall the definition of $A_{ij}$ from \cref{greenD,greenB}, and let $Y_r$ be the $\frac{r}{2}$-eigenspace of $A_{11}$ on $X_r$.  Since $\lambda_1 \le \frac{r}{2}$ for all weights $\lambda$ appearing as a highest weight in $S^{\otimes r}$, $Y_r$ is also the $\frac{r}{2}$-eigenspace of $A_{11}$ on $S^{\otimes r}$.  Recall the definition of the elements $x_I \in S$ from \cref{xIdef}.  It follows from \cref{bamboo,bambooB} that the space $Y_r$ is the span of all $x_{I_1} \otimes x_{I_2} \otimes \dotsb \otimes x_{I_r}$ where $1 \in I_i$ for all $i$.

Let
\[
    W =
    \begin{cases}
        \Span_\kk \left\{ \psi_2,\psi_3,\dotsc,\psi_n,\psi_n^\dagger,\dotsc,\psi_3^\dagger,\psi_2^\dagger \right\} \subseteq V & \text{if } N \in 2\N, \\
        \Span_\kk \left\{ \psi_2,\psi_3,\dotsc,\psi_n,e_N,\psi_n^\dagger,\dotsc,\psi_3^\dagger,\psi_2^\dagger \right\} \subseteq V & \text{if } N \in 2\N + 1,
    \end{cases}
\]
where we adopt the convention that $W = \{0\}$ when $N=2$.  We have a natural inclusion of groups $\Group(W) \subseteq \Group(V)$.  Since the actions of $A_{11}$ and $\Group(W)$ on $V$ commute, $Y_r$ is a $\Group(W)$-submodule of $V$.
\details{
    We justify the fact the actions of $A_{11}$ and $\Pin(W)$ commute.  Pick an orthonormal basis $v_1,\dotsc,v_N$ of $V$ such that $v_1$ and $v_N$ span the same subspace as $\psi_1$ and $\psi_1^\dagger$, and that $v_2,\dotsc,v_{N-1}$ span $W$.  Then any element of the 1-parameter subgroup generated by $A_{11}$ is of the form $a + b v_1 v_N$ for $a,b\in\C$, $a^2+b^2=1$. We compute directly that $v_2(a+bv_1v_N) = (a+bv_1v_N)v_2$. This proves the commutativity result for one element of $\Pin(W)$ not in the identity component; the Lie algebra computation implies the rest.
}
Let $S_W$ be the spin module for $\Group(W)$.  Then $x_I$, $I \subseteq \{2,3,\dotsc,n\}$, is a basis for $S_W$.  It is straightforward to verify that the map
\begin{equation} \label{beach}
    Y_r \to S_W^{\otimes r},\quad
    x_{I_1} \otimes \dotsb \otimes x_{I_r}
    \mapsto x_{I_1 \setminus \{1\}} \otimes \dotsb \otimes x_{I_r \setminus \{1\}},
\end{equation}
is an isomorphism of $\Group(W)$-modules.

\begin{lem}
    For every $f \in \End_{\Group(V)}(X_r)$, we have $f(Y_r) \subseteq Y_r$.  Furthermore, restriction to $Y_r$ yields an isomorphism of $\kk$-modules
    \begin{equation} \label{snafu}
        \End_{\Group(V)}(X_r) \xrightarrow{\cong} \End_{\Group(W)}(Y_r).
    \end{equation}
\end{lem}

\begin{proof}
    The first assertion follows from the fact that any element of $f \in \End_{\Group(V)}(X_r)$ commutes with the action of $A_{11}$, hence leaves the eigenspace $Y_r$ invariant.  Since $\Group(W)$ is a subgroup of $\Group(V)$, it follows that the restriction of $f$ lies in $\End_{\Group(W)}(Y_r)$. Thus, we have a homomorphism of $\kk$-modules
    \begin{equation} \label{pdub}
        \End_{\Group(V)}(X_r) \to \End_{\Group(W)}(Y_r).
    \end{equation}
    Our goal is to show that this linear map is an isomorphism.  The result is trivial if $N=2$, and so we assume $N>2$.

    Since $\Group(V)$-mod is a semisimple category, any element of $\End_{\Group(V)}(X_r)$ is determined by its action on highest-weight vectors.  An analogous statement holds for $\End_{\Group(W)}(Y_r)$.  Therefore, to see that \cref{pdub} is injective, it suffices to show that the space of highest-weight vectors of the $\Group(V)$-module $X_r$ is equal to the space of highest-weight vectors of the $\Group(W)$-module $Y_r$.

    First suppose that $v$ is a highest-weight vector in $X_r$ of weight $(\lambda_1,\dotsc,\lambda_n)$.  Then $v \in Y_r$, since $\lambda_1 = \frac{r}{2}$ by definition of $X_r$.  Since the inclusion $\Group(W) \subseteq \Group(V)$ respects upper triangularity, it follows that $v$ is a highest-weight vector of the $\Group(W)$-module $Y_r$.   Conversely, suppose $v$ is a highest-weight vector in the $\Group(W)$-module $Y_r$. Then, since $[A_{11},A_{12}] = A_{12}$, we have
    \[
        A_{11} A_{12} v
        = A_{12} A_{11} v + [A_{11},A_{12}] v
        = \left( \tfrac{r}{2}+1 \right) A_{12} v.
    \]
    Since there are no nonzero $w \in S^{\otimes r}$ satisfying $A_{11}w = (\frac{r}{2}+1)w$, we have $A_{12}(v)=0$. If $\n$ and $\n'$ are the subalgebras of strictly upper-triangular matrices in $\fso(V)$ and $\fso(W)$ respectively, then $\n$ is generated by $\n'$ and $A_{12}$. Hence $\n$ annihilates $v$, and so $v$ is a highest-weight vector in $X_r$.  This completes the argument that \cref{pdub} is injective.

    To finish the proof of the lemma, it suffices to show that the dimensions of $\End_{\Group(V)}(X_r)$ and $\End_{\Group(W)}(Y_r)$ are equal.  Since $\Group(V)$-mod and $\Group(W)$-mod are both semisimple categories, these endomorphism algebras are isomorphic to products of matrix algebras.  To see that their dimensions are equal, it is enough to show that the identification of highest-weight spaces given above preserves multiplicities of weights.  This, in turn, follows from the fact that two highest-weight vectors in $X_r$ have equal $\Group(V)$-weights if and only if they have equal $\Group(W)$-weights, since the highest weights in $X_r$ are constrained to have $\lambda_1 = \frac{r}{2}$.
\end{proof}

Since $X_r$ is a sum of isotypic components of $S^{\otimes r}$, any endomorphism of $S^{\otimes r}$ leaves $X_r$ invariant.  Therefore, restriction to $X_r$ yields a natural projection $\End_{\Group(V)}(S^{\otimes r}) \twoheadrightarrow \End_{\Group(V)}(X_r)$.  We thus have a surjective composite of $\kk$-module homomorphisms
\begin{equation} \label{okinawa}
    \Psi_r \colon
    \End_{\Group(V)}(S^{\otimes r}) \twoheadrightarrow \End_{\Group(V)}(X_r)
    \xrightarrow[\cref{snafu}]{\cong} \End_{\Group(W)}(Y_r)
    \xrightarrow{\cong} \End_{\Group(W)}(S_W^{\otimes r}),
\end{equation}
where the final isomorphism is induced by \cref{beach}.

\begin{lem} \label{couch}
    For $0 \le t \le r-2$, $r \ge 2$, we have
    \[
        \Psi_r \left( 1^{\otimes r} \otimes \beta \otimes 1^{\otimes (t-r-2)} \right)
        = 1^{\otimes r} \otimes \beta_W \otimes 1^{\otimes (t-r-2)},
    \]
    where $\beta_W = \sum_{i=3}^N e_i \otimes e_i$ is the barbell for $W$.  (By convention, $\beta_W = 0$ if $N=2$.)
\end{lem}

\begin{proof}
    Using \cref{tuna}, we compute that
    \begin{equation} \label{beqn}
        e_1 \otimes e_1 + e_2 \otimes e_2
        = 2 ( \psi_1 \otimes \psi_1^\dagger + \psi_1^\dagger \otimes \psi_1 ).
    \end{equation}
    Now suppose that $I_1,I_2, \dotsc, I_r \subseteq [n]$ satisfy $1 \in I_k$ for all $1 \le k \le r$.  Then
    \begin{multline*}
        \left( 1^{\otimes t} \otimes \psi_1 \otimes \psi_1^\dagger \otimes 1^{\otimes (r-t-2)} \right)
        (x_{I_1} \otimes x_{I_2} \otimes \dotsb \otimes x_{I_r})
        = 0
        \\
        = \left( 1^{\otimes t} \otimes \psi_1^\dagger \otimes \psi_1 \otimes 1^{\otimes (r-t-2)} \right)
        (x_{I_1} \otimes x_{I_2} \otimes \dotsb \otimes x_{I_r}).
    \end{multline*}
    Thus, the result follows from \cref{beqn}.
\end{proof}

For the remainder of this section, we drop the assumption that $N \ge 2$.

\begin{theo} \label{surly}
    Suppose $r,r_1,r_2 \in \N$.
    \begin{enumerate}
        \item \label{Sfull} The incarnation functor $\bF$ induces a surjection
            \[
                \Hom_{\SB(V)}(\Sgo^{\otimes r_1}, \Sgo^{\otimes r_2})
                \twoheadrightarrow \Hom_{\Group(V)}(S^{\otimes r_1}, S^{\otimes r_2}).
            \]

        \item \label{BBgen} The algebra $\End_{\Group(V)}(S^{\otimes r})$ is generated by barbells.
    \end{enumerate}
\end{theo}

\begin{proof}
    Since the components of the weights of $S^{\otimes r}$ lie in $\frac{r}{2} + \Z$, we have $\Hom_{\Group(V)}(S^{\otimes r_1}, S^{\otimes r_2}) = 0$ when $r_1+r_2 \notin 2\Z$.  Thus, for statement \cref{Sfull}, we will assume for the remainder of this proof that $r_1 + r_2 \in 2\Z$.

    We have a commutative diagram
    \begin{equation} \label{dogman}
        \begin{tikzcd}
            \Hom_{\SB(V)}(\Sgo^{\otimes r_1}, \Sgo^{\otimes r_2}) \arrow[d, "\bF"] \arrow[r, "\cong"]
            & \Hom_{\SB(V)}(\Sgo^{\otimes (r_1+r_2)}, \one) \arrow[d, "\bF"]
            \\
            \Hom_{\Group(V)} \big( S^{\otimes r_1}, S^{\otimes r_2} \big) \arrow[r, "\cong"]
            & \Hom_{\Group(V)} \big( S^{\otimes (r_1+r_2)}, \triv^0 \big)
        \end{tikzcd}
    \end{equation}
    where the horizontal maps are the usual isomorphisms that hold in any rigid monoidal category.  In particular, the top horizontal map is the $\kk$-linear isomorphism given on diagrams by
    \[
        \begin{tikzpicture}[centerzero]
            \draw (0,-0.2) rectangle (1,0.2);
            \draw[spin] (0.2,0.2) -- (0.2,0.7);
            \node at (0.52,0.45) {$\cdots$};
            \draw[spin] (0.8,0.2) -- (0.8,0.7);
            \draw[spin] (0.2,-0.2) -- (0.2,-0.7);
            \draw[spin] (0.8,-0.2) -- (0.8,-0.7);
            \node at (0.52,-0.45) {$\cdots$};
        \end{tikzpicture}
        \mapsto
        \begin{tikzpicture}[centerzero]
            \draw (0,-0.2) rectangle (1,0.2);
            \draw[spin] (0.2,0.2) arc(0:180:0.2) -- (-0.2,-0.7);
            \node at (0,0.8) {$\vdots$};
            \draw[spin] (0.8,0.2) arc(0:180:0.8) -- (-0.8,-0.7);
            \draw[spin] (0.2,-0.2) -- (0.2,-0.7);
            \draw[spin] (0.8,-0.2) -- (0.8,-0.7);
            \node at (0.52,-0.45) {$\cdots$};
            \node at (-0.48,-0.45) {$\cdots$};
        \end{tikzpicture}
        \qquad \text{with inverse} \qquad
        \begin{tikzpicture}[centerzero]
            \draw (-1,-0.2) rectangle (1,0.2);
            \draw[spin] (-0.8,-0.2) -- (-0.8,-1);
            \draw[spin] (-0.2,-0.2) -- (-0.2,-1);
            \node at (-0.48,-0.6) {$\cdots$};
            \draw[spin] (0.2,-0.2) -- (0.2,-1);
            \draw[spin] (0.8,-0.2) -- (0.8,-1);
            \node at (0.52,-0.6) {$\cdots$};
        \end{tikzpicture}
        \mapsto\,
        \begin{tikzpicture}[centerzero]
            \draw (-1,-0.2) rectangle (1,0.2);
            \draw[spin] (-0.8,-0.2) arc(360:180:0.2) -- (-1.2,1);
            \draw[spin] (-0.2,-0.2) arc(360:180:0.8) -- (-1.8,1);
            \node at (-1.45,0.6) {$\cdots$};
            \node at (-1,-0.6) {$\vdots$};
            \draw[spin] (0.2,-0.2) -- (0.2,-1);
            \draw[spin] (0.8,-0.2) -- (0.8,-1);
            \node at (0.52,-0.6) {$\cdots$};
        \end{tikzpicture}
        \ ,
    \]
    where the rectangles denote some diagram.  Therefore, part \cref{Sfull} holds for $r_1,r_2 \in \N$ if and only if it holds for all other $r_1',r_2' \in \N$ satisfying $r_1 + r_2 = r_1' + r_2'$.  It also follows that, for $0 \le k \le \frac{r}{2}$,
    \begin{gather} \label{smoke1}
        \Hom_{\Group(V)}(S^{\otimes (r-2k)}, S^{\otimes r})
        = \left( 1^{\otimes k} \otimes \End_{\Group(V)}(S^{\otimes (r-k)}) \right) \circ \bF
        \left(
            \begin{tikzpicture}[centerzero]
                \draw[spin] (-0.2,0.2) -- (-0.2,0) arc(180:360:0.2) -- (0.2,0.2);
                \node at (0,0) {\strandlabel{k}};
            \end{tikzpicture}
            \otimes 1^{\otimes (r-2k)}
        \right),
        \\ \label{smoke2}
        \Hom_{\Group(V)}(S^{\otimes r}, S^{\otimes (r-2k)})
        = \bF
        \left(
            \begin{tikzpicture}[centerzero]
                \draw[spin] (-0.2,-0.2) -- (-0.2,0) arc(180:0:0.2) -- (0.2,-0.2);
                \node at (0,0) {\strandlabel{k}};
            \end{tikzpicture}
            \otimes 1^{\otimes (r-2k)}
        \right)
        \circ \left( 1^{\otimes k} \otimes \End_{\Group(V)}(S^{\otimes (r-k)}) \right),
    \end{gather}
    where, as in \cref{calculator}, the thick cup and cap labelled by $k$ denote $k$ nested cups and caps, respectively.

    We now prove the theorem by induction on $N = \dim V$. The base cases are $N=0$ and $N=1$. In these cases, $S$ is one-dimensional, and so $\End_{\Group(V)}(S^{\otimes r})$ only consists of scalars, which makes the theorem trivial in these cases.

    Now suppose that $N \ge 2$ and that the result holds for $0 \le \dim V < N$.  For $r \in \N$, let $F(r)$ be the statement that \cref{Sfull} holds for all $r_1+r_2 = 2r$ and that \cref{BBgen} holds.  By the argument given above, $F(r)$ is equivalent to the statement that \cref{Sfull} holds for \emph{some} $r_1,r_2 \in \N$ satisfying $r_1+r_2=2r$ and that \cref{BBgen} holds.  We will prove that $F(r)$ holds for all $r \in \N$ by induction on $r$.  The base of the induction consists of the cases $r \leq 2$. The cases $r \leq 1$ are trivial as $S$ is a simple $\Group(V)$-module, and so $\Hom_{\Group(V)}(S^{\otimes r})$ consists of scalar multiples of the identity. The case $r=2$ follows from \cref{barbellgen}.

    Now suppose that $r \ge 3$, and that $F(k)$ holds for all $k<r$.  Recall the surjective $\kk$-linear map
    \[
        \Psi_r \colon \End_{\Group(V)}(S^{\otimes r}) \twoheadrightarrow \End_{\Group(W)}(S_W^{\otimes r})
    \]
    defined in \cref{okinawa}.  The kernel of $\Psi_r$ consists of all elements of $\End_{\Group(V)}(S^{\otimes r})$ that factor through simple modules with highest weights $\la$ with $\la_1<r/2$ (when $N>2$) or $-r/2 < \lambda_1 < r_2$ (when $N=2$).  By \cref{pinsummand,mitre}, these simple modules are precisely the simple modules that occur as summands in $S^{\otimes (r-2k)}$ for $0 < k \leq \frac{r}{2}$.
    \details{
        Here we need $r \ge 3$ because of the hypotheses in \cref{pinsummand}; see \cref{mitre}.  This is why we need the base case $r=2$.
    }
    Therefore, $\ker \Psi_r$ is the sum of all images of all compositions
    \[
        \Hom_{\Group(V)}(S^{\otimes r},S^{\otimes (r-2k)}) \times \Hom_{\Group(V)}(S^{\otimes (r-2k)},S^{\otimes r})
        \to \End_{\Group(V)}(S^{\otimes r}),\qquad
        0 < k \leq  \frac{r}{2}.
    \]
    It follows from \cref{smoke1}, \cref{smoke2}, \cref{calculator}, and the inductive hypothesis that $\ker \Psi_r$ is generated by barbells and hence is in the image of $\bF$.

    By the inductive hypothesis for our induction on $N$, barbells generate $\End_{\Group(W)}(S_W^{\otimes r})$.  By \cref{couch}, every barbell in $\End_{\Group(W)}(S_W^{\otimes r})$ is in the image of $\Psi_r$.  Thus, if $U$ denotes the subalgebra of $\End_{\Group(V)}(S^{\otimes r})$ generated by barbells, we have $\Psi_r(U) = \End_{\Group(W)}(S_W^{\otimes r})$.  Since $\Psi_r$ is surjective, this implies that
    \[
        U + \ker \Psi_r = \End_{\Group(V)}(S^{\otimes r}).
    \]
    The subspace $U$ lies in the image of $\bF$ since all barbells do.  This completes the proof of the statement $F(r)$.
\end{proof}

\begin{theo} \label{fullness}
    The incarnation functor $\bF$ is full.
\end{theo}

\begin{proof}
    We must show that
    \[
        \bF \colon \Hom_{\SB(V)}(\mathsf{X}, \mathsf{Y})
        \to \Hom_{\Group(V)}(\bF(\mathsf{X}), \bF(\mathsf{Y}))
    \]
    is surjective for all objects $\mathsf{X}$ and $\mathsf{Y}$ in $\SB(V)$.  By the first relation in \cref{brauer}, we have mutually inverse isomorphisms
    \[
        \crossmor{spin}{vec} \colon \Sgo \otimes \Vgo \xrightarrow{\cong} \Vgo \otimes \Sgo
        \qquad \text{and} \qquad
        \crossmor{vec}{spin} \colon \Vgo \otimes \Sgo \xrightarrow{\cong} \Sgo \otimes \Vgo.
    \]
    Therefore, it suffices to consider the case where $\mathsf{X} = \Sgo^{\otimes k_1} \otimes \Vgo^{\otimes l_1}$ and $\mathsf{Y} = \Sgo^{\otimes k_2} \otimes \Vgo^{\otimes l_2}$ for some $k_1,l_1,k_2,l_2 \in \N$.  Consider an arbitrary morphism
    \[
        f \in \Hom_{\Group(V)} \left( S^{\otimes k_1} \otimes V^{\otimes l_1}, S^{\otimes k_2} \otimes V^{\otimes l_2} \right).
    \]
    Define
    \[
        f' =
        \left(
            1_S^{\otimes k_2} \otimes \bF
            \left(
                \begin{tikzpicture}[anchorbase]
                    \draw[spin] (-0.15,0.25) -- (-0.15,0.15) arc(180:360:0.15) -- (0.15,0.25);
                    \draw[vec] (0,0) -- (0,-0.3);
                \end{tikzpicture}
            \right)^{\otimes l_2}
        \right)
        \circ f \circ
        \left(
            1_S^{\otimes k_1} \otimes \bF
            \left(
                \begin{tikzpicture}[anchorbase]
                    \draw[spin] (-0.15,-0.25) -- (-0.15,-0.15) arc(180:0:0.15) -- (0.15,-0.25);
                    \draw[vec] (0,0) -- (0,0.3);
                \end{tikzpicture}
            \right)^{\otimes l_1}
        \right)
        \in \Hom_{\Group(V)} \left( S^{\otimes (k_1+2l_1)}, S^{\otimes (k_2+2l_2)} \right).
    \]
    By \cref{surly}\cref{Sfull}, there exists a $g' \in \Hom_{\SB(V)}(\Sgo^{\otimes (k_1+2l_1)}, \Sgo^{\otimes (k_2+2l_2)})$ such that $\bF(g') = f'$.  Let
    \[
        g =
        \left(
            1_S^{\otimes k_2} \otimes
            \left(
                \begin{tikzpicture}[anchorbase]
                    \draw[spin] (-0.15,-0.25) -- (-0.15,-0.15) arc(180:0:0.15) -- (0.15,-0.25);
                    \draw[vec] (0,0) -- (0,0.3);
                \end{tikzpicture}
            \right)^{\otimes l_2}
        \right)
        \circ g' \circ
        \left(
            1_S^{\otimes k_1} \otimes
            \left(
                \begin{tikzpicture}[anchorbase]
                    \draw[spin] (-0.15,0.25) -- (-0.15,0.15) arc(180:360:0.15) -- (0.15,0.25);
                    \draw[vec] (0,0) -- (0,-0.3);
                \end{tikzpicture}
            \right)^{\otimes l_1}
        \right).
    \]
    Using \cref{Delprop} with $r=2$ we have
    \[
        0
        =
        \begin{tikzpicture}[anchorbase]
            \draw[spin] (-0.1,0) -- (0.1,0) arc(-90:90:0.15) -- (-0.1,0.3) arc(90:270:0.15);
            \draw[vec] (-0.1,-0.4) -- (-0.1,0);
            \draw[vec] (0.1,-0.4) -- (0.1,0);
        \end{tikzpicture}
        -
        \begin{tikzpicture}[anchorbase]
            \draw[spin] (-0.1,0) -- (0.1,0) arc(-90:90:0.15) -- (-0.1,0.3) arc(90:270:0.15);
            \draw[vec] (-0.1,-0.4) -- (0.1,0);
            \draw[vec] (0.1,-0.4) -- (-0.1,0);
        \end{tikzpicture}
        \overset{\cref{oist}}{=}
        2\
        \begin{tikzpicture}[anchorbase]
            \draw[spin] (-0.1,0) -- (0.1,0) arc(-90:90:0.15) -- (-0.1,0.3) arc(90:270:0.15);
            \draw[vec] (-0.1,-0.4) -- (-0.1,0);
            \draw[vec] (0.1,-0.4) -- (0.1,0);
        \end{tikzpicture}
        - 2\
        \begin{tikzpicture}[anchorbase]
            \draw[spin] (-0.1,0) -- (0.1,0) arc(-90:90:0.15) -- (-0.1,0.3) arc(90:270:0.15);
            \draw[vec] (-0.15,-0.4) -- (-0.15,-0.3) arc(180:0:0.15) -- (0.15,-0.4);
        \end{tikzpicture}
        \overset{\cref{dimrel}}{=}
        2\
        \begin{tikzpicture}[anchorbase]
            \draw[spin] (-0.1,0) -- (0.1,0) arc(-90:90:0.15) -- (-0.1,0.3) arc(90:270:0.15);
            \draw[vec] (-0.1,-0.4) -- (-0.1,0);
            \draw[vec] (0.1,-0.4) -- (0.1,0);
        \end{tikzpicture}
        - 2 D\
        \begin{tikzpicture}[anchorbase]
            \draw[vec] (-0.2,-0.4) -- (-0.2,-0.2) arc(180:0:0.2) -- (0.2,-0.4);
        \end{tikzpicture}
        \implies
        \begin{tikzpicture}[centerzero]
            \draw[vec] (0,-0.5) -- (0,-0.2);
            \draw[vec] (0,0.5) -- (0,0.2);
            \draw[spin] (0,0) circle(0.2);
        \end{tikzpicture}
        = D\
        \begin{tikzpicture}[centerzero]
            \draw[vec] (0,-0.5) -- (0,0.5);
        \end{tikzpicture}
        \ .
    \]
    It follows that
    \[
        \bF(g)
        = D^{l_1 + l_2} f,
    \]
    completing the proof.
\end{proof}

\section{Essential surjectivity of the incarnation functor\label{sec:essential}}

In this section, we prove that, after passing to the additive Karoubi envelope, the incarnation functor is essentially surjective, that is, it induces a surjection on isomorphism classes of objects.   We then give explicit descriptions of some important idempotents.  When discussing the incarnation functor, we always assume that $\kk =\C$.

Let $\Kar(\overline{\SB}(d,D;\kappa))$ be the additive Karoubi envelope (that is, the idempotent completion of the additive envelope) of $\SB(d,D;\kappa)$.  Since $\Group(V)\md$ is additive and idempotent complete, $\bF$ induces a monoidal functor
\begin{gather} \label{Karnation}
    \Kar(\bF) \colon \Kar \big( \overline{\SB}(V) \big) \to \Group(V)\md.
\end{gather}

\begin{theo} \label{essential}
    For all $N \in \N$, the functor $\Kar(\bF)$ is essentially surjective.
\end{theo}

\begin{proof}
    The spin module $S$ is a self-dual faithful $\Group(V)$-module. Hence it is a tensor generator of $\Group(V)\md$. Let $M$ be a simple $\Group(V)$-module. The above implies that $M$ is a direct summand of $S^{\otimes k}$ for some $k$. By \cref{surly}\cref{Sfull}, the idempotent in $\End_{\Group(V)}(S^{\otimes k})$ projecting onto $M$ is in the image of the incarnation functor $\bF$, hence $M$ is in the essential image of $\Kar(\bF)$. Since $\Group(V)\md$ is semisimple, this completes the proof that $\Kar(\bF)$ is essentially surjective.\footnote{Note added after publication:  There is a gap in this proof.  The lifting of idempotents needs further justification, since the morphism spaces in the spin Brauer category may be infinite dimensional.  This gap can be fixed using the strategy of the proof of \cite[Th.~8.5]{SM25} which considers the quantum setting.}
\end{proof}

\begin{rem} \label{panda}
    The Lie algebra of $\Group(V)$ is $\fso(V)$. Passage to the Lie algebra induces a functor $\Group(V)\md \to \fso(V)\md$, where $\fso(V)\md$ denotes the category of finite-dimensional $\fso(V)$-modules.  We can compose the incarnation functor $\bF$ with this passage to the Lie algebra to yield a functor  $\bF' \colon \SB(V) \to \fso(V)\md$, which factors through $\overline{\SB}(V)$.
    However, while we have shown in \cref{fullness} and \cref{essential} that $\bF$ is full and $\Kar(\bF)$ is essentially surjective, the functor $\bF'$ is \emph{not} full and the functor $\Kar(\bF')$ is \emph{not} essentially surjective when $N$ is even.  For example, as $\fso(V)$-modules, $\Lambda^N(V)$ is isomorphic to the trivial module.  But this isomorphism is not contained in the image of $\bF'$ since $\Lambda^N(V)$ is nontrivial as a $\Pin(V)$-module when $N$ is even, by \cref{lemonD1}.  The functor $\Kar(\bF')$ is not essentially surjective since there are modules for $\fso(V)$ that are not the restriction of $\Group(V)$-modules; an example is either of the two summands of the spin module $S$.
     This is our main motivation for considering the larger group $\Pin(V)$ when $N$ is even.
\end{rem}

It is straightforward to verify that $\SB(d,D;\kappa)$ is a spherical pivotal category, hence so is its idempotent completion $\Kar(\SB(d,D;\kappa))$.  (We refer the reader to \cite[\S 4.4.3]{Sel11} for the definition of a spherical pivotal category.)  In any spherical pivotal category $\cC$, we have a trace map $\Tr \colon \bigoplus_{X \in \cC} \End_\cC(X) \to \End_\cC(\one)$.  In terms of string diagrams, this corresponds to closing a diagram off to the right or left:
\begin{equation} \label{natal}
    \Tr
    \left(
        \begin{tikzpicture}[centerzero]
            \draw[line width=2] (0,-0.5) -- (0,0.5);
            \filldraw[fill=white,draw=black] (-0.25,0.2) rectangle (0.25,-0.2);
            \node at (0,0) {$\scriptstyle{f}$};
        \end{tikzpicture}
    \right)
    =
    \begin{tikzpicture}[centerzero]
        \draw[line width=2] (0,0.2) arc(180:0:0.3) -- (0.6,-0.2) arc(360:180:0.3);
        \filldraw[fill=white,draw=black] (-0.25,0.2) rectangle (0.25,-0.2);
        \node at (0,0) {$\scriptstyle{f}$};
    \end{tikzpicture}
    =
    \begin{tikzpicture}[centerzero]
        \draw[line width=2] (0,0.2) arc(0:180:0.3) -- (-0.6,-0.2) arc(180:360:0.3);
        \filldraw[fill=white,draw=black] (-0.25,0.2) rectangle (0.25,-0.2);
        \node at (0,0) {$\scriptstyle{f}$};
    \end{tikzpicture}
    \ ,
\end{equation}
where the second equality follows from the axioms of a spherical category.  We say that a morphism $f \in \Hom_\cC(X,Y)$ is \emph{negligible} if $\Tr(f \circ g) = 0$ for all $g \in \Hom_\cC(Y,X)$.  The negligible morphisms form a two-sided tensor ideal $\cN$ of $\cC$, and the quotient $\cC/\cN$ is called the \emph{semisimplification} of $\cC$.

\begin{theo} \label{negligible}
    For all $N \in \N$, the kernel of the functor $\Kar(\bF)$ of \cref{Karnation} is equal to the tensor ideal of negligible morphisms of $\Kar(\overline{\SB}(V))$.  The functor $\Kar(\bF)$ induces an equivalence of categories from the semisimplification of $\Kar(\overline{\SB}(V))$ to $\Group(V)\md$.
\end{theo}

\begin{proof}
    By \cref{fullness,essential} the functor $\Kar(\bF)$ is full and essentially surjective.  It follows from \cref{popping} and \cite[Prop.~6.9]{SW22} that its kernel is the tensor ideal of negligible morphisms.
\end{proof}

We spend the rest of this section explicitly constructing idempotents in the spin Brauer category that correspond, under the incarnation functor, to projections onto the simple summands of the tensor products $S^{\otimes 2}$, as in \cref{Sdub}, and onto the summand $S$ in $S \otimes V$.  For the rest of this section, we assume that $\kk$ is a field of characteristic zero.  For statements involving the incarnation functor, we assume that $\kk = \C$.

Recall the definition of the antisymmetrizer \cref{altbox}.  If $D \ne 0$, define
\begin{equation}
    \pi_r :=
    \frac{1}{D(r!)^2}\
    \begin{tikzpicture}[centerzero]
        \draw[vec] (0,-0.5) -- (0,0.5);
        \altbox{-0.2,-0.15}{0.2,0.15}{r};
        \draw[spin] (-0.2,0.8) -- (-0.2,0.7) arc(180:360:0.2) -- (0.2,0.8);
        \draw[spin] (-0.2,-0.8) -- (-0.2,-0.7) arc(180:0:0.2) -- (0.2,-0.8);
    \end{tikzpicture}
    \ \in \End_{\SB(d,D;\kappa)} \left( S^{\otimes 2} \right),\qquad
    r \in \N.
\end{equation}

\begin{prop} \label{Sproject}
    \begin{enumerate}
        \item \label{Sproject1} If $d \notin 2\N + 1$, then $\pi_r \pi_s = 0$ for all $r \ne s$.

        \item \label{Sproject1b} If $d \in 2\N + 1$, then $\pi_r \pi_s = 0$ for all $r \ne s$, $0 \le r+s < d$.

        \item \label{Sproject2} If $d \notin \{0,1,\dotsc,r-1\}$, then $\pi_r^2 = \pi_r$.

        \item \label{Sproject3} If $d=N$, $D = \sigma_N 2^n$, and $0 \le r \le N$, then $\bF(\pi_r)$ is the projection $S^{\otimes 2} \twoheadrightarrow \Lambda^r(V)$ with respect to the decompositions of \cref{Sdub}.
    \end{enumerate}
\end{prop}

\begin{proof}
    Recall that, for $r,s \in \N$, an \emph{$(r,s)$-shuffle} is a permutation $g$ of the set $\{1,\dotsc,r+s\}$ such that
    \[
        g(1) < \dotsb < g(r)
        \quad \text{and} \quad
        g(r+1) < \dotsb g(r+s).
    \]
    The set $\Shuffle(r,s)$ of $(r,s)$-shuffles is a complete set of representatives of the left cosets of the subgroup $\fS_r \times \fS_s$ of $\fS_{r+s}$.  Thus, we have
    \begin{equation} \label{beard}
        \begin{tikzpicture}[anchorbase]
            \draw[spin] (-0.1,0) -- (0.1,0) arc(-90:90:0.15) -- (-0.1,0.3) arc(90:270:0.15);
            \draw[vec] (0,-1) -- (0,0);
            \altbox{-0.4,-0.65}{0.4,-0.35}{r+s};
        \end{tikzpicture}
        = \sum_{g \in \Shuffle(r,s)} \sgn(g)\
        \begin{tikzpicture}[anchorbase]
            \draw[spin] (-0.1,0) -- (0.1,0) arc(-90:90:0.15) -- (-0.1,0.3) arc(90:270:0.15);
            \draw[vec] (0,-0.5) -- (0,0);
            \draw[vec] (-0.3,-1.4) -- (-0.3,-0.5);
            \draw[vec] (0.3,-1.4) -- (0.3,-0.5);
            \altbox{-0.4,-0.65}{0.4,-0.35}{g};
            \altbox{-0.5,-1.2}{-0.1,-0.9}{r};
            \altbox{0.5,-1.2}{0.1,-0.9}{s};
        \end{tikzpicture}
        \overset{\cref{oist}}{\underset{\cref{absorb}}{=}} \sum_{g \in \Shuffle(r,s)}
        \begin{tikzpicture}[anchorbase]
            \draw[spin] (-0.4,0) -- (0.4,0) arc(-90:90:0.15) -- (-0.4,0.3) arc(90:270:0.15);
            \draw[vec] (-0.27,-1) -- (-0.27,0);
            \draw[vec] (0.27,-1) -- (0.27,0);
            \altbox{-0.5,-0.65}{-0.1,-0.35}{r};
            \altbox{0.5,-0.65}{0.1,-0.35}{s};
        \end{tikzpicture}
        + \sum_{a=1}^{\min(r,s)} c_a\
        \begin{tikzpicture}[anchorbase]
            \draw[spin] (-0.4,0.2) -- (0.4,0.2) arc(-90:90:0.15) -- (-0.4,0.5) arc(90:270:0.15);
            \draw[vec] (-0.4,-0.8) -- (-0.4,0.2);
            \draw[vec] (-0.3,-1.1) -- (-0.3,-0.8);
            \draw[vec] (0.4,-0.8) -- (0.4,0.2);
            \draw[vec] (0.3,-1.1) -- (0.3,-0.8);
            \draw[vec] (-0.2,-0.5) arc(180:0:0.2) node[midway,anchor=south] {\strandlabel{a}};
            \altbox{-0.5,-0.8}{-0.1,-0.5}{r};
            \altbox{0.5,-0.8}{0.1,-0.5}{s};
        \end{tikzpicture}
        \ ,
    \end{equation}
    for some $c_a \in \kk$, $1 \le a \le \min(r,s)$.

    We now prove \cref{Sproject1} by induction on $r+s$.  Suppose $d \notin 2\N+1$ and $r \ne s$.  It suffices to show that
    \[
        \begin{tikzpicture}[anchorbase]
            \draw[spin] (-0.4,0) -- (0.4,0) arc(-90:90:0.15) -- (-0.4,0.3) arc(90:270:0.15);
            \draw[vec] (-0.27,-1) -- (-0.27,0);
            \draw[vec] (0.27,-1) -- (0.27,0);
            \altbox{-0.5,-0.65}{-0.1,-0.35}{r};
            \altbox{0.5,-0.65}{0.1,-0.35}{s};
        \end{tikzpicture}
        = 0.
    \]
    For the base case $r+s=1$, the result follows immediately from the $r=1$ case of \cref{Deligne}.  Now suppose $r + s > 1$.  Then, for $1 \le a \le \min(r,s)$, we have
    \[
        \begin{tikzpicture}[anchorbase]
            \draw[spin] (-0.4,0.2) -- (0.4,0.2) arc(-90:90:0.15) -- (-0.4,0.5) arc(90:270:0.15);
            \draw[vec] (-0.4,-0.8) -- (-0.4,0.2);
            \draw[vec] (-0.3,-1.1) -- (-0.3,-0.8);
            \draw[vec] (0.4,-0.8) -- (0.4,0.2);
            \draw[vec] (0.3,-1.1) -- (0.3,-0.8);
            \draw[vec] (-0.2,-0.5) arc(180:0:0.2) node[midway,anchor=south] {\strandlabel{a}};
            \altbox{-0.5,-0.8}{-0.1,-0.5}{r};
            \altbox{0.5,-0.8}{0.1,-0.5}{s};
        \end{tikzpicture}
        \overset{\cref{absorb}}{=} \tfrac{1}{(r-a)!(s-a)!}\
        \begin{tikzpicture}[anchorbase]
            \draw[spin] (-0.6,0.7) -- (0.6,0.7) arc(-90:90:0.15) -- (-0.6,1) arc(90:270:0.15);
            \draw[vec] (-0.4,-0.8) -- (-0.4,0.7);
            \draw[vec] (-0.3,-1.1) -- (-0.3,-0.8);
            \draw[vec] (0.4,-0.8) -- (0.4,0.7);
            \draw[vec] (0.3,-1.1) -- (0.3,-0.8);
            \draw[vec] (-0.2,-0.5) arc(180:0:0.2) node[midway,anchor=south] {\strandlabel{a}};
            \altbox{-0.5,-0.8}{-0.1,-0.5}{r};
            \altbox{0.5,-0.8}{0.1,-0.5}{s};
            \altbox{-0.8,0.4}{-0.1,0.1}{r-a};
            \altbox{0.8,0.4}{0.1,0.1}{s-a};
        \end{tikzpicture}
        = 0,
    \]
    where the last equality follows from the inductive hypothesis.  Therefore,
    \begin{equation} \label{hail}
        0 \overset{\cref{Deligne}}{=}
        \begin{tikzpicture}[anchorbase]
            \draw[spin] (-0.1,0) -- (0.1,0) arc(-90:90:0.15) -- (-0.1,0.3) arc(90:270:0.15);
            \draw[vec] (0,-1) -- (0,0);
            \altbox{-0.4,-0.65}{0.4,-0.35}{r+s};
        \end{tikzpicture}
        \overset{\cref{beard}}{=} \tfrac{(r+s)!}{r!s!}
        \begin{tikzpicture}[anchorbase]
            \draw[spin] (-0.4,0) -- (0.4,0) arc(-90:90:0.15) -- (-0.4,0.3) arc(90:270:0.15);
            \draw[vec] (-0.27,-1) -- (-0.27,0);
            \draw[vec] (0.27,-1) -- (0.27,0);
            \altbox{-0.5,-0.65}{-0.1,-0.35}{r};
            \altbox{0.5,-0.65}{0.1,-0.35}{s};
        \end{tikzpicture}
        + \sum_{a=1}^{\min(r,s)} c_a\
        \begin{tikzpicture}[anchorbase]
            \draw[spin] (-0.4,0.2) -- (0.4,0.2) arc(-90:90:0.15) -- (-0.4,0.5) arc(90:270:0.15);
            \draw[vec] (-0.4,-0.8) -- (-0.4,0.2);
            \draw[vec] (-0.3,-1.1) -- (-0.3,-0.8);
            \draw[vec] (0.4,-0.8) -- (0.4,0.2);
            \draw[vec] (0.3,-1.1) -- (0.3,-0.8);
            \draw[vec] (-0.2,-0.5) arc(180:0:0.2) node[midway,anchor=south] {\strandlabel{a}};
            \altbox{-0.5,-0.8}{-0.1,-0.5}{r};
            \altbox{0.5,-0.8}{0.1,-0.5}{s};
        \end{tikzpicture}
        = \tfrac{(r+s)!}{r!s!}
        \begin{tikzpicture}[anchorbase]
            \draw[spin] (-0.4,0) -- (0.4,0) arc(-90:90:0.15) -- (-0.4,0.3) arc(90:270:0.15);
            \draw[vec] (-0.27,-1) -- (-0.27,0);
            \draw[vec] (0.27,-1) -- (0.27,0);
            \altbox{-0.5,-0.65}{-0.1,-0.35}{r};
            \altbox{0.5,-0.65}{0.1,-0.35}{s};
        \end{tikzpicture}
        \ ,
    \end{equation}
    and the result follows.  The proof of \cref{Sproject1b} is identical, except that the first equality in \cref{hail} uses the assumption that $r+s < d$.

    Next, we prove \cref{Sproject2}.  We first show, by induction on $r$, that
    \begin{equation} \label{lego}
        \begin{tikzpicture}[centerzero]
            \draw[vec] (0,-1) -- (0,-0.15);
            \draw[vec] (0,0.15) -- (0,1);
            \draw[spin] (-0.1,-0.15) -- (0.1,-0.15) arc(-90:90:0.15) -- (-0.1,0.15) arc(90:270:0.15);
            \altbox{-0.2,-0.75}{0.2,-0.45}{r};
            \altbox{-0.2,0.75}{0.2,0.45}{r};
        \end{tikzpicture}
        = b_r\
        \begin{tikzpicture}[centerzero]
            \draw[vec] (0,-0.6) -- (0,0.6);
            \altbox{-0.2,-0.15}{0.2,0.15}{r};
        \end{tikzpicture}
        \qquad \text{for some } b_r \in \kk.
    \end{equation}
    The base case $r=0$ follows immediately from the second relation in \cref{dimrel}.  Now assume $r > 1$ and that the result holds for $r-1$.  Then, for $a \ge 1$,
    \begin{equation} \label{rocks}
        \begin{tikzpicture}[anchorbase]
            \draw[spin] (-0.4,0.2) -- (0.4,0.2) arc(-90:90:0.15) -- (-0.4,0.5) arc(90:270:0.15);
            \draw[vec] (-0.4,-0.8) -- (-0.4,0.2);
            \draw[vec] (-0.3,-1.1) -- (-0.3,-0.8);
            \draw[vec] (0.4,-0.8) -- (0.4,0.2);
            \draw[vec] (0.3,-1.1) -- (0.3,-0.8);
            \draw[vec] (-0.2,-0.5) arc(180:0:0.2) node[midway,anchor=south] {\strandlabel{a}};
            \altbox{-0.5,-0.8}{-0.1,-0.5}{r};
            \altbox{0.5,-0.8}{0.1,-0.5}{r};
        \end{tikzpicture}
        \overset{\cref{absorb}}{=} \tfrac{1}{((r-a)!)^2}\
        \begin{tikzpicture}[anchorbase]
            \draw[spin] (-0.6,0.7) -- (0.6,0.7) arc(-90:90:0.15) -- (-0.6,1) arc(90:270:0.15);
            \draw[vec] (-0.4,-0.8) -- (-0.4,0.7);
            \draw[vec] (-0.3,-1.1) -- (-0.3,-0.8);
            \draw[vec] (0.4,-0.8) -- (0.4,0.7);
            \draw[vec] (0.3,-1.1) -- (0.3,-0.8);
            \draw[vec] (-0.2,-0.5) arc(180:0:0.2) node[midway,anchor=south] {\strandlabel{a}};
            \altbox{-0.5,-0.8}{-0.1,-0.5}{r};
            \altbox{0.5,-0.8}{0.1,-0.5}{r};
            \altbox{-0.8,0.4}{-0.1,0.1}{r-a};
            \altbox{0.8,0.4}{0.1,0.1}{r-a};
        \end{tikzpicture}
        = \tfrac{b_{r-a}}{((r-a)!)^2}\
        \begin{tikzpicture}[anchorbase]
            \draw[vec] (-0.4,-0.8) -- (-0.4,0.4) arc(180:0:0.4) -- (0.4,-0.8);
            \draw[vec] (-0.3,-1.1) -- (-0.3,-0.8);
            \draw[vec] (0.3,-1.1) -- (0.3,-0.8);
            \draw[vec] (-0.2,-0.5) arc(180:0:0.2) node[midway,anchor=south] {\strandlabel{a}};
            \altbox{-0.5,-0.8}{-0.1,-0.5}{r};
            \altbox{0.5,-0.8}{0.1,-0.5}{r};
            \altbox{-0.8,0.4}{-0.1,0.1}{r-a};
        \end{tikzpicture}
        \overset{\cref{absorb}}{=} \tfrac{b_{r-a}}{(r-a)!}\
        \begin{tikzpicture}[anchorbase]
            \draw[vec] (-0.3,-1.1) -- (-0.3,-0.4) arc(180:0:0.3) -- (0.3,-1.1);
            \altbox{-0.5,-0.8}{-0.1,-0.5}{r};
            \altbox{0.5,-0.8}{0.1,-0.5}{r};
        \end{tikzpicture}
        \ .
    \end{equation}
    Thus,
    \[
        0 \overset{\cref{Deligne}}{=}
        \begin{tikzpicture}[anchorbase]
            \draw[spin] (-0.1,0) -- (0.1,0) arc(-90:90:0.15) -- (-0.1,0.3) arc(90:270:0.15);
            \draw[vec] (0,-1) -- (0,0);
            \altbox{-0.3,-0.65}{0.3,-0.35}{2r};
        \end{tikzpicture}
        \overset{\cref{beard}}{=} \tfrac{(2r)!}{(r!)^2}\
        \begin{tikzpicture}[anchorbase]
            \draw[spin] (-0.4,0) -- (0.4,0) arc(-90:90:0.15) -- (-0.4,0.3) arc(90:270:0.15);
            \draw[vec] (-0.27,-1) -- (-0.27,0);
            \draw[vec] (0.27,-1) -- (0.27,0);
            \altbox{-0.5,-0.65}{-0.1,-0.35}{r};
            \altbox{0.5,-0.65}{0.1,-0.35}{r};
        \end{tikzpicture}
        + \sum_{a=1}^r c_a\
        \begin{tikzpicture}[anchorbase]
            \draw[spin] (-0.4,0.2) -- (0.4,0.2) arc(-90:90:0.15) -- (-0.4,0.5) arc(90:270:0.15);
            \draw[vec] (-0.4,-0.8) -- (-0.4,0.2);
            \draw[vec] (-0.3,-1.1) -- (-0.3,-0.8);
            \draw[vec] (0.4,-0.8) -- (0.4,0.2);
            \draw[vec] (0.3,-1.1) -- (0.3,-0.8);
            \draw[vec] (-0.2,-0.5) arc(180:0:0.2) node[midway,anchor=south] {\strandlabel{a}};
            \altbox{-0.5,-0.8}{-0.1,-0.5}{r};
            \altbox{0.5,-0.8}{0.1,-0.5}{r};
        \end{tikzpicture}
        \overset{\cref{rocks}}{=} \tfrac{(2r)!}{(r!)^2}\
        \begin{tikzpicture}[anchorbase]
            \draw[spin] (-0.4,0) -- (0.4,0) arc(-90:90:0.15) -- (-0.4,0.3) arc(90:270:0.15);
            \draw[vec] (-0.27,-1) -- (-0.27,0);
            \draw[vec] (0.27,-1) -- (0.27,0);
            \altbox{-0.5,-0.65}{-0.1,-0.35}{r};
            \altbox{0.5,-0.65}{0.1,-0.35}{r};
        \end{tikzpicture}
        + \sum_{a=1}^r \tfrac{c_a b_{r-a}}{(r-a)!}\
        \begin{tikzpicture}[anchorbase]
            \draw[vec] (-0.3,-1.1) -- (-0.3,-0.4) arc(180:0:0.3) -- (0.3,-1.1);
            \altbox{-0.5,-0.8}{-0.1,-0.5}{r};
            \altbox{0.5,-0.8}{0.1,-0.5}{r};
        \end{tikzpicture}
        \ ,
    \]
    and \cref{lego} follows.

    Now, taking the trace of both sides of \cref{lego}, we see that
    \[
        b_r\
        \begin{tikzpicture}[anchorbase]
            \draw[vec] (0,-0.15) arc(180:360:0.2) -- (0.4,0.15) arc(0:180:0.2);
            \altbox{-0.2,-0.15}{0.2,0.15}{r};
        \end{tikzpicture}
        \ =\
        \begin{tikzpicture}[centerzero]
            \draw[vec] (0,0.15) -- (0,0.8) arc (180:0:0.2) -- (0.4,-0.8) arc(360:180:0.2) -- (0,-0.15);
            \draw[spin] (-0.1,-0.15) -- (0.1,-0.15) arc(-90:90:0.15) -- (-0.1,0.15) arc(90:270:0.15);
            \altbox{-0.2,-0.75}{0.2,-0.45}{r};
            \altbox{-0.2,0.75}{0.2,0.45}{r};
        \end{tikzpicture}
        \ \overset{\cref{absorb}}{=} r!\
        \begin{tikzpicture}[anchorbase]
            \draw[spin] (-0.1,0) -- (0.1,0) arc(-90:90:0.15) -- (-0.1,0.3) arc(90:270:0.15);
            \draw[vec] (0,0) -- (0,-0.65) arc(180:360:0.2) -- (0.4,0.3) arc(0:180:0.2);
            \altbox{-0.2,-0.65}{0.2,-0.35}{r};
        \end{tikzpicture}
        \ .
    \]
    Thus, by \cref{monkey1,monkey2}, we have
    \[
        b_r d(d-1) \dotsm (d-r+1) = D (r!)^2 d(d-1) \dotsb (d-r+1),
    \]
    and so $b_r = D(r!)^2$.

    Finally, to prove \cref{Sproject3}, note that $F(\pi_r)$ is an idempotent $\Pin(V)$-module homomorphism $S^{\otimes 2} \to \Lambda^r(V) \to S^{\otimes 2}$.  Since its trace is nonzero by \cref{monkey1}, the result follows from \cref{Sdub}.
\end{proof}

For $N \ge 1$, we have
\begin{equation} \label{volley1}
    V^{\otimes 2} \cong S^2(V) \oplus \Lambda^2(V),\qquad
    S^2(V) \cong \triv^0 \oplus W,\qquad
    \text{as $\Pin(V)$-modules}.
\end{equation}
When $N=1$, we have $W = \Lambda^2(V) = 0$.  For $N \ge 2$, the $\Pin(V)$-modules $W$ and $\Lambda^2(V)$ are simple.  We have
\begin{equation} \label{temple}
    \Lambda^2(V) = \triv^1
    \quad \text{and} \quad
    W = \Ind(L(2\epsilon_1))
    \quad \text{when } N=2.
\end{equation}
\details{
    As a $(\GG_m \rtimes C_2)$-module, we have $W = L_{-4} \oplus L_4$, with the generator of $C_2$ interchanging the two summands.
}
Moreover, when $N \ge 3$, we have
\begin{equation} \label{volley2}
    W \cong L(2 \epsilon_1),
    \qquad
    \Lambda^2(V) \cong
    \begin{cases}
        L(\epsilon_1 + \epsilon_2) & \text{if } N > 3, \\
        L(\epsilon_1) & \text{if } N = 3,
    \end{cases}
    \qquad \text{as $\Spin(V)$-modules}.
\end{equation}

\begin{rem} \label{Delignefail}
    Suppose $d=N$ is odd.  It follows from \cref{Sproject}\cref{Sproject3}, \cref{lemon}, and \cref{Sdub} that $\bF(\pi_0)$ and $\bF(\pi_N)$ are both the projection from $S^{\otimes 2}$ onto its trivial $\Pin(V)$-module summand.  It follows that $\bF(\pi_0 \pi_N) \ne 0$, and so $\pi_0 \pi_N \ne 0$.  This shows that the equality in \cref{Deligne} \emph{fails} when $r = d$ is odd.
\end{rem}

\begin{prop} \label{Vproject}
    If $d \ne 0$, the morphisms
    \begin{equation} \label{Vprojectmor}
        \frac{1}{d}\
        \begin{tikzpicture}[centerzero]
            \draw[vec] (-0.2,-0.35) -- (-0.2,-0.3) arc(180:0:0.2) -- (0.2,-0.35);
            \draw[vec] (-0.2,0.35) -- (-0.2,0.3) arc(180:360:0.2) -- (0.2,0.35);
        \end{tikzpicture}
        \ ,\qquad
        \frac{1}{2}
        \left(\,
            \begin{tikzpicture}[centerzero]
                \draw[vec] (-0.2,-0.35) -- (-0.2,0.35);
                \draw[vec] (0.2,-0.35) -- (0.2,0.35);
            \end{tikzpicture}
            \, +\,
            \begin{tikzpicture}[centerzero]
                \draw[vec] (-0.2,-0.35) -- (0.2,0.35);
                \draw[vec] (0.2,-0.35) -- (-0.2,0.35);
            \end{tikzpicture}
        \, \right)
        \, - \frac{1}{d}\
        \begin{tikzpicture}[centerzero]
            \draw[vec] (-0.2,-0.35) -- (-0.2,-0.3) arc(180:0:0.2) -- (0.2,-0.35);
            \draw[vec] (-0.2,0.35) -- (-0.2,0.3) arc(180:360:0.2) -- (0.2,0.35);
        \end{tikzpicture}
        \ ,\qquad
        \frac{1}{2}
        \left(\,
            \begin{tikzpicture}[centerzero]
                \draw[vec] (-0.2,-0.35) -- (-0.2,0.35);
                \draw[vec] (0.2,-0.35) -- (0.2,0.35);
            \end{tikzpicture}
            \, -\,
            \begin{tikzpicture}[centerzero]
                \draw[vec] (-0.2,-0.35) -- (0.2,0.35);
                \draw[vec] (0.2,-0.35) -- (-0.2,0.35);
            \end{tikzpicture}
        \, \right)
        \quad \in \End_{\SB(d,D;\kappa)}(\Vgo^{\otimes 2})
    \end{equation}
    are orthogonal idempotents.  When $d=N \ge 1$ and $D = \sigma_N 2^n$, their images under the incarnation functor $\bF$ are the projections onto the summands $\triv^0$, $W$, and $\Lambda^2(V)$, respectively, of $V^{\otimes 2}$.
\end{prop}

\begin{proof}
    The proof that these are orthogonal idempotents is a straightforward diagrammatic computation, analogous to the corresponding computation in the Brauer category.  Since the images under $\bF$ of
    \[
        \frac{1}{2}
        \left(\,
            \begin{tikzpicture}[centerzero]
                \draw[vec] (-0.2,-0.35) -- (-0.2,0.35);
                \draw[vec] (0.2,-0.35) -- (0.2,0.35);
            \end{tikzpicture}
            \, +\,
            \begin{tikzpicture}[centerzero]
                \draw[vec] (-0.2,-0.35) -- (0.2,0.35);
                \draw[vec] (0.2,-0.35) -- (-0.2,0.35);
            \end{tikzpicture}
        \, \right)
        \qquad \text{and} \qquad
        \frac{1}{2}
        \left(\,
            \begin{tikzpicture}[centerzero]
                \draw[vec] (-0.2,-0.35) -- (-0.2,0.35);
                \draw[vec] (0.2,-0.35) -- (0.2,0.35);
            \end{tikzpicture}
            \, -\,
            \begin{tikzpicture}[centerzero]
                \draw[vec] (-0.2,-0.35) -- (0.2,0.35);
                \draw[vec] (0.2,-0.35) -- (-0.2,0.35);
            \end{tikzpicture}
        \, \right)
    \]
    are the symmetrizer and antisymmetrizer, respectively, and the first morphism in \cref{Vprojectmor} clearly factors through the trivial module, the final statement in the proposition follows.
\end{proof}

Recall the decomposition of $S \otimes V$ given in \cref{SV}.

\begin{lem} \label{soju}
    If $d \ne 0$, the morphism
    \begin{equation} \label{fork}
        \frac{1}{d}\
        \begin{tikzpicture}[centerzero]
            \draw[spin] (-0.35,0.35) -- (0,0.15) -- (0,-0.15) -- (-0.35,-0.35);
            \draw[vec] (0,0.15) -- (0.35,0.35);
            \draw[vec] (0,-0.15) -- (0.35,-0.35);
        \end{tikzpicture}
        \quad \in \End_{\SB(d,D;\kappa)}(\Sgo \otimes \Vgo)
    \end{equation}
    is an idempotent.  When $d=N \ge 1$ and $D = \sigma_N 2^n$, its image under the incarnation functor $\bF$ is the projection onto the summand $S$ of $S \otimes V$.
\end{lem}

\begin{proof}
    We have
    \[
        \left(
            \frac{1}{d}\
            \begin{tikzpicture}[centerzero]
                \draw[spin] (-0.35,0.35) -- (0,0.15) -- (0,-0.15) -- (-0.35,-0.35);
                \draw[vec] (0,0.15) -- (0.35,0.35);
                \draw[vec] (0,-0.15) -- (0.35,-0.35);
            \end{tikzpicture}
        \right)^{\circ 2}
        = \frac{1}{d^2}\
        \begin{tikzpicture}[centerzero]
            \draw[spin] (-0.35,-0.55) -- (0,-0.35) -- (0,0.35) -- (-0.35,0.55);
            \draw[vec] (0,0.35) -- (0.35,0.55);
            \draw[vec] (0,-0.35) -- (0.35,-0.55);
            \draw[vec] (0,-0.2) arc(-90:90:0.2);
        \end{tikzpicture}
        \overset{\cref{bump}}{=}
        \frac{1}{d}\
        \begin{tikzpicture}[centerzero]
            \draw[spin] (-0.35,0.35) -- (0,0.15) -- (0,-0.15) -- (-0.35,-0.35);
            \draw[vec] (0,0.15) -- (0.35,0.35);
            \draw[vec] (0,-0.15) -- (0.35,-0.35);
        \end{tikzpicture}
        \ .
    \]
    Thus, the morphism \cref{fork} is an idempotent.  When $d=N$ and $D = \sigma_N 2^n$, its image under $\bF$ is a morphism $S \otimes V \to S \to S \otimes V$.  Thus, it is the projection onto the summand $S$ of $S \otimes V$ as long as it is nonzero.  Recalling the trace map of \cref{natal}, we have
    \[
        \Tr
        \left(
            \frac{1}{d}\
            \begin{tikzpicture}[centerzero]
                \draw[spin] (-0.35,0.35) -- (0,0.15) -- (0,-0.15) -- (-0.35,-0.35);
                \draw[vec] (0,0.15) -- (0.35,0.35);
                \draw[vec] (0,-0.15) -- (0.35,-0.35);
            \end{tikzpicture}
            \,
        \right)
        =
        \begin{tikzpicture}[centerzero]
            \draw[spin] (0,-0.2) -- (0,0.2) arc(180:0:0.2) -- (0.4,-0.15) arc(360:180:0.2);
            \draw[vec] (0,0.2) arc(90:-90:0.2);
        \end{tikzpicture}
        \overset{\cref{bump}}{\underset{\cref{dimrel}}{=}} dD 1_\one \ne 0,
    \]
    and the result follows.
\end{proof}

\section{The affine spin Brauer category\label{sec:affine}}

In this section we define an affine version of the spin Brauer category, together with an affine incarnation functor.  This can be thought of as a spin version of the affine Brauer category introduced in \cite{RS19}.

\begin{defin} \label{ABdef}
    For $d,D \in \kk$ and $\kappa \in \{\pm 1\}$, the \emph{affine spin Brauer category} is the strict $\kk$-linear monoidal category $\ASB(d,D;\kappa)$ obtained from $\SB(d,D;\kappa)$ by adjoining two additional generating morphisms
    \[
        \dotstrand{spin} \colon \Sgo \to \Sgo,\qquad
        \dotstrand{vec} \colon \Vgo \to \Vgo,
    \]
    which we call \emph{dots}, subject to the relations
    \begin{gather} \label{dotcross1}
        \begin{tikzpicture}[centerzero]
            \draw[vec] (-0.35,-0.35) -- (0.35,0.35);
            \draw[vec] (0.35,-0.35) -- (-0.35,0.35);
            \singdot{-0.17,-0.17};
        \end{tikzpicture}
        -
        \begin{tikzpicture}[centerzero]
            \draw[vec] (-0.35,-0.35) -- (0.35,0.35);
            \draw[vec] (0.35,-0.35) -- (-0.35,0.35);
            \singdot{0.17,0.17};
        \end{tikzpicture}
        = 2
        \left(\,
            \begin{tikzpicture}[centerzero]
                \draw[vec] (-0.2,-0.35) -- (-0.2,0.35);
                \draw[vec] (0.2,-0.35) -- (0.2,0.35);
            \end{tikzpicture}
            \, -\,
            \begin{tikzpicture}[centerzero]
                \draw[vec] (-0.2,-0.35) -- (-0.2,-0.3) arc(180:0:0.2) -- (0.2,-0.35);
                \draw[vec] (-0.2,0.35) -- (-0.2,0.3) arc(180:360:0.2) -- (0.2,0.35);
            \end{tikzpicture}
        \, \right)
        \ ,\qquad
        \begin{tikzpicture}[centerzero]
            \draw[spin] (-0.35,-0.35) -- (0.35,0.35);
            \draw[spin] (0.35,-0.35) -- (-0.35,0.35);
            \singdot{-0.17,-0.17};
        \end{tikzpicture}
        -
        \begin{tikzpicture}[centerzero]
            \draw[spin] (-0.35,-0.35) -- (0.35,0.35);
            \draw[spin] (0.35,-0.35) -- (-0.35,0.35);
            \singdot{0.17,0.17};
        \end{tikzpicture}
        = \frac{1}{8}
        \left(
            \begin{tikzpicture}[centerzero]
                \draw[spin] (-0.35,-0.35) -- (0.35,0.35);
                \draw[spin] (0.35,-0.35) -- (-0.35,0.35);
                \draw[vec] (-0.25,-0.25) -- (-0.25,0.25);
                \draw[vec] (0.25,-0.25) -- (0.25,0.25);
            \end{tikzpicture}
            -
            \begin{tikzpicture}[centerzero]
                \draw[spin] (-0.35,-0.35) -- (0.35,0.35);
                \draw[spin] (0.35,-0.35) -- (-0.35,0.35);
                \draw[vec] (-0.25,-0.25) -- (0.25,-0.25);
                \draw[vec] (-0.25,0.25) -- (0.25,0.25);
            \end{tikzpicture}
        \right)
        ,\
        \\ \label{dotcross2}
        \begin{tikzpicture}[centerzero]
            \draw[spin] (-0.35,-0.35) -- (0.35,0.35);
            \draw[vec] (0.35,-0.35) -- (-0.35,0.35);
            \singdot{-0.17,-0.17};
        \end{tikzpicture}
        -
        \begin{tikzpicture}[centerzero]
            \draw[spin] (-0.35,-0.35) -- (0.35,0.35);
            \draw[vec] (0.35,-0.35) -- (-0.35,0.35);
            \singdot{0.17,0.17};
        \end{tikzpicture}
        =
        \begin{tikzpicture}[centerzero]
            \draw[spin] (-0.35,-0.35) -- (0.35,0.35);
            \draw[vec] (0.35,-0.35) -- (-0.35,0.35);
        \end{tikzpicture}
        - \kappa
        \begin{tikzpicture}[centerzero]
            \draw[spin] (0.35,0.35) -- (0,0.15) -- (0,-0.15) -- (-0.35,-0.35);
            \draw[vec] (0,0.15) -- (-0.35,0.35);
            \draw[vec] (0,-0.15) -- (0.35,-0.35);
        \end{tikzpicture}
        \ ,\qquad
        \begin{tikzpicture}[centerzero]
            \draw[vec] (-0.35,-0.35) -- (0.35,0.35);
            \draw[spin] (0.35,-0.35) -- (-0.35,0.35);
            \singdot{-0.17,-0.17};
        \end{tikzpicture}
        -
        \begin{tikzpicture}[centerzero]
            \draw[vec] (-0.35,-0.35) -- (0.35,0.35);
            \draw[spin] (0.35,-0.35) -- (-0.35,0.35);
            \singdot{0.17,0.17};
        \end{tikzpicture}
        =
        \begin{tikzpicture}[centerzero]
            \draw[vec] (-0.35,-0.35) -- (0.35,0.35);
            \draw[spin] (0.35,-0.35) -- (-0.35,0.35);
        \end{tikzpicture}
        - \kappa
        \begin{tikzpicture}[centerzero]
            \draw[spin] (-0.35,0.35) -- (0,0.15) -- (0,-0.15) -- (0.35,-0.35);
            \draw[vec] (0,0.15) -- (0.35,0.35);
            \draw[vec] (0,-0.15) -- (-0.35,-0.35);
        \end{tikzpicture}
        \ ,
        \\ \label{dotcap}
        \begin{tikzpicture}[anchorbase]
            \draw[spin] (-0.2,-0.3) -- (-0.2,-0.1) arc(180:0:0.2) -- (0.2,-0.3);
            \singdot{-0.2,-0.1};
        \end{tikzpicture}
        \ = -\
        \begin{tikzpicture}[anchorbase]
            \draw[spin] (-0.2,-0.3) -- (-0.2,-0.1) arc(180:0:0.2) -- (0.2,-0.3);
            \singdot{0.2,-0.1};
        \end{tikzpicture}
        \ ,\qquad
        \begin{tikzpicture}[anchorbase]
            \draw[vec] (-0.2,-0.3) -- (-0.2,-0.1) arc(180:0:0.2) -- (0.2,-0.3);
            \singdot{-0.2,-0.1};
        \end{tikzpicture}
        \ = -\
        \begin{tikzpicture}[anchorbase]
            \draw[vec] (-0.2,-0.3) -- (-0.2,-0.1) arc(180:0:0.2) -- (0.2,-0.3);
            \singdot{0.2,-0.1};
        \end{tikzpicture}
        \ ,
        \\ \label{dotvertex}
        \begin{tikzpicture}[centerzero]
            \draw[spin] (0,0.4) -- (0,0) -- (0.35,-0.35);
            \draw[vec] (-0.35,-0.35) -- (0,0);
            \singdot{0,0.2};
        \end{tikzpicture}
        =
        \begin{tikzpicture}[centerzero]
            \draw[spin] (0,0.4) -- (0,0) -- (0.35,-0.35);
            \draw[vec] (-0.35,-0.35) -- (0,0);
            \singdot{-0.175,-0.175};
        \end{tikzpicture}
        +
        \begin{tikzpicture}[centerzero]
            \draw[spin] (0,0.4) -- (0,0) -- (0.35,-0.35);
            \draw[vec] (-0.35,-0.35) -- (0,0);
            \singdot{0.175,-0.175};
        \end{tikzpicture}
        \ .
    \end{gather}
    Let $\overline{\ASB}(d,D;\kappa)$ denote the quotient of $\ASB(d,D;\kappa)$ by \cref{extra}.
\end{defin}

\begin{prop}
    The following relations hold in $\ASB(d,D;\kappa)$:
    \begin{gather} \label{dotcross3}
        \begin{tikzpicture}[centerzero]
            \draw[vec] (-0.35,-0.35) -- (0.35,0.35);
            \draw[vec] (0.35,-0.35) -- (-0.35,0.35);
            \singdot{-0.17,0.17};
        \end{tikzpicture}
        -
        \begin{tikzpicture}[centerzero]
            \draw[vec] (-0.35,-0.35) -- (0.35,0.35);
            \draw[vec] (0.35,-0.35) -- (-0.35,0.35);
            \singdot{0.17,-0.17};
        \end{tikzpicture}
        \ = 2
        \left(\,
            \begin{tikzpicture}[centerzero]
                \draw[vec] (-0.2,-0.35) -- (-0.2,0.35);
                \draw[vec] (0.2,-0.35) -- (0.2,0.35);
            \end{tikzpicture}
            \, -\,
            \begin{tikzpicture}[centerzero]
                \draw[vec] (-0.2,-0.35) -- (-0.2,-0.3) arc(180:0:0.2) -- (0.2,-0.35);
                \draw[vec] (-0.2,0.35) -- (-0.2,0.3) arc(180:360:0.2) -- (0.2,0.35);
            \end{tikzpicture}
        \, \right)
        ,\qquad
        \begin{tikzpicture}[centerzero]
            \draw[spin] (-0.35,-0.35) -- (0.35,0.35);
            \draw[spin] (0.35,-0.35) -- (-0.35,0.35);
            \singdot{-0.17,0.17};
        \end{tikzpicture}
        -
        \begin{tikzpicture}[centerzero]
            \draw[spin] (-0.35,-0.35) -- (0.35,0.35);
            \draw[spin] (0.35,-0.35) -- (-0.35,0.35);
            \singdot{0.17,-0.17};
        \end{tikzpicture}
        = \frac{1}{8}
        \left(
            \begin{tikzpicture}[centerzero]
                \draw[spin] (-0.35,-0.35) -- (0.35,0.35);
                \draw[spin] (0.35,-0.35) -- (-0.35,0.35);
                \draw[vec] (-0.25,-0.25) -- (-0.25,0.25);
                \draw[vec] (0.25,-0.25) -- (0.25,0.25);
            \end{tikzpicture}
            -
            \begin{tikzpicture}[centerzero]
                \draw[spin] (-0.35,-0.35) -- (0.35,0.35);
                \draw[spin] (0.35,-0.35) -- (-0.35,0.35);
                \draw[vec] (-0.25,-0.25) -- (0.25,-0.25);
                \draw[vec] (-0.25,0.25) -- (0.25,0.25);
            \end{tikzpicture}
        \right)
        ,
        \\ \label{dotcross4}
        \begin{tikzpicture}[centerzero]
            \draw[vec] (-0.35,-0.35) -- (0.35,0.35);
            \draw[spin] (0.35,-0.35) -- (-0.35,0.35);
            \singdot{-0.17,0.17};
        \end{tikzpicture}
        -
        \begin{tikzpicture}[centerzero]
            \draw[vec] (-0.35,-0.35) -- (0.35,0.35);
            \draw[spin] (0.35,-0.35) -- (-0.35,0.35);
            \singdot{0.17,-0.17};
        \end{tikzpicture}
        =
        \begin{tikzpicture}[centerzero]
            \draw[vec] (-0.35,-0.35) -- (0.35,0.35);
            \draw[spin] (0.35,-0.35) -- (-0.35,0.35);
        \end{tikzpicture}
        - \kappa
        \begin{tikzpicture}[centerzero]
            \draw[spin] (-0.35,0.35) -- (0,0.15) -- (0,-0.15) -- (0.35,-0.35);
            \draw[vec] (0,0.15) -- (0.35,0.35);
            \draw[vec] (0,-0.15) -- (-0.35,-0.35);
        \end{tikzpicture}
        \ ,\qquad
        \begin{tikzpicture}[centerzero]
            \draw[spin] (-0.35,-0.35) -- (0.35,0.35);
            \draw[vec] (0.35,-0.35) -- (-0.35,0.35);
            \singdot{-0.17,0.17};
        \end{tikzpicture}
        -
        \begin{tikzpicture}[centerzero]
            \draw[spin] (-0.35,-0.35) -- (0.35,0.35);
            \draw[vec] (0.35,-0.35) -- (-0.35,0.35);
            \singdot{0.17,-0.17};
        \end{tikzpicture}
        =
        \begin{tikzpicture}[centerzero]
            \draw[spin] (-0.35,-0.35) -- (0.35,0.35);
            \draw[vec] (0.35,-0.35) -- (-0.35,0.35);
        \end{tikzpicture}
        - \kappa
        \begin{tikzpicture}[centerzero]
            \draw[spin] (0.35,0.35) -- (0,0.15) -- (0,-0.15) -- (-0.35,-0.35);
            \draw[vec] (0,0.15) -- (-0.35,0.35);
            \draw[vec] (0,-0.15) -- (0.35,-0.35);
        \end{tikzpicture}
        \ ,
        \\ \label{dotcup}
        \begin{tikzpicture}[anchorbase]
            \draw[spin] (-0.2,0.3) -- (-0.2,0.1) arc(180:360:0.2) -- (0.2,0.3);
            \singdot{-0.2,0.1};
        \end{tikzpicture}
        \ = -\
        \begin{tikzpicture}[anchorbase]
            \draw[spin] (-0.2,0.3) -- (-0.2,0.1) arc(180:360:0.2) -- (0.2,0.3);
            \singdot{0.2,0.1};
        \end{tikzpicture}
        \ ,\qquad
        \begin{tikzpicture}[anchorbase]
            \draw[vec] (-0.2,0.3) -- (-0.2,0.1) arc(180:360:0.2) -- (0.2,0.3);
            \singdot{-0.2,0.1};
        \end{tikzpicture}
        \ = -\
        \begin{tikzpicture}[anchorbase]
            \draw[vec] (-0.2,0.3) -- (-0.2,0.1) arc(180:360:0.2) -- (0.2,0.3);
            \singdot{0.2,0.1};
        \end{tikzpicture}
        \ ,
        \\ \label{dotvertex2}
        \begin{tikzpicture}[centerzero]
            \draw[spin] (0,0.4) -- (0,0) -- (-0.35,-0.35);
            \draw[vec] (0.35,-0.35) -- (0,0);
            \singdot{0,0.2};
        \end{tikzpicture}
        =
        \begin{tikzpicture}[centerzero]
            \draw[spin] (0,0.4) -- (0,0) -- (-0.35,-0.35);
            \draw[vec] (0.35,-0.35) -- (0,0);
            \singdot{-0.175,-0.175};
        \end{tikzpicture}
        +
        \begin{tikzpicture}[centerzero]
            \draw[spin] (0,0.4) -- (0,0) -- (-0.35,-0.35);
            \draw[vec] (0.35,-0.35) -- (0,0);
            \singdot{0.175,-0.175};
        \end{tikzpicture}
        \ .
    \end{gather}
\end{prop}

\begin{proof}
    Relations \cref{dotcross3,dotcup} follow from rotating \cref{dotcross1,dotcap} using cups and caps.  The first relation in \cref{dotcross4} follows from the first relation in \cref{dotcross2} after composing on the top and bottom with $\crossmor{vec}{spin}$.  The second relation in \cref{dotcross4} follows similarly from the second relation in \cref{dotcross2}.

    To prove \cref{dotvertex2}, we compute
    \[
        \kappa
        \begin{tikzpicture}[centerzero]
            \draw[spin] (0,0.4) -- (0,0) -- (-0.35,-0.35);
            \draw[vec] (0.35,-0.35) -- (0,0);
            \singdot{0,0.2};
        \end{tikzpicture}
        \overset{\cref{lobster}}{=}
        \begin{tikzpicture}[anchorbase]
            \draw[vec] (0.3,-0.7) to [out=135,in=down] (-0.2,-0.25) to[out=up,in=-135] (0,0);
            \draw[spin] (-0.3,-0.7) to[out=45,in=down] (0.2,-0.25) to[out=up,in=-45] (0,0) -- (0,0.4);
            \singdot{0,0.2};
        \end{tikzpicture}
        \overset{\cref{dotvertex}}{=}
        \begin{tikzpicture}[anchorbase]
            \draw[vec] (0.3,-0.7) to [out=135,in=down] (-0.2,-0.25) to[out=up,in=-135] (0,0);
            \draw[spin] (-0.3,-0.7) to[out=45,in=down] (0.2,-0.25) to[out=up,in=-45] (0,0) -- (0,0.4);
            \singdot{-0.2,-0.25};
        \end{tikzpicture}
        +
        \begin{tikzpicture}[anchorbase]
            \draw[vec] (0.3,-0.7) to [out=135,in=down] (-0.2,-0.25) to[out=up,in=-135] (0,0);
            \draw[spin] (-0.3,-0.7) to[out=45,in=down] (0.2,-0.25) to[out=up,in=-45] (0,0) -- (0,0.4);
            \singdot{0.2,-0.25};
        \end{tikzpicture}
        \overset{\cref{dotcross2}}{\underset{\cref{dotcross4}}{=}}
        \begin{tikzpicture}[anchorbase]
            \draw[vec] (0.3,-0.7) to [out=135,in=down] (-0.2,-0.25) to[out=up,in=-135] (0,0);
            \draw[spin] (-0.3,-0.7) to[out=45,in=down] (0.2,-0.25) to[out=up,in=-45] (0,0) -- (0,0.4);
            \singdot{0.15,-0.6};
        \end{tikzpicture}
        +
        \begin{tikzpicture}[anchorbase]
            \draw[vec] (0.3,-0.7) to [out=135,in=down] (-0.2,-0.25) to[out=up,in=-135] (0,0);
            \draw[spin] (-0.3,-0.7) to[out=45,in=down] (0.2,-0.25) to[out=up,in=-45] (0,0) -- (0,0.4);
            \singdot{-0.15,-0.6};
        \end{tikzpicture}
        \overset{\cref{lobster}}{=}
        \kappa \left(
        \begin{tikzpicture}[centerzero]
            \draw[spin] (0,0.4) -- (0,0) -- (-0.35,-0.35);
            \draw[vec] (0.35,-0.35) -- (0,0);
            \singdot{0.175,-0.175};
        \end{tikzpicture}
        +
        \begin{tikzpicture}[centerzero]
            \draw[spin] (0,0.4) -- (0,0) -- (-0.35,-0.35);
            \draw[vec] (0.35,-0.35) -- (0,0);
            \singdot{-0.175,-0.175};
        \end{tikzpicture}
        \right)
        . \qedhere
    \]
\end{proof}

The symmetries \cref{opflip,revflip} can be extended to $\ASB(d,D;\kappa)$.  Precisely, we have an isomorphism of monoidal categories
\begin{equation}
    \ASB(d,D;\kappa) \to \ASB(d,D;\kappa)^\op
\end{equation}
that is the identity on objects and reflects morphisms in the horizontal axis.  We also have an isomorphism of monoidal categories
\begin{equation}
    \ASB(d,D;\kappa) \to \ASB(d,D;\kappa)^\rev
\end{equation}
that is the identity on objects and, on morphisms, reflects diagrams in the vertical axis and multiplies dots by $-1$.

Our goal in the remainder of this section is to define an affine version of the incarnation functor of \cref{sec:incarnation}.  Since our construction will be based on the Lie algebra $\fso(V)$, we assume throughout that $N \ge 2$ and we work over the ground field $\kk = \C$.  However, see \cref{lowaffine} for the cases $N=0$ and $N=1$.

Let $\bB_{\fso(V)}$ be a basis of $\fso(V)$ and let $\{X^\vee : X \in \bB_{\fso(V)}\}$ denote the dual basis with respect to the symmetric bilinear form
\[
    \langle X,Y \rangle = \tfrac{1}{2} \tr(XY),\qquad X,Y \in \fso(V),
\]
where $\tr$ denotes the usual trace on the space of linear operators on $V$.  We have
\begin{align*}
    \langle M_{e_i,e_j}, M_{e_k,e_l} \rangle
    &= \frac12\sum_{m=1}^N \langle M_{e_i,e_j} M_{e_k,e_l} e_m, e_m \rangle
    \\
    &\overset{\mathclap{\cref{thai}}}{=}\ \frac12\sum_{m=1}^N \langle M_{e_i,e_j} ( \delta_{lm} e_k - \delta_{km} e_l ), e_m \rangle
    \\
    &\overset{\mathclap{\cref{thai}}}{=}\ \frac12\sum_{m=1}^N \langle \delta_{jk} \delta_{lm} e_i - \delta_{ik} \delta_{lm} e_j - \delta_{jl} \delta_{km} e_i + \delta_{il} \delta_{km} e_j, e_m \rangle
    \\
    &= \delta_{jk} \delta_{il} - \delta_{ik} \delta_{jl}.
\end{align*}
Thus, if we take the basis
\[
    \bB_{\fso(V)} = \{M_{e_i,e_j} : 1 \le i < j \le N\},
\]
then
\[
    M_{e_i,e_j}^\vee = M_{e_j,e_i} = - M_{e_i,e_j},\qquad
    1 \le i < j \le N.
\]

Define
\begin{gather} \label{Omega}
    \Omega
    = \sum_{X \in \bB_{\fso(V)}} X \otimes X^\vee
    = \sum_{1 \le i < j \le N} M_{e_i,e_j} \otimes M_{e_j,e_i}
    \in \fso(V) \otimes \fso(V),
    \\
    C
    = \sum_{X \in \bB_{\fso(V)}} X X^\vee
    = \sum_{1 \le i < j \le N} M_{e_i,e_j} M_{e_j,e_i} \in U(\fso(V)).
\end{gather}
The elements $\Omega$ and $C$ are both independent of the chosen basis $\bB_{\fso(V)}$.  Note that $C$ is the Casimir element and we have
\begin{equation} \label{koala}
    \Omega = \tfrac{1}{2}(\Delta(C) - C \otimes 1 - 1 \otimes C),
\end{equation}
where $\Delta$ is the standard coproduct on $\fso(V)$.
\details{
    We have
    \begin{align*}
        \Delta(C)
        &= \sum_{X \in \bB_{\fso(V)}} \Delta(X) \Delta(X^\vee)
        \\
        &= \sum_{X \in \bB_{\fso(V)}} (X \otimes 1 + 1 \otimes X)(X^\vee \otimes 1 + 1 \otimes X^\vee)
        \\
        &= \sum_{X \in \bB_{\fso(V)}} (X X^\vee \otimes 1 + X \otimes X^\vee + X^\vee \otimes X + 1 \otimes X X^\vee)
        \\
        &= C \otimes 1 + 1 \otimes C + 2 \Omega.
    \end{align*}
}
Define
\begin{equation} \label{flatwhite}
    \dotop := 2 \Omega + C \otimes 1
    = \Delta(C) - 1 \otimes C.
\end{equation}

The nondegenerate form $\langle\cdot,\cdot\rangle$ remains nondegenerate when restricted to $\mathfrak{h}$, hence induces a pairing $\langle\cdot,\cdot\rangle \colon \mathfrak{h}^\ast\times \h^\ast\to \C$, which we denote by the same symbol.

\begin{lem} \label{Cact}
    The element $C$ acts on the simple $\Spin(V)$-module $L(\lambda)$ of highest weight $\lambda$ as $\langle \lambda, \lambda + 2\rho\rangle$, where
    \begin{equation} \label{rho}
        \rho = \frac{1}{2} \sum_{i=1}^n \left( N - 2i \right) \epsilon_i.
    \end{equation}
\end{lem}

\begin{proof}
    This is well known.  See, for example, \cite[Prop.~11.36]{Car05}.
\end{proof}

\begin{cor} \label{bass}
    The action of $C$ commutes with the action of $\Pin(V)$.
\end{cor}

\begin{proof}
    Note that, if $N=2n$, and $\tilde{\lambda}$ is defined as in \cref{reflect}, then $\langle \lambda, \lambda + 2 \rho \rangle = \langle \tilde{\lambda}, \tilde{\lambda} + 2 \rho \rangle$.  Thus, $C$ acts on the simple $\Pin(V)$-module $\Ind(L(\lambda))$ as $\langle \lambda, \lambda + 2 \rho \rangle$.  Then the corollary follows from the fact that $\Pin(V)\md$ is a semisimple category.
\end{proof}

\details{
    One can also prove \cref{bass} by noting that $C$ is invariant under any automorphism of $\fg$.  In particular, it is invariant under the adjoint action of $\Pin(V)$.
}

\begin{lem}
    We have
    \begin{equation} \label{stand}
        \beta^2 (x \otimes y) = \left( N - 8 \Omega \right) (x \otimes y)
        \qquad \text{for all } x,y \in S.
    \end{equation}
\end{lem}

\begin{proof}
    Throughout this proof, we view all elements of $\fso(V) \otimes \fso(V)$ as operators on $S \otimes S$.  Then we have, via \cref{Canberra}, $M_{e_i,e_j} = \frac{1}{2} e_i e_j$ for $i\neq j$.  Thus,
    \[
        \Omega = \frac{1}{4} \sum_{1 \le i<j \le N}e_i e_j \otimes e_j e_i.
    \]
    On the other hand,
    \[
        \beta^2 = \sum_{i,j=1}^N e_i e_j \otimes e_i e_j
        \overset{\cref{ecomm}}{=} N - 2 \sum_{1 \le i<j \le N} e_i e_j \otimes e_j e_i
        = N - 8 \Omega.
        \qedhere
    \]
\end{proof}

\begin{lem} \label{ibis}
    The element $C$ acts as
    \begin{enumerate}
        \item $k(N-k)$ on $L(\epsilon_1 + \dotsb + \epsilon_k)$, $0 \le k \le n$;
        \item $\frac{N(N-1)}{8}$ on the spin representation $S$;
        \item $2N$ on $L(2 \epsilon_1)$;
        \item $\frac{N^2}{4}$ on $L(\epsilon_1 + \dotsb + \epsilon_{n-1} - \epsilon_n)$ when $N=2n \ge 4$ (i.e.\ type $D_n$).
        \item $\frac{N(N+7)}{8}$ on $L \left( \frac{3}{2} \epsilon_1 + \frac{1}{2} \epsilon_2 + \dotsb + \frac{1}{2} \epsilon_n \right)$, $n \ge 2$.
        \item $\frac{N(N+7)}{8}$ on $L \left( \frac{3}{2} \epsilon_1 + \frac{1}{2} \epsilon_2 + \dotsb + \frac{1}{2} \epsilon_{n-1} - \frac{1}{2} \epsilon_n \right)$ when $N=2n \ge 4$ (i.e.\ type $D_n$).
    \end{enumerate}
\end{lem}

\begin{proof}
    These are all direct computations using \cref{Cact}.  First note that $\epsilon_1,\dotsc,\epsilon_n$ is an orthonormal basis of $\fh^*$.  (It is dual to the orthonormal basis $A_{11},\dotsc,A_{nn}$ of $\fh$.)
    \begin{enumerate}[wide]
        \item We have
            \[
                \left\langle \sum_{i=1}^k \epsilon_i, \sum_{i=1}^k \epsilon_i + 2 \rho \right\rangle
                = k + \sum_{i=1}^k (N-2i)
                = k+kN-k(k+1)
                = k(N-k).
            \]

        \item In type $D_n$, so that $N=2n$, we have $S^\pm = L \big(\frac{1}{2}( \epsilon_1 + \dotsb + \epsilon_{n-1} \pm \epsilon_n) \big)$.  Then we compute
            \[
                \left\langle \frac{1}{2} \sum_{i=1}^{n-1} \epsilon_i \pm \frac{1}{2} \epsilon_n, \frac{1}{2} \sum_{i=1}^{n-1} \epsilon_i \pm \frac{1}{2} \epsilon_n + 2 \rho \right\rangle
                = \frac{n}{4} + \frac{1}{2} \sum_{i=1}^{n-1} (N-2i)
                = \frac{N(N-1)}{8}.
            \]
            In type $B_n$, so that $N=2n+1$, we have $S = L \big( \frac{1}{2}(\epsilon_1 + \dotsb + \epsilon_n) \big)$, and we compute
            \[
                \left\langle \frac{1}{2} \sum_{i=1}^n \epsilon_i, \frac{1}{2} \sum_{i=1}^n \epsilon_i + 2 \rho \right\rangle
                = \frac{n}{4} + \frac{1}{2} \sum_{i=1}^n (N-2i)
                = \frac{n}{4} + \frac{nN-n(n+1)}{2}
                = \frac{N(N-1)}{8}.
            \]

        \item We compute
            \[
                \langle 2 \epsilon_1, 2 \epsilon_1 + 2\rho \rangle
                = 4 + 2(N-2) = 2N.
            \]

        \item We compute
            \[
                \left\langle \sum_{i=1}^{n-1} \epsilon_i - \epsilon_n, \sum_{i=1}^{n-1} \epsilon_i - \epsilon_n + 2 \rho \right\rangle
                = n + \sum_{i=1}^{n-1} (N-2i)
                = n + (n-1)N - n(n-1)
                = \frac{N^2}{4}.
            \]

        \item We compute
            \begin{multline*}
                \left\langle \frac{3}{2} \epsilon_1 + \frac{1}{2} \sum_{i=2}^n \epsilon_i, \frac{3}{2} + \frac{1}{2} \sum_{i=2}^n \epsilon_i + 2 \rho \right\rangle
                = \frac{n+8}{4} + \frac{3}{2}(N-2) + \frac{1}{2} \sum_{i=2}^n (N-2i)
                \\
                = \frac{2Nn-2n^2+4N-n}{4}.
            \end{multline*}
            When, $N=2n+1$, we have
            \[
                2Nn-2n^2+4N-n
                = (2n+1)(n+4)
                = \frac{N(N+7)}{2}.
            \]
            On the other hand, when $N=2n$, we have
            \[
                2Nn-2n^2+4N-n
                = n(2n+7)
                = \frac{N(N+7)}{2}.
            \]

        \item This computation is almost identical to the previous one, using the fact that the $\epsilon_n$ component of $\rho$ is zero when $N=2n$. \qedhere
    \end{enumerate}
\end{proof}

\details{
    As a consistency check, we compute the action of $C$ on the spin representations in another way.  For $v \in V$, we have
    \[
        M_{e_i,e_j}v
        \overset{\cref{Canberra}}{=} \frac{1}{4} (e_i e_j - e_j e_i) v
        \overset{\cref{ecomm}}{=} \frac{1}{2} e_i e_j v,\qquad
        i \ne j.
    \]
    Thus,
    \[
        Cv
        = \frac{1}{4} \sum_{1 \le i <j \le N} e_i e_j e_j e_i v
        \overset{\cref{ecomm}}{=} \frac{N(N-1)}{8} v,\qquad
        v \in V.
    \]
}

The following lemma will play a key role in our proof that the affine incarnation functor satisfies the dot-crossing relations \cref{dotcross1,dotcross2}.  It will describe the image of
\[
    \begin{tikzpicture}[centerzero]
        \draw[any] (-0.2,-0.4) -- (-0.2,0.4);
        \draw[any] (0.2,-0.4) -- (0.2,0.4);
        \singdot{-0.2,0};
    \end{tikzpicture}
    \ -\
    \begin{tikzpicture}[centerzero]
        \draw[any] (0.2,-0.4) to[out=135,in=down] (-0.15,0) to[out=up,in=225] (0.2,0.4);
        \draw[any] (-0.2,-0.4) to[out=45,in=down] (0.15,0) to[out=up,in=-45] (-0.2,0.4);
        \singdot{0.15,0};
    \end{tikzpicture}
\]
under our affine incarnation functor.

\begin{lem} \label{enmore}
    For any $M_1,M_2,M_3 \in \fso(V)\md$, we have
    \begin{equation}
        (1 \otimes \Delta)(\dotop) - (\flip \otimes 1) \circ (1 \otimes \dotop) \circ (\flip \otimes 1)
        = 2 \Omega \otimes 1
        \quad \text{as operators on } M_1 \otimes M_2 \otimes M_3,
    \end{equation}
    where $\Delta$ is the usual coproduct of $\fso(V)$ given by $\Delta(X) = X \otimes 1 + 1 \otimes X$.
\end{lem}

\begin{proof}
    For $m_1 \in M_1$, $m_2 \in M_2$, $m_3 \in M_3$, we have
    \begin{multline*}
        \big( (1 \otimes \Delta)(\dotop) \big) (m_1 \otimes m_2 \otimes m_3)
        \\
        = \sum_{X \in \bB_{\fso(V)}} (2 X m_1 \otimes X^\vee m_2 \otimes m_3 + 2X m_1 \otimes m_2 \otimes X^\vee m_3 + X X^\vee m_1 \otimes m_2 \otimes m_3)
    \end{multline*}
    and
    \begin{multline*}
        (\flip \otimes 1) \circ (1 \otimes \dotop) \circ (\flip \otimes 1) (m_1 \otimes m_2 \otimes m_3)
        \\
        = \sum_{X \in \bB_{\fso(V)}} (2 X m_1 \otimes m_2 \otimes X^\vee m_3 + XX^\vee m_1 \otimes m_2 \otimes m_3).
    \end{multline*}
    Subtracting these two sums proves the lemma.
\end{proof}

For a $\kk$-linear category $\cC$, let $\cEnd_\kk(\cC)$ denote the strict monoidal category of $\kk$-linear endofunctors and natural transformations.  An \emph{action} of a monoidal category $\cD$ on a category $\cC$ is a monoidal functor $\cD \to \cEnd_\kk(\cC)$.  It follows immediately from \cref{incarnation} that $\SB(V)$ acts on $\Group(V)\md$ via
\[
    X \mapsto \bF(X) \otimes -,\quad
    f \mapsto \bF(f) \otimes -,
\]
for objects $X$ and morphisms $f$ in $\SB(V)$.  The following result extends this action to
\begin{equation}
    \ASB(V) := \ASB(N, \sigma_N 2^n; \kappa_N).
\end{equation}
Let $\overline{\ASB}(V)$ denote the quotient of $\ASB(V)$ by \cref{extra}.

\begin{theo} \label{affinc}
    There is a unique monoidal functor $\bAF \colon \ASB(V) \to \cEnd_\C(\Group(V)\md)$ given on objects by $\Sgo \mapsto S \otimes -$, $\Vgo \mapsto V \otimes -$, and on morphisms by
    \begin{gather}
        f \mapsto \bF(f) \otimes -,\qquad
        f \in \left\{ \capmor{any}\, ,\ \cupmor{any}\, ,\ \crossmor{any}{any}\, ,\ \mergemor{vec}{spin}{spin} \right\},
    \end{gather}
    and $\bAF \left( \dotstrand{spin} \right) \colon S \otimes - \to S \otimes -$, $\bAF \left( \dotstrand{vec} \right) \colon V \otimes - \to V \otimes -$ are the natural transformations with components
    \begin{align*}
        \bAF \left(\, \dotstrand{spin}\, \right)_M &\colon S \otimes M \to S \otimes M,&
        x \otimes m \mapsto \dotop (x \otimes m),
        \\
        \bAF \left(\, \dotstrand{vec}\, \right)_M &\colon V \otimes M \to V \otimes M,&
        v \otimes m \mapsto \dotop (v \otimes m),
    \end{align*}
    for $M \in \fso(V)\md$,  where $\dotop$ is the element defined in \cref{flatwhite}.  The functor $\bAF$ factors through $\overline{\ASB}(V)$.
\end{theo}

\begin{proof}
    When $N$ is odd, the functor $\bAF$ factors respects the relation \cref{extra} since $\bF$ does.  It follows from \cref{bass} that $\bAF \left(\, \dotstrand{spin}\, \right)$ is a natural transformation of the functor $S \otimes -$ and that $\bAF  \left(\, \dotstrand{vec}\, \right)$ is a natural transformation of the functor $V \otimes -$.  Thus, it remains to verify that $\bAF$ respects the relations \cref{dotcross1,dotcross2,dotcap,dotvertex}.  Throughout, $M$ will denote an arbitrary object in $\Group(V)\md$.

    \medskip

    \emph{First relation in \cref{dotcross1}.}
    Composing on the top of the first relation in \cref{dotcross1} with the invertible morphism $\crossmor{vec}{vec}$, then using the fifth relation in \cref{brauer}, we see that the first relation in \cref{dotcross1} is equivalent to
    \begin{equation} \label{rooVV}
        \begin{tikzpicture}[centerzero]
            \draw[vec] (-0.2,-0.4) -- (-0.2,0.4);
            \draw[vec] (0.2,-0.4) -- (0.2,0.4);
            \singdot{-0.2,0};
        \end{tikzpicture}
        \ -\
        \begin{tikzpicture}[centerzero]
            \draw[vec] (0.2,-0.4) to[out=135,in=down] (-0.15,0) to[out=up,in=225] (0.2,0.4);
            \draw[vec] (-0.2,-0.4) to[out=45,in=down] (0.15,0) to[out=up,in=-45] (-0.2,0.4);
            \singdot{0.15,0};
        \end{tikzpicture}
        \ = 2
        \left(\,
        \begin{tikzpicture}[centerzero]
            \draw[vec] (-0.2,-0.35) -- (0.2,0.35);
            \draw[vec] (0.2,-0.35) -- (-0.2,0.35);
        \end{tikzpicture}
        \, -\,
        \begin{tikzpicture}[centerzero]
            \draw[vec] (-0.2,-0.35) -- (-0.2,-0.3) arc(180:0:0.2) -- (0.2,-0.35);
            \draw[vec] (-0.2,0.35) -- (-0.2,0.3) arc(180:360:0.2) -- (0.2,0.35);
        \end{tikzpicture}
        \, \right).
    \end{equation}
    By \cref{enmore}, the image under $\bAF$ of the left-hand side of \cref{rooVV} is the natural endomorphism of the functor $V \otimes V \otimes -$ given by $2 \Omega \otimes -$.  Recall the decompositions \cref{volley1,volley2}.  Let $1_\lambda \colon V^{\otimes 2} \to V^{\otimes 2}$ denote the projection onto the summand isomorphic to $L(\lambda)$ as a $\Spin(V)$-module.  Recall also that $V = L(\epsilon_1)$.  When $N>3$, it follows from \cref{koala,ibis} that $2\Omega$ acts on $V^{\otimes 2}$ as
    \begin{align*}
        2 \Omega(u \otimes v)
        &= C(u \otimes v) - Cu \otimes v - u \otimes Cv
        \\
        &= (2N 1_{2 \epsilon_1} + 2(N-2) 1_{\epsilon_1 + \epsilon_2})(u \otimes v) - (N-1) (u \otimes v) - (N-1) (u \otimes v)
        \\
        &= \big( 2(1-N) 1_0 + 2 ( 1_{2 \epsilon_1} - 1_{\epsilon_1 + \epsilon_2}) \big) (u \otimes v).
    \end{align*}
    On the other hand, since $d=N$, it follows from \cref{Vproject} that
    \[
        \bAF
        \left(\,
            \begin{tikzpicture}[centerzero]
                \draw[vec] (-0.2,-0.35) -- (0.2,0.35);
                \draw[vec] (0.2,-0.35) -- (-0.2,0.35);
            \end{tikzpicture}
            \, -\,
            \begin{tikzpicture}[centerzero]
                \draw[vec] (-0.2,-0.35) -- (-0.2,-0.3) arc(180:0:0.2) -- (0.2,-0.35);
                \draw[vec] (-0.2,0.35) -- (-0.2,0.3) arc(180:360:0.2) -- (0.2,0.35);
            \end{tikzpicture}
        \, \right)
        = \flip - N 1_0
        = (1-N) 1_0 + 1_{2 \epsilon_1} - 1_{\epsilon_1 + \epsilon_2}.
    \]
    The cases $N=2$ and $N=3$ are analogous.
    \details{
        Suppose $N=3$.  It follows from \cref{koala,ibis} that $2\Omega$ acts on $V^{\otimes 2}$ as
        \begin{align*}
            2 \Omega(u \otimes v)
            &= C(u \otimes v) - Cu \otimes v - u \otimes Cv
            \\
            &= ( 6 \cdot 1_{2 \epsilon_1} + 2 \cdot 1_{\epsilon_1} ) (u \otimes v) - 2(u \otimes v) - 2(u \otimes v)
            \\
            &= ( -4 \cdot 1_0 + 2 \cdot 1_{2 \epsilon_1} - 2 \cdot 1_{\epsilon_1} ) (u \otimes v).
        \end{align*}
        On the other hand, by \cref{Vproject},
        \[
            \bAF
            \left(\,
                \begin{tikzpicture}[centerzero]
                    \draw[vec] (-0.2,-0.35) -- (0.2,0.35);
                    \draw[vec] (0.2,-0.35) -- (-0.2,0.35);
                \end{tikzpicture}
                \, -\,
                \begin{tikzpicture}[centerzero]
                    \draw[vec] (-0.2,-0.35) -- (-0.2,-0.3) arc(180:0:0.2) -- (0.2,-0.35);
                    \draw[vec] (-0.2,0.35) -- (-0.2,0.3) arc(180:360:0.2) -- (0.2,0.35);
                \end{tikzpicture}
            \, \right)
            = \flip - 3 1_0
            = -2 1_0 + 1_{2 \epsilon_1} - 1_{\epsilon_1}.
        \]
        Now suppose $N=2$.  In this case, $C$ acts on $V$ as the identity and on the $W$ of \cref{temple} by $4$.  It follows from \cref{koala,ibis} that $2\Omega$ acts on $V^{\otimes 2}$ as
        \begin{align*}
            2 \Omega(u \otimes v)
            &= C(u \otimes v) - Cu \otimes v - u \otimes Cv
            \\
            &= ( 4 \cdot 1_W ) (u \otimes v) - (u \otimes v) - (u \otimes v)
            \\
            &= ( 2 \cdot 1_W - 2 \cdot 1_{\triv^0} - 2 \cdot 1_{\triv^1} ) (u \otimes v).
        \end{align*}
        On the other hand, by \cref{Vproject},
        \[
            \bAF
            \left(\,
                \begin{tikzpicture}[centerzero]
                    \draw[vec] (-0.2,-0.35) -- (0.2,0.35);
                    \draw[vec] (0.2,-0.35) -- (-0.2,0.35);
                \end{tikzpicture}
                \, -\,
                \begin{tikzpicture}[centerzero]
                    \draw[vec] (-0.2,-0.35) -- (-0.2,-0.3) arc(180:0:0.2) -- (0.2,-0.35);
                    \draw[vec] (-0.2,0.35) -- (-0.2,0.3) arc(180:360:0.2) -- (0.2,0.35);
                \end{tikzpicture}
            \, \right)
            = \flip - 2 1_{\triv^0}
            = 1_W - 1_{\triv^0} - 1_{\triv^1}.
        \]
    }
    Thus, $\bAF$ respects the first relation in \cref{dotcross1}.

    \medskip

    \emph{Second relation in \cref{dotcross1}}.  Composing on the top of the second relation in \cref{dotcross1} with the invertible morphism $\crossmor{spin}{spin}$, then using \cref{swishy} and the first relation in \cref{brauer}, we see that the second relation in \cref{dotcross1} is equivalent to
    \[
        \begin{tikzpicture}[centerzero]
            \draw[spin] (-0.2,-0.4) -- (-0.2,0.4);
            \draw[spin] (0.2,-0.4) -- (0.2,0.4);
            \singdot{-0.2,0};
        \end{tikzpicture}
        \ -\
        \begin{tikzpicture}[centerzero]
            \draw[spin] (0.2,-0.4) to[out=135,in=down] (-0.15,0) to[out=up,in=225] (0.2,0.4);
            \draw[spin] (-0.2,-0.4) to[out=45,in=down] (0.15,0) to[out=up,in=-45] (-0.2,0.4);
            \singdot{0.15,0};
        \end{tikzpicture}
        \ = \frac{1}{8}
        \left(\,
            \begin{tikzpicture}[centerzero]
                \draw[spin] (-0.2,-0.4) -- (-0.2,0.4);
                \draw[spin] (0.2,-0.4) -- (0.2,0.4);
                \draw[vec] (-0.2,-0.2) -- (0.2,0.2);
                \draw[vec] (-0.2,0.2) -- (0.2,-0.2);
            \end{tikzpicture}
            \, -\,
            \begin{tikzpicture}[centerzero]
                \draw[spin] (-0.2,-0.4) -- (-0.2,0.4);
                \draw[spin] (0.2,-0.4) -- (0.2,0.4);
                \draw[vec] (-0.2,-0.2) -- (0.2,-0.2);
                \draw[vec] (-0.2,0.2) -- (0.2,0.2);
            \end{tikzpicture}
        \, \right)
        \overset{\cref{oist}}{\underset{\cref{bump}}{=}}
        \frac{1}{4}
        \left(
            N\
            \begin{tikzpicture}[centerzero]
                \draw[spin] (-0.2,-0.4) -- (-0.2,0.4);
                \draw[spin] (0.2,-0.4) -- (0.2,0.4);
            \end{tikzpicture}
            \, -\,
            \begin{tikzpicture}[centerzero]
                \draw[spin] (-0.2,-0.4) -- (-0.2,0.4);
                \draw[spin] (0.2,-0.4) -- (0.2,0.4);
                \draw[vec] (-0.2,-0.2) -- (0.2,-0.2);
                \draw[vec] (-0.2,0.2) -- (0.2,0.2);
            \end{tikzpicture}
        \, \right)
        .
    \]
    Thus, the fact that $\bAF$ respects the second relation in \cref{dotcross1} follows from \cref{enmore,baction,stand}.

    \medskip

    \emph{Relations \cref{dotcross2}}.  Composing on the top of the first relation in \cref{dotcross2} with the invertible morphism $\crossmor{vec}{spin}$, then using \cref{lobster} and the first relation in \cref{brauer}, we see that the first relation in \cref{dotcross2} is equivalent to
    \begin{equation} \label{rooSV}
        \begin{tikzpicture}[centerzero]
            \draw[spin] (-0.2,-0.4) -- (-0.2,0.4);
            \draw[vec] (0.2,-0.4) -- (0.2,0.4);
            \singdot{-0.2,0};
        \end{tikzpicture}
        \ -\
        \begin{tikzpicture}[centerzero]
            \draw[vec] (0.2,-0.4) to[out=135,in=down] (-0.15,0) to[out=up,in=225] (0.2,0.4);
            \draw[spin] (-0.2,-0.4) to[out=45,in=down] (0.15,0) to[out=up,in=-45] (-0.2,0.4);
            \singdot{0.15,0};
        \end{tikzpicture}
        \ = \
        \begin{tikzpicture}[centerzero]
            \draw[spin] (-0.2,-0.4) -- (-0.2,0.4);
            \draw[vec] (0.2,-0.4) -- (0.2,0.4);
        \end{tikzpicture}
        -
        \begin{tikzpicture}[centerzero]
            \draw[spin] (-0.35,0.35) -- (0,0.15) -- (0,-0.15) -- (-0.35,-0.35);
            \draw[vec] (0,0.15) -- (0.35,0.35);
            \draw[vec] (0,-0.15) -- (0.35,-0.35);
        \end{tikzpicture}
        \ .
    \end{equation}
    By \cref{enmore}, the image under $\bAF$ of the left-hand side of \cref{rooSV} is the natural endomorphism of the functor $S \otimes V \otimes -$ given by $2 \Omega \otimes -$.  When $N \ge 3$, we have, from \cref{SV},
    \[
        S \otimes V
        \cong S \otimes L(\epsilon_1)
        \cong S \oplus W,
    \]
    where
    \[
        W =
        \begin{cases}
            L \left( \tfrac{3}{2} \epsilon_1 + \tfrac{1}{2} \epsilon_2 + \dotsb + \tfrac{1}{2} \epsilon_n \right) & \text{if } N=2n+1, \\
            \Ind \left( L \left( \tfrac{3}{2} \epsilon_1 + \tfrac{1}{2} \epsilon_2 + \dotsb + \tfrac{1}{2} \epsilon_n \right) \right) & \text{if } N=2n.
        \end{cases}
    \]
    Then, as in our verification of the first relation in \cref{dotcross1}, we use \cref{koala,ibis} to compute that $2 \Omega$ acts on $S \otimes V$ as
    \[
        \tfrac{N(N-1)}{8} 1_S + \tfrac{N(N+7)}{8} 1_W - \tfrac{N(N-1)}{8} - (N-1)
        = 1 - N 1_S.
    \]
    By \cref{soju}, this is the also the action on $S \otimes V$ of the image under $\bF$ of the left-hand side of \cref{rooSV}.  The case $N=2$ is similar.
    \details{
        Suppose $N=2$.  Then $S \otimes V \cong L \left( -\frac{3}{2}\epsilon_1 \right) \oplus L \left( -\frac{1}{2}\epsilon_1 \right) \oplus L \left( \frac{1}{2}\epsilon_1 \right) \oplus L \left( \frac{3}{2}\epsilon_1 \right)$ as $\fso(V)$-modules.  Then we use \cref{koala} to compute that $2 \Omega$ acts on $S \otimes V$ as
        \[
            \left( \tfrac{9}{4} \cdot 1_{-3\epsilon_1/2} + \tfrac{9}{4} \cdot 1_{3\epsilon_1/2} + \tfrac{1}{4} \cdot 1_V \right) - \tfrac{1}{4} - 1
            = 1 - 2 \cdot 1_V.
        \]
        By \cref{soju}, this is the also the action on $S \otimes V$ of the image under $\bF$ of the left-hand side of \cref{rooSV}.
    }
    Thus, $\bAF$ respects the first relation in \cref{dotcross2}. The proof that $\bAF$ respects the second relation in \cref{dotcross2} is almost identical.

    \medskip

    \emph{Relations \cref{dotcap}}.  Let $U$ denote either $V$ or $S$.  The image under $\bAF$ of the left-hand side of relations \cref{dotcap} is the natural transformation with components $U \otimes U \otimes M \to U \otimes U \otimes M$ given by
    \begin{align*}
        u \otimes v \otimes m
        &\mapsto \sum_{X \in \bB_{\fso(V)}} \left( 2\Phi_U(Xu \otimes X^\vee v) m + 2\Phi_U(Xu \otimes v) X^\vee m + \Phi_U(X X^\vee u, v) m \right)
        \\
        &= -\sum_{X \in \bB_{\fso(V)}} \left(2\Phi_U(u \otimes Xv) X^\vee m + \Phi_U(u, X X^\vee v) m \right),
    \end{align*}
    where the equality follows from \cref{dark} in the case $U=S$ and from the definition of $\fso(V)$ in the case $U=V$.  Since the last sum above is precisely the image under $\bAF$ of the right-hand side of relations \cref{dotcap}, we see that $\bAF$ preserves these relations.

    \medskip

    \emph{Relation \cref{dotvertex}}.  We will show that, for any homomorphism $f \colon U_1 \otimes U_2 \to W$ of $\Group(V)$-modules, we have
    \begin{equation} \label{racoon}
        \dotop \circ (f \otimes 1)
        = (f \otimes 1) \circ \big( (1 \otimes \Delta)(\dotop) + (1 \otimes \dotop) \big)
        \colon U_1 \otimes U_2 \otimes M \to W \otimes M,
    \end{equation}
    for any $M \in \Group(V)\md$.  Then the fact that $\bAF$ respects \cref{dotvertex} follows from taking $f = \tau$, given by \cref{triaction}.  To prove \cref{racoon}, we compute
    \begin{align*}
        \dotop \circ &(f \otimes 1)(u_1 \otimes u_2 \otimes m)
        = \sum_{X \in \bB_{\fso(V)}} \left( 2Xf(u_1 \otimes u_2) \otimes X^\vee m + X X^\vee f(u_1 \otimes u_2) \otimes m \right)
        \\
        &\begin{multlined}
            = (f \otimes 1) \sum_{X \in \bB_{\fso(V)}} \left( 2Xu_1 \otimes u_2 \otimes X^\vee m + 2u_1 \otimes Xu_2 \otimes X^\vee m + XX^\vee u_1 \otimes u_2 \otimes m \right. \\
            \left. + 2Xu_1 \otimes X^\vee u_2 \otimes m + u_1 \otimes XX^\vee u_2 \otimes m \right)
        \end{multlined}
        \\
        &= (f \otimes 1) \circ \big( (1 \otimes \Delta)(\dotop) + (1 \otimes \dotop) \big) (u_1 \otimes u_2 \otimes m),
    \end{align*}
    proving \cref{racoon}.
\end{proof}

\begin{rem} \label{lowaffine}
    Although we assumed above that $N \ge 2$, one can define the affine incarnation functor for $N=0$ and $N=1$.  In these cases, the functor is defined as in \cref{affinc}, except that both dots are sent to the zero natural transformation.
\end{rem}

\begin{rem} \label{meow}
    Replacing $\bF$ by the functor $\bF' \colon \SB(V) \to \fso(V)\md$ of \cref{panda}, we can define an affine version $\bAF' \colon \ASB(V) \to \cEnd_\kk(\fso(V)\Md)$ of that functor, defined in the same way as $\bAF$.  Here we choose to work with the category $\fso(V)\Md$ of \emph{all} $\fso(V)$-modules (as opposed to just finite-dimensional ones) for reasons that will be become apparent in \cref{sec:center}.
\end{rem}

\section{Central elements\label{sec:center}}

We assume throughout this section that $\kk = \C$.  Let $\fg=\fso(V)$ and let $Z(\fg)$ be the centre of its universal enveloping algebra $U(\fg)$.  This centre is identified with the endomorphism algebra of the identity functor $\Id_{\fg\Md}$.  Precisely, evaluation on the identity element of the regular representation of $U(\fg)$ defines a canonical algebra isomorphism $\End(\Id_{\fg\Md}) \xrightarrow{\cong} Z(\fg)$.  (It is here that we need to consider all $\fso(V)$-modules, not just finite-dimensional ones; see \cref{meow}.)  It follows that the affine incarnation functor
\[
    \bAF' \colon \ASB(V) \to \cEnd_\C(\fg\Md)
\]
of \cref{meow} induces a homomorphism
\[
    \chi \colon \End_{\ASB(V)}(\one) \to Z(\fg).
\]
The goal of this section is to describe the image of $\chi$.  We will prove the following result.

\begin{theo}\label{centresurjection}
The image of $\chi$ is equal to $Z(\fg)^{\Group(V)}$.
\end{theo}

We first entertain a discussion of the structure of $Z(\fg)$, which is given by the Harish-Chandra isomorphism.  Recall that, for $\la \in X^\ast(H)^+$, $L(\la)$ is the simple highest-weight $\fg$-module with highest weight $\la$.
By Schur's Lemma, any $z \in Z(\fg)$ acts on $L(\la)$ by a scalar. The Harish-Chandra isomorphism is an isomorphism of algebras
\[
    \Gamma \colon Z(\fg)\to \C[\fh^\ast]^W, \qquad z\mapsto (f_z:\fh^\ast\to \C),
\]
 where $W$ is the Weyl group, uniquely characterised by the identity
\[
    z v = f_z(\la+\rho) v,
\]
for all $z\in Z(\fg)$ and $v\in L(\la)$.  The Harish-Chandra isomorphism $\Gamma$ is equivariant with respect to the natural actions of $\Group(V)$ on both sides by conjugation.
\details{
    The action of $\Group(V)$ on the domain comes from the adjoint action of $\Group(V)$ on $\fg$.  The action on the codomain is given as follows.  Let $\fb_1$ and $\fb_2$ be two Borel subalgebras of $\fg$.  Since $\Spin(V)$ is connected, there exists $g \in \Spin(V)$ such that $\fb_2 = g \fb_1 g^{-1}$.  Then conjugation by $g$ defines a linear map $\fb_1/[\fb_1,\fb_1] \to \fb_2/[\fb_2,\fb_2]$.  Since the normaliser of a Borel subalgebra is $\Spin(V)$ is a Borel subgroup, the map $\fb_1/[\fb_1,\fb_1] \to \fb_2/[\fb_2,\fb_2]$ is independent of the choice of $g$.  So, we can define the abstract Cartan to be $\fh = \fb/[\fb,\fb]$ for any choice of Borel subalgebra $\fb$, and then any other Borel subalgebra canonically maps to it.  Now, we also have an adjoint action of $\Group(V)$ on $\fg$, and this action maps Borel subalgebras to Borel subalgebras.  Thus, we have an induced action of $\Group(V)$ on $\fh$, with $\Spin(V)$ acting trivially.  This induces an action of $\Group(V)$ on $\fh^*$, and hence on $\C[\fh^*]$ by precomposition.  Using the notation of \cref{twist}, we have
    \[
        g z g^{-1} v
        = z \cdot v
        = f_z(g(\lambda + \rho)) v, \qquad
        g \in \Group(V),\ z \in Z(\fg),\ v \in L(\lambda),
    \]
    and so $\Gamma$ intertwines the $\Group(V)$ actions.
}
Since $\Spin(V)$ acts trivially on $Z(\fg)$, we have, by \cref{hydro},
\begin{equation} \label{mango}
    Z(\fg)^{\Group(V)}
    =
    \begin{cases}
        Z(\fg) & \text{if } N \in 2\N + 1, \\
        Z(\fg)^P & \text{if } N \in 2\N,
    \end{cases}
\end{equation}
where $P$ is as in \cref{Pdef}.
\details{
    Connected Lie groups acts trivially on the centre of their Lie algebras since they are generated by exponentials of elements of the Lie algebra, which commute with $Z(\fg)$.
}

To simplify notation, we define $x_i = A_{ii}$ for $1 \le i \le n$, where $A_{ii}$ is defined as in \cref{greenD,greenB}.  If $N$ is odd, then $W=C_2^n \rtimes \fS_n$ and
\[
    Z(\fg)^{\Group(V)}
    \overset{\cref{mango}}{=} Z(\fg)
    \cong \C[\fh^\ast]^W
    \cong \C[x_1,x_2,\dotsc,x_n]^{C_2^n \rtimes \fS_n}=\C[x_1^2,x_2^2,\dotsc,x_n^2]^{\fS_n},
\]
the ring of symmetric polynomials in $x_1^2,x_2^2,\dotsc,x_n^2$. (Here $C_2$ is the cyclic group on two elements.)

If $N$ is even, then the action of $\Group(V)$ on $Z(\fg)$ is no longer trivial; see \cref{pactweights}.
\details{
    For $z \in Z(\fg)$ and $v \in L(\lambda)$, we have
    \[
        f_{P \cdot z}(\lambda + \rho) = P z P^{-1} v = z \cdot v = f_z(\tilde{\lambda}+\rho),
    \]
    where $z \cdot v$ denotes the twisted action as in \cref{twist}, and the final equality uses \cref{hydro}.  When $\lambda_n \ne 0$, the weights $\lambda+\rho$ and $\tilde{\lambda} + \rho$ are not in the same orbit of the Weyl group.  Thus, using the Harish-Chandra isomorphism there exists a $z \in Z(\fg)$ such that $f_z(\tilde{\lambda}+\rho) \ne f_z(\lambda+\rho)$.  Hence, $P \cdot z \ne z$, and so the action of $\Group(V)$ on $Z(\fg)$ is not trivial.
}
Here the action of the component group of $\Group(V)$ precisely compensates for the difference between the type $B$ and type $D$ Weyl groups.  More precisely, we have
\[
    Z(\fg)^{\Group(V)}
    \cong (\C[\h^\ast]^W)^P
    \cong \C[x_1,x_2,\dotsc,x_n]^{C_2^n \rtimes \fS_n}
    = \C[x_1^2,x_2^2,\dotsc,x_n^2]^{\fS_n},
\]
where the first isomorphism arises from \cref{mango} and the Harish-Chandra isomorphism, while the second isomorphism follows from the fact that $W$ and $\pi_0(\Group(V))$ generate the action of $C_2^n \rtimes \fS_n$, using \cref{pactweights}.  So, in either case, we have the isomorphism
\begin{equation} \label{cupcake}
    Z(\fg)^{\Group(V)} \cong \C[x_1^2,x_2^2,\dotsc,x_n^2]^{\fS_n}.
\end{equation}

Define
\begin{equation}
    z_r := \chi \left( \multbubble{spin}{r}\!\! \right),\qquad
    r \in \N.
\end{equation}

\begin{prop} \label{fizz}
    For each $r \in \N$,
    \[
        f_{z_r} \in (-1)^{\binom{n}{2}+nN} \sum_{\varsigma_1,\dotsc,\varsigma_n\in \{\pm 1\}} \left(\sum_{i=1}^n \varsigma_i x_i \right)^r + \C[\fh^\ast]_{<r},
    \]
    where $\C[\fh^\ast]_{<r}$ denotes the space of polynomial functions on $\fh^\ast$ of degree strictly less than $r$.
\end{prop}

\begin{proof}
    Let $v$ be a highest-weight vector of $L(\lambda-\rho)$, and let $\prescript{\vee}{}{x_I}$, $I \subseteq [n]$, be the right dual basis to $x_I$, $I \subseteq [n]$, defined by $\Phi_S(x_I, \prescript{\vee}{}{x_J}) = \delta_{IJ}$.  Note that
    \begin{equation} \label{lizard}
        \Phi_S(\prescript{\vee}{}{x_I}, x_J)
        \overset{\cref{formsym}}{=} (-1)^{\binom{n}{2} + nN} \Phi_S(x_J, \prescript{\vee}{}{x_I})
        = \delta_{IJ} (-1)^{\binom{n}{2}+nN}.
    \end{equation}

    Unravelling the definition of $\bAF' \left( \multbubble{spin}{r} \right)$, we get
    \begin{equation} \label{cream}
        f_{z_r}(\lambda) v
        = z_r v
        = (\Phi_S\otimes \id) (1\otimes \dotop)^r \sum_{I \subseteq [n]} \prescript{\vee}{}{x_I} \otimes x_I \otimes v.
    \end{equation}
    Note that
    \[
        \dotop=2\sum_{i=1}^n  A_{ii}\otimes A_{ii} + \sum_{\alpha \in\Phi}  X_\alpha \otimes Y_\alpha+C\otimes 1,
    \]
    for some $X_\alpha\in \mathfrak{g}_\alpha$, $Y_\alpha\in \mathfrak{g}_{-\alpha}$, where $\Phi$ is the set of roots of $\fg$

    Write $\dotop^r$ in the form $\dotop^r=\sum_j A_j\otimes B_j$, where each term $B_j$ is a monomial in a Poincaré--Birkhoff--Witt (PBW) basis of $U(\fg)\cong U(\n^-)\otimes U(\fh)\otimes U(\n^+)$.  The terms with degree equal to $r$ that give a nonzero contribution to \cref{cream} are exactly the monomials involving only elements of $\fh$, and for these monomials, we compute
    \begin{multline*}
        (\Phi_S\otimes \id) \left( 2\sum_{i=1}^n 1 \otimes A_{ii} \otimes A_{ii} \right)^r \sum_{I\subseteq [n]} \prescript{\vee}{}{x_I} \otimes x_I \otimes v
        \\
        \overset{\cref{lizard}}{=} (-1)^{\binom{n}{2}+nN} \sum_{\varsigma_1,\dotsc,\varsigma_n\in \{\pm \frac{1}{2}\}} \left(2\sum_{i=1}^n \varsigma_i (\la_i - \rho_i) \right)^r v,
    \end{multline*}
    which is equal to $(-1)^{\binom{n}{2}+nN} \sum_{\varsigma_1,\dotsc,\varsigma_n\in \{\pm 1\}} \left(\sum_{i=1}^n \varsigma_i \la_i \right)^r$ modulo terms of degree strictly less than $r$ in the $\lambda_i$.
    \details{
        When you compute $\dotop^r$, you'll get terms like $F_i E_i$.  When you write these in the PBW basis, you pick up an elements in $\fh$, e.g.\ $F_i E_i = E_i F_i - H_i$.  But the $H_i$ appearing here produces a term lower in the filtration.  This explains why the $U(\fh) \otimes U(\fh)$ part of the degree $r$ part of $\dotop$ is just $(2\sum_{i=1}^r A_{ii} \otimes A_{ii})^r$.
    }
    The remaining terms that give a nonzero contribution all have degree less than $r$, and so lie in $\C[\fh^\ast]_{<r}$.
\end{proof}

We pause to introduce some symmetric functions notation. Let $\La$ denote the ring of symmetric functions with coefficients in $\Q$. We use $p_r$
and $h_r$ to denote the power sum and complete symmetric functions respectively, and $\langle\cdot,\cdot\rangle$ to denote the Hall inner product.  Let $m_\pi$ denote the monomial symmetric function associated to a partition $\pi$. Given a partition $\pi=1^{m_1} 2^{m_2} \cdots$, we define $\delta(\pi)=(m_1,m_2,\dotsc)$. This is a composition of the length, $\ell(\pi)$, of $\pi$.

For $r \in \N$, define the symmetric polynomial
\begin{equation}
    W_r(x_1,x_2,\dotsc,x_n)
    = \frac{1}{2^n}
    \sum_{\varsigma_1,\dotsc,\varsigma_n\in \{\pm 1\}} \left( \sum_{i=1}^n \varsigma_i \sqrt{x_i} \right)^{2r}.
\end{equation}
\details{
    It is clear that $W_r$ is invariant under permutation of the $x_i$.  Since it is invariant under the transformation $\sqrt{x_i} \mapsto -\sqrt{x_i}$, it also lies in $\kk[x_1,\dotsc,x_n]$.  Hence $W_r \in \kk[x_1,\dotsc,x_n]^{\fS_n}$.
}
Taking the inverse limit over $n$, these define a symmetric function $W_r \in \Lambda$.  In terms of the monomial symmetric functions, we have the expansion
\begin{equation}\label{wmonomial}
    W_r= \sum_{\pi\vdash r} \binom{2r}{2\pi} m_{\pi},
\end{equation}
where $\binom{2r}{2\pi}$ is a multinomial coefficient and $2\pi$ denotes the partition obtained from $\pi$ by multiplying all parts by $2$.

\begin{prop}\label{wnpnpairing}
    Let $B_{2r}$ denote the $(2r)$-th Bernoulli number.  Then
    \[
        \langle W_r, p_r\rangle = -2^{2r-1}(2^{2r}-1)B_{2r}
        \qquad \text{for all } r \in \N.
    \]
\end{prop}

\begin{proof}
    Begin with the generating function identity
    \[
        \sum_{k=1}^\infty \frac{p_k}{k}t^k
        = \log \left( \sum_{j=0}^\infty h_j t^j \right)
        = \sum_{m=0}^\infty \frac{(-1)^m}{m}\left( \sum_{j=1}^\infty h_j t^j \right)^m
    \]
    and expand it to obtain
    \[
        \frac{-p_r}{r} = \sum_{\pi\vdash r} \binom{\ell(\pi)}{\delta(\pi)} \frac{(-1)^{\ell(\pi)}}{\ell(\pi)} h_\pi,
    \]
    where, again, $\binom{\ell(\pi)}{\delta(\pi)}$ is a multinomial coefficient.  Since the complete symmetric functions are dual to the monomial symmetric functions, this implies that
    \[
        \frac{-1}{r}\langle p_r,m_\pi\rangle =  {\ell(\pi) \choose \delta(\pi)} \frac{(-1)^{\ell(\pi)}}{\ell(\pi)}.
    \]
    The above equation, together with \cref{wmonomial}, implies that
    \begin{equation}\label{wnpnsum}
        \frac{-1}{r}\langle W_r,p_r\rangle = \sum_{\pi\vdash r} \binom{2r}{2\pi}{\ell(\pi) \choose \delta(\pi)} \frac{(-1)^{\ell(\pi)}}{\ell(\pi)}.
    \end{equation}
    We now compute
    \begin{multline*}
        \log(\cosh(x))
        = \log\left(1+\sum_{n=1}^\infty \frac{x^{2n}}{(2n)!}\right)
        = \sum_{m=1}^\infty \frac{(-1)^m}{m}\left( \sum_{n=1}^\infty \frac{x^{2n}}{(2n)!} \right)^m \\
        = \sum_{m=1}^\infty \frac{(-1)^m}{m} \sum_{n_1,\dotsc,n_m=1}^\infty \frac{x^{2(n_1+\dotsb+n_m)}}{(2n_1)!(2n_2)!\dotsm(2n_m)!}.
    \end{multline*}
    Collect all terms with the same multiset $\{n_1,n_2,\ldots,n_m\}$ to make the inner sum into a sum over all partitions of length $m$ and we get
    \[
        \log(\cosh(x))
        = \sum_{m=1}^\infty \frac{(-1)^m}{m} \sum_{\ell(\pi)=m} \binom{m}{\delta(\pi)} \frac{x^{2|\pi|}}{(2\pi_1)!\cdots (2\pi_m)!}
        = \sum_{r=1}^\infty \sum_{\pi \vdash r} \binom{2r}{2\pi} \binom{\ell(\pi)}{\delta(\pi)} \frac{(-1)^{\ell(\pi)}}{\ell(\pi)} \frac{x^{2r}}{(2r)!}.
    \]
    Comparing this with \cref{wnpnsum}, we obtain
    \[
        \sum_{r=1}^\infty \frac{-1}{r}\langle W_r,p_r\rangle \frac{x^{2r}}{(2r)!}
        = \log(\cosh(x)).
    \]
    Differentiating with respect to $x$ gives
    \[
        \sum_{r=1}^\infty \frac{-2}{(2r)!} \langle W_r, p_r \rangle x^{2r-1}
        = \tanh(x)
        = \sum_{r=1}^\infty \frac{ 2^{2r}(2^{2r}-1)B_{2r}}{(2r)!} x^{2r-1}.
    \]
    Comparing the coefficients of $x^{2r-1}$ gives the desired result.
\end{proof}

We use the following criterion for determining generators for $\La$.

\begin{prop}\label{gencriterion}
    Let $q_1,q_2,\dotsc$ be elements of $\La$ with $q_i$ of degree $i$, and such that $\langle q_i, p_i \rangle \neq 0$ for all $i$.  Then the $q_i$ are algebraically independent and generate $\Lambda$.
\end{prop}

\begin{proof}
    Let $\La' = \Q[q_1',q_2',\dotsc]$ be the polynomial algebra on indeterminates $q_1',q_2',\dotsc$ and consider the algebra homomorphism
    \[
        \alpha \colon \La' \to \La,\qquad
        q_i' \mapsto q_i,\quad i \ge 1.
    \]
    Write $\La_r$ for the $r$-th graded piece of $\La$ and let $X_r$ be the subspace of $\La_r$ spanned by all products of terms of lower degrees. Define $\La_r'$ and $X_r'$ similarly. It suffices to show that the induced map $\alpha_r \colon \La_r' \to \La_r$ is an isomorphism for all $r \in \N$.  We prove this by induction.  The base case $r=0$ is trivial.

    Now suppose $r \ge 1$.  Since $\La \cong \Q[h_1,h_2,\dotsc]$, we know that $X_r$ is of codimension 1 in $\La_r$.  If $a,b \in \La$ are of positive degree, then $\langle ab,p_r\rangle = \langle a\otimes b, p_r\otimes 1+1\otimes p_r\rangle = 0$. Therefore $\langle X_r, p_r \rangle = 0$. In particular $q_r\notin X_r$ and, since $X_r$ is of codimension 1 in $\La_r$, this implies that $\La_r$ is spanned by $X_r$ and $q_r$.  Thus, $\alpha_r$ is surjective.  Since $X_r$ and $X_r'$ have the same dimension, it follows that $\alpha_r$ is an isomorphism, as desired.
\end{proof}

\begin{cor} \label{pizza}
    The symmetric functions $W_r$, $r \ge 1$, are algebraically independent and generate $\Lambda$.
\end{cor}

\begin{proof}
    This follows from \cref{wnpnpairing,gencriterion} and the fact that the even Bernoulli numbers are nonzero.
\end{proof}

We can now prove \cref{centresurjection}.

\begin{proof}[Proof of \cref{centresurjection}]
    We first show that the image of $\chi$ lies in $Z(\fg)^{\Group(V)}$.  By \cref{mango}, it suffices to consider the case where $N$ is even.  Let $a \in \End_{\ASB(V)}(\one)$.  We must show that $(\Gamma \circ \chi)(a) \in (\C[\fh^\ast]^W)^P$.  By \cref{hydro}, it suffices to show that
    \begin{equation} \label{bangalore}
        (\Gamma \circ \chi)(a)(\lambda) = (\Gamma \circ \chi)(a)(\tilde{\lambda})
    \end{equation}
    for all $\lambda \in \fh^\ast$, where $\tilde{\lambda}$ is defined as in \cref{reflect}.  In fact, since the set of dominant integral $\lambda$ for which $\tilde{\lambda} \ne \lambda$ is Zariski dense in $\fh$*, it suffices to prove that \cref{bangalore} holds for all such $\lambda$.

    Suppose that $\la$ is a dominant integral weight satisfying $\tilde{\lambda} \neq \lambda$.  Then $\Ind(L(\la))$ is a simple $\Pin(V)$-module by \cref{sponge}, and so $\bAF(a)$ acts on it by a scalar.  The action of $\bAF(a)$ on $\Ind(L(\la))$ is the same as the action of $\bAF'(a)$ on $\Res \circ \Ind(L(\la))\cong L(\la)\oplus L(\tilde{\lambda})$.  Therefore, \cref{bangalore} holds, as desired.

    It remains to prove that $\chi$ surjects onto $Z(\fg)^{\Group(V)}$.  But this follows from \cref{fizz}, \cref{pizza}, and the isomorphism \cref{cupcake}.
\end{proof}

\begin{cor}
    The elements
    \[
        \multbubble{spin}{r} \!\!\in \End_{\ASB(V)}(\one),\qquad r \ge 1,
    \]
    are algebraically independent.
\end{cor}

\begin{proof}
    This follows immediately from the fact that their images under $\chi$ are algebraically independent, by \cref{fizz,pizza}.
\end{proof}

Given their role above, it would be interesting to further study the symmetric functions $W_r$.  Recall that a symmetric function is \emph{Schur-positive} if, when written as a linear combination of Schur functions, all coefficients are nonnegative.  Computer computations suggest the following conjecture.

\begin{conj}
    The symmetric functions $W_r$ are Schur-positive.
\end{conj}


\bibliographystyle{alphaurl}
\bibliography{SpinBrauer}

\end{document}